\documentclass[letterpaper]{article}
\usepackage[preprint]{aaai2027}
\usepackage[hyphens]{url}
\usepackage{graphicx}
\usepackage{natbib}
\usepackage{caption}
\usepackage{amsmath,amssymb,amsthm}
\usepackage{array}
\usepackage{bm}
\usepackage{enumerate}
\usepackage{mathrsfs}
\usepackage{algorithm}
\usepackage{algpseudocode}

\newcommand{\R}{\mathbb{R}}

\newcommand{\M}{\mathcal{M}}

\newcommand{\id}{\mathrm{id}}

\newcommand{\rank}{\mathrm{rank}}

\newcommand{\pf}{\mathrm{pf}}
\newcommand{\qf}{\mathrm{qf}}

\newcommand{\St}{\mathrm{St}}

\theoremstyle{plain}
\newtheorem{theorem}{Theorem}[section]
\newtheorem{lemma}[theorem]{Lemma}
\newtheorem{proposition}[theorem]{Proposition}

\newtheorem{problem}[theorem]{Problem}
\newtheorem{ur}[theorem]{Stepsize Rule}

\theoremstyle{definition}
\newtheorem{definition}[theorem]{Definition}

\theoremstyle{remark}
\newtheorem{remark}[theorem]{Remark}

\numberwithin{theorem}{section}
\numberwithin{equation}{section}

\makeatletter
\let\aaai@original@maketitle\maketitle
\def\aaai@singlecolumn[#1]{#1}
\renewcommand{\maketitle}{%
  \begingroup
  \let\twocolumn\aaai@singlecolumn
  \aaai@original@maketitle
  \endgroup
}
\makeatother

\title{Retraction-Based Gradient Projection Algorithms on Manifolds}
\author{Conglong Xu, Hao Wu\thanks{Phillips Hall, Room 739, 801 22nd Street, N.W., Washington, DC 20052, USA.
Emails: \texttt{xuconglong@gwmail.gwu.edu}, \texttt{haowu@gwu.edu}.
Correspondence to Hao Wu.}}
\affiliations{Department of Mathematics, The George Washington University}

\begin{document}

\onecolumn
\maketitle
\pagestyle{plain}
\thispagestyle{plain}

\begin{abstract}
We introduce a framework for retraction-based convex optimization on Riemannian manifolds, which includes a notion of retraction-specific convex sets and
retraction-based gradient projection algorithms. The standard theory of gradient projection algorithms generalizes easily to this framework.
Within this framework, we establish convergence results for retraction-based gradient projection algorithms with various stepsize rules.
As an application, we use our framework to study the weighted low-rank approximation.
We also provide numerical validation of our convergence results on the image completion task.
\end{abstract}

\section{Introduction}

In this paper, we extend the theory of gradient projection algorithms from Euclidean spaces to Riemannian manifolds.
More precisely, we introduce a notion of retraction-specific convex sets on manifolds in Section~\ref{sec-R-convexity}
and establish retraction-based gradient projection algorithms on manifolds and their convergence in Section~\ref{sec-rgpa}.
This provides a more general framework for convex analysis on manifolds than the existing theory of geodesic convexity.
In Section \ref{sec-wlra-constraint}, we apply our algorithms to the weighted low-rank approximation
and numerically validate their convergence on the image completion task.


\textbf{Related Work and Our Contributions.} The existing theory of geodesic convexity in differential geometry has motivated recent efforts
to extend the theory of convex optimization from Euclidean spaces to Riemannian manifolds
\cite{becigneul-ganea-2019-riemannian-adaptive,nguyen-et-al-2019-optimistic-likelihoods,alimisis-et-al-2021-momentum-riemannian,martinez-rubio-2022-global-riemannian-acceleration,weber-sra-2023-riemannian-frank-wolfe,weber-sra-2023-global-cccp,criscitiello-boumal-2023-curvature-complexity,hu-wang-abernethy-2023-riemannian-projection-free-online,fan-yang-vemuri-2023-horospherical-classification,han-et-al-2024-dp-riemannian,yao-chen-yang-2024-wasserstein-proximal-coordinate,wang-et-al-2024-augustin-information,roux-martinez-rubio-pokutta-2025-riemannian-optimism,cheng-weber-2026-structured-spd,carrillo-et-al-2026-fisher-rao,hirai-nieuwboer-walter-2026-interior-point-manifolds,vishnoi-2018-geodesic-convex-optimization}.
Our notion of retraction-specific convex sets explicitly encodes convexity on manifolds via their tangent bundles and generalizes the existing notion of geodesically convex sets.

The theory of gradient projection methods on Euclidean spaces is a classical topic in optimization.
Rosen's gradient projection methods for linearly and
nonlinearly constrained nonlinear programming \cite{rosen-1960-gradient-projection-linear,rosen-1961-gradient-projection-nonlinear}, along with
the constrained minimization work of Levitin and Polyak \cite{levitin-polyak-1966-constrained-minimization},
established the early theoretical framework for the projection-based feasible direction methods on Euclidean spaces.
The studies on Armijo's sufficient-decrease rule \cite{armijo-1966-minimization} and
the feasible-direction convergence \cite{topkis-veinott-1967-feasible-direction-convergence}
provide the backtracking line-search framework adapted in this paper.
The expository work in \cite{Bertsekas-nonlinear-programming} provides a detailed survey on the theoretical framework of the gradient projection methods on Euclidean spaces,
which inspired our theoretical framework.

The studies of Euclidean gradient projection methods led to numerous applications on, for example,
gradient projection methods for linearly constrained problems \cite{calamai-more-1987-projected-gradient},
active-set and bound-constrained quadratic programming methods \cite{more-toraldo-1991-bound-qp},
spectral projected-gradient methods \cite{birgin-martinez-raydan-2000-spg,birgin-martinez-raydan-2014-spg-review},
and projected Barzilai-Borwein methods \cite{dai-fletcher-2005-projected-bb}, among other problems.
But attempts to extend these methods to Riemannian manifolds are scarce.
Some recent works \cite{mahler-et-al-2024-on-manifold-pgd, xiong-et-al-2026-manifold-aware-gradient-projection} mentioned ``gradient projection onto manifolds''.
However, the algorithms they established are applications or extensions of the orthogonal projection of Euclidean gradients onto tangent spaces in some machine learning tasks, which are unrelated to convex optimization on manifolds.

The works in \cite{AMS, absil-malick-2012} systematically introduce the projection-like retractions on Riemannian manifolds and their use in matrix-manifold optimization
as computationally efficient approximations of exponential maps.
The work in \cite{Boumal-Book} gives a comprehensive treatment of retraction-based Riemannian optimization methods, including their applications in geodesically convex settings.
These expository works on Riemannian optimization motivated the retraction-based algorithm designed in this paper.
Moreover, the studies of the convergence and complexity of retraction-based Riemannian first-order methods
in \cite{bento-ferreira-melo-2017-riemannian-complexity} partly inspired our convergence results for the Riemannian gradient projection methods.

As an example, we apply our algorithms to the weighted low-rank approximation.
This problem is NP-hard \cite{gillis-glineur-2011-wlra-np-hard}.
The work in \cite{markovsky-2012-low-rank-approximation} supplies a broader survey on the background of the problem.
The work in \cite{manton-mahony-hua-2003-wlra-geometry} studies the geometry of the problem from a matrix-manifold viewpoint.
The work in \cite{werner-jansson-2006-reduced-rank-wlra} is an early attempt to address the problem from a Riemannian optimization perspective.
However, our approach to the problem is more inspired by Stiefel manifold optimization methods for orthogonality-constrained problems
(for example, see \cite{edelman-arias-smith-1998-orthogonality, jiang-dai-2015-stiefel-update-schemes}).
The weighted low-rank approximation is a classical modeling setting for multiple machine learning tasks
including matrix completion, robust principal component analysis, recommender-system matrix factorization,
and collaborative filtering \cite{candes-recht-2009-matrix-completion,koren-bell-volinsky-2009-recommender-mf}.
We numerically validate our algorithm on the image completion task,
which is a variation of the matrix completion task in the context of computer vision and image processing.

\section{Retraction-Specific Convex Sets on Manifolds}\label{sec-R-convexity}
We first recall the definition of retractions on differentiable manifolds and some basic properties of projections onto closed convex sets in inner product spaces.

\begin{definition}\cite[Definition 4.1.1]{AMS}\label{def-retraction}
Let $M$ be a differentiable manifold. A retraction on $M$ is a $C^1$ map $R:TM \rightarrow M$ such that, for every $x \in M$, the restriction $R_x=R|_{T_x M}$ satisfies
\begin{enumerate}
	\item $R_x(\mathbf{0}_x)=x$, where $\mathbf{0}_x$ is the zero vector in $T_x M$,
	\item $dR_x|_{\mathbf{0}_x}=\id_{T_x M}$ under the standard identification $T_{\mathbf{0}_x}T_x M \cong T_x M$, where $dR_x$ is the differential of $R_x$ and $\id_{T_x M}$ is the identity map of $T_x M$.
\end{enumerate}
\end{definition}

\begin{proposition}\cite[Proposition 1.1.4]{Bertsekas-nonlinear-programming}\cite[Proposition 2.1.3]{Bertsekas-nonlinear-programming-second-edition}
\label{prop-convex-projection}
Let $(V, \left\langle\ast,\ast \right\rangle)$ be a finite dimensional inner product space over $\R$, and $C$ a closed convex subset of $V$.
	\begin{enumerate}
		\item For every $\mathbf{v} \in V$, there is a unique vector $\Pi_C(\mathbf{v}) \in C$ that minimizes $\|\mathbf{v} -\mathbf{w}\|$ for all $\mathbf{w} \in C$,
		where $\|\ast\|$ is the norm induced by the inner product $\left\langle\ast,\ast \right\rangle$. $\Pi_C(\mathbf{v})$
		is called the projection of $\mathbf{v}$ on $C$.
		\item For a given $\mathbf{v} \in V$, a vector $\mathbf{w}^\ast \in C$ is equal to the projection $\Pi_C(\mathbf{v})$ if and only if
		\begin{equation}\label{eq-convex-proj-inequality-usual}
			\left\langle \mathbf{v} - \mathbf{w}^\ast , \mathbf{w} - \mathbf{w}^\ast \right\rangle \leq 0 \text{ for all } \mathbf{w} \in C.
		\end{equation}
		\item The mapping $\Pi_C: V \rightarrow C$ is continuous and non-expansive, that is,
		\begin{equation}\label{eq-convex-proj-non-expansive}
			\|\Pi_C(\mathbf{v}_1) - \Pi_C(\mathbf{v}_2)\| \leq \|\mathbf{v}_1 - \mathbf{v}_2\| \text{ for all } \mathbf{v}_1, \mathbf{v}_2 \in V.
		\end{equation}
	\end{enumerate}
\end{proposition}

Now we can state our definition of retraction-specific convex subsets on manifolds.

\begin{definition}\label{def-R-convex}
Let $M$ be a Riemannian manifold and $R$ a retraction on $M$. For $x\in M$, denote by $\left\langle \ast,\ast\right\rangle_x$ the Riemannian inner product on $T_x M$. An $R$-convex subset of $M$ is a closed subset $C$ of $M$ equipped with an $R$-convexity bundle, that is, a collection of subsets $\{C_x \subset T_x M ~\big{|}~x \in C\}$ satisfying:
\begin{enumerate}
	\item (\textbf{Fiber-wise convexity}) $C_x$ is closed and convex in $T_x M$ for every $x \in C$.
	\item (\textbf{Local surjectivity}) For every $x \in C$,
	\begin{enumerate}[(a)]
		\item $R_x(C_x) \subset C$,
		\item $\mathbf{0}_x \in C_x$,
		\item there is an open neighborhood\footnote{Here, $V_x$ is an open set in $C_x$, not necessarily open in $T_x M$. Also, $R_x (V_x)$ is an open set in $C$, not necessarily open in $M$.} $V_x$ of $\mathbf{0}_x$ in $C_x$ such that $R_x(V_x)$ is an open neighborhood of $x$ in $C$ and $R_x|_{V_x}: V_x \rightarrow R_x (V_x)$ is a homeomorphism.
	\end{enumerate}
	\item (\textbf{Projective continuity}) For every $x \in C$, denote by $\Pi_{C_x}: T_x M \rightarrow C_x$ the projection map given by Proposition \ref{prop-convex-projection}(1) for the inner product space $(T_x M, \left\langle \ast,\ast\right\rangle_x)$ and the convex set $C_x$. If a sequence of points $\{x_k\}_{k=1}^\infty \subset C$ and a sequence of tangent vectors $\{\mathbf{v}_k \in T_{x_k} M\}_{k=1}^\infty$ satisfy $\lim_{k \rightarrow \infty} x_k = \tilde{x} \in C$ and $\lim_{k \rightarrow \infty} \mathbf{v}_k = \tilde{\mathbf{v}} \in T_{\tilde{x}} M$, then $\lim_{k \rightarrow \infty} \Pi_{C_{x_k}}(\mathbf{v}_k) = \Pi_{C_{\tilde{x}}}(\tilde{\mathbf{v}}) \in C_{\tilde{x}}$.
\end{enumerate}
\end{definition}

\begin{remark}\label{rk-convexity-def}
\begin{enumerate}
	\item In this paper, all $R$-convex sets are closed in $M$. Generally, the $R$-convexity bundle over a given $R$-convex set is not unique.
	However, in the rest of the paper, when we discuss an $R$-convex set, we always fix an $R$-convexity bundle over it.
	\item Every geodesically totally convex or geodesically strongly convex set is $R$-convex when the retraction $R$ is chosen to be the exponential map.
	\footnote{For definitions of geodesically totally convex and geodesically strongly convex sets,
	see Definitions 11.16 and 11.17 in \cite{Boumal-Book}.} However, an $R$-convex set is not necessarily geodesically totally convex or geodesically strongly convex.
	\item $R$-convexity is a local property. By the local surjectivity, every $R$-convex subset of $M$ is locally path connected.
	But an $R$-convex subset of $M$ can be disconnected. For example, the union of two disjoint non-empty closed convex subsets of $\R^p$ is disconnected. But one can check that this union is $R^{\R^p}$-convex, where $R^{\R^p}$ is the standard retraction on $\R^p$ given in Equation \eqref{eq-def-retraction-euclidean} below.
\end{enumerate}

\end{remark}


\section{Retraction-based Gradient Projection Algorithms on Manifolds}\label{sec-rgpa}

In this section, (1) $M$ is a Riemannian manifold with a retraction $R$;
(2) for every $x \in M$, $\left\langle \ast, \ast\right\rangle_x$ is the Riemannian inner product on $T_x M$, and $\|\ast\|_x$ is the norm on $T_x M$ induced by $\left\langle  \ast, \ast \right\rangle_x$;
(3) $f:M\rightarrow \R$ is a $C^1$ function; (4) $C$ is an $R$-convex subset of $M$ equipped with a fixed $R$-convexity bundle $\{C_x\}_{x \in C}$.
The goal of this section is to generalize gradient projection algorithms to $f|_C$.

\subsection{Retraction-based gradient projection}
We introduce our gradient projection algorithms in Algorithm \ref{algorithm-rgpa} below.
\begin{algorithm}[H]
\caption{Retraction-Based Gradient Projection}
\label{algorithm-rgpa}
\begin{algorithmic}[1]
\Require a Riemannian manifold $M$ with a retraction $R$, a $C^1$ objective function $f:M\to\R$,
an $R$-convex subset $C$ of $M$ with an $R$-convexity bundle $\{C_x\}_{x \in C}$,
and one of the stepsize rules \ref{ur-limited-mini}, \ref{ur-Armijo-feasible}, \ref{ur-Armijo-projection-arc}, or \ref{ur-constant} below for the sequences $\{\alpha_k\}$ and $\{s_k\}$
\For{$k=0,1,2,\ldots$}
    \State Let $\Pi_{C_{x_k}}:T_{x_k}M\to C_{x_k}$ be the projection onto $C_{x_k}$ given by Proposition \ref{prop-convex-projection}(1)
    \State $d_k \gets \Pi_{C_{x_k}}\bigl(-s_k\nabla f(x_k)\bigr)$
    \State $x_{k+1} \gets R_{x_k}(\alpha_k d_k)$
\EndFor
\end{algorithmic}
\end{algorithm}
In Algorithm \ref{algorithm-rgpa}, $\nabla f$ is the gradient of $f$, which is the dual of $df$ via the Riemannian inner product.
Note that, in Algorithm \ref{algorithm-rgpa}, $x_k \in C$ and $d_k \in C_{x_k}$ for $k = 0, 1, 2, \ldots$.

\subsection{Stepsize rules for Algorithm \ref{algorithm-rgpa}}
We consider the following stepsize rules for choosing the sequences $\{\alpha_k\}$ and $\{s_k\}$ in Algorithm \ref{algorithm-rgpa}.
They are adaptations of the corresponding rules for gradient projection algorithms on convex sets in $\R^n$ \cite[Section 3.3]{Bertsekas-nonlinear-programming}.

\begin{ur}[Limited Minimization Rule]\label{ur-limited-mini}
Fix a constant $s>0$,
\begin{enumerate}
	\item set $s_k =s$ for $k =0,1,2, \dots$
	\item choose $\alpha_k \in [0,1]$ so that
		\begin{equation*}
		f(R_{x_k}(\alpha_k \Pi_{C_{x_k}}(-s\nabla f (x_k))))
		= \min_{\alpha \in [0,1]} f(R_{x_k}(\alpha \Pi_{C_{x_k}}(-s\nabla f (x_k)))).
		\end{equation*}
\end{enumerate}
\end{ur}

\begin{ur}[Armijo Rule Along the Feasible Direction]\label{ur-Armijo-feasible}
Fix constants $s>0$ and $\beta, \sigma \in (0,1)$,
\begin{enumerate}
	\item set $s_k =s$ for $k =0,1,2, \dots$
	\item set $\alpha_k = \beta^{m_k}$, where $m_k$ is the smallest non-negative integer $m$ satisfying
		\begin{equation}\label{eq-Armijo-feasible-inequality}
		f(x_k) - f(R_{x_k}(\beta^m\Pi_{C_{x_k}}(-s\nabla f (x_k))))
		\geq -\sigma \beta^m \left\langle \nabla f(x_k), \Pi_{C_{x_k}}(-s\nabla f (x_k))\right\rangle_{x_k}.
		\end{equation}
\end{enumerate}
\end{ur}

\begin{ur}[Armijo Rule Along the Projection Arc]\label{ur-Armijo-projection-arc}
Fix constants $\bar{s}>0$ and $\beta, \sigma \in (0,1)$,
\begin{enumerate}
	\item set $\alpha_k =1$ for $k =0,1,2, \dots$
	\item set $s_k = \beta^{m_k}\bar{s}$, where $m_k$ is the smallest non-negative integer $m$ satisfying
		\begin{equation}\label{eq-Armijo-projection-arc-inequality}
		f(x_k) - f(R_{x_k}(\Pi_{C_{x_k}}(-\beta^{m}\bar{s}\nabla f (x_k))))
		\geq -\sigma \left\langle \nabla f(x_k), \Pi_{C_{x_k}}(-\beta^{m}\bar{s}\nabla f (x_k))\right\rangle_{x_k}.
		\end{equation}
\end{enumerate}
\end{ur}

\begin{ur}[Constant Stepsize]\label{ur-constant}
Fix a constant $s>0$,
\begin{enumerate}
	\item set $\alpha_k =1$ for $k =0,1,2, \dots$
	\item set $s_k = s$ for $k =0,1,2, \dots$
\end{enumerate}
\end{ur}

\subsection{Stationary points}

The definition of stationary points below is a generalization of that for stationary points on a convex subset of $\R^n$.

\begin{definition}\label{def-stationary-pt}
A point $x^\ast \in C$ is a stationary point of $f|_{C}$ if and only if $df|_{x^\ast}(\mathbf{v}) \geq 0$ for all $\mathbf{v}\in C_{x^\ast}$, where $df$ is the differential of $f$.
\end{definition}

\begin{lemma}\label{lemma-local-min-stationary}
$x^\ast\in C$ is a local minimum of $f|_C$ if and only if $\mathbf{0}_{x^\ast}$ is a local minimum of $(f \circ R_{x^\ast})|_{C_{x^\ast}}$, which implies that $x^\ast$ is a stationary point of $f|_C$.
\end{lemma}

\subsection{Convergence results}
We next state our convergence results for Algorithm \ref{algorithm-rgpa} with the Stepsize Rules \ref{ur-limited-mini}-\ref{ur-constant}.

\begin{proposition}\label{prop-rgpa-convergence}
Assume that $x_0 \in C$ and $\{x_k\}_{k=0}^\infty$ is the sequence generated by Algorithm \ref{algorithm-rgpa},
where the stepsize sequences $\{\alpha_k\}$ and $\{s_k\}$ are chosen
by one of Stepsize Rules \ref{ur-limited-mini}, \ref{ur-Armijo-feasible} or \ref{ur-Armijo-projection-arc}.
Then every limit point of $\{x_k\}$ is a stationary point of $f|_C$.
\end{proposition}

To state our convergence result for Stepsize Rule \ref{ur-constant}, we need to define $R$-Lipschitz gradients.

\begin{definition}\label{def-R-lipschitz}
For each $x \in M$ the function $f\circ R_x: T_x M \rightarrow \R$ is a $C^1$ function. Its gradient $\nabla (f\circ R_x)$ is the vector field over $T_x M$ dual to the differential $d(f\circ R_x)$ via the inner product $\left\langle \ast,\ast\right\rangle_x$.
$\nabla f$ is said to be $R$-Lipschitz over $C$ if there is a constant $L>0$ such that
\begin{equation}\label{eq-def-R-Lipschitz}
\|\nabla (f\circ R_x)(\mathbf{v}) - \nabla (f\circ R_x)(\mathbf{0}_x)\|_x \leq L\|\mathbf{v}\|_x
\end{equation}
for every $x \in C$ and $\mathbf{v} \in C_x$.
\end{definition}

\begin{proposition}\label{prop-rgpa-convergence-constant-stepsize}
Assume that $\nabla f$ is $R$-Lipschitz over $C$. Fix an $L>0$ such that Inequality \eqref{eq-def-R-Lipschitz} is true and
a constant $s$ satisfying $0 < s< \frac{2}{L}$. For any $x_0 \in C$, let $\{x_k\}_{k=0}^\infty$ be the sequence generated
by Algorithm \ref{algorithm-rgpa}, where the stepsize sequences $\{\alpha_k\}$ and $\{s_k\}$ are chosen
by Stepsize Rule \ref{ur-constant}. Then every limit point of $\{x_k\}$ is a stationary point of $f|_C$.
\end{proposition}

Propositions \ref{prop-rgpa-convergence} and \ref{prop-rgpa-convergence-constant-stepsize} are proved in Appendix \ref{app-proofs-convergence}
by adapting the corresponding proofs for gradient projection algorithms on convex sets in $\R^n$ given in \cite[Section 3.3]{Bertsekas-nonlinear-programming}.


\section{Application: Weighted Low-Rank Approximation with Retraction-Specific Convex Constraint}\label{sec-wlra-constraint}
We denote the space of $m\times n$ matrices by $\R^{m\times n}$ and the $n\times p$ Stiefel manifold by $\St(n,p)$.
In this section, (1) $m$, $n$ and $p$ are fixed positive integers with the low-rank $p$ satisfying $p \leq \min\{m,n\}$;
(2) $A=[a_{i,j}]$ in $\R^{m\times n}$ is a given matrix of constants, that is, the matrix of data to be approximated;
(3) $W=[w_{i,j}]$ in $\R^{m\times n}$ is a given matrix of weights satisfying $w_{i,j} \geq 0$ for $(i,j) \in \{1,2,\dots,m\}\times \{1,2,\dots,n\}$ and $\sum_{i=1}^m\sum_{j=1}^n w_{i,j} > 0$;
(4) the regularization parameter $\lambda > 0$ is a fixed positive real value.

\subsection{Formulation}
\begin{problem}[A Regularized WLRA]\label{prob-wlra-reg}\cite{xu-yang-wu-2025-wlra-sgd-manifolds}
Define a function $h:\R^{m\times n} \rightarrow \R$ by
\begin{equation}\label{eq-def-wlra-loss}
	h(P) = \sum_{i=1}^m \sum_{j=1}^n w_{i,j}(a_{i,j}-p_{i,j})^2
\end{equation}
for $P=[p_{i,j}] \in \R^{m\times n}$. Solve for
\[
	\mathrm{argmin}\{h(P) + \lambda \|P\|_F^2 ~|~ P\in \R^{m\times n},~  \rank P \leq p\},
\]
where $\|P\|_F$ is the Frobenius norm of $P$.
\end{problem}

The function $h$ given in Equation \eqref{eq-def-wlra-loss} is the WLRA loss, which is widely used in various machine learning tasks \cite{candes-recht-2009-matrix-completion,koren-bell-volinsky-2009-recommender-mf}.
Prior studies applied Frobenius regularizations to the WLRA loss to alleviate ill-posedness and degeneracy
\cite{ban-woodruff-zhang-2019-regularized-wlra, xu-yang-wu-2025-wlra-sgd-manifolds}.

Define the product manifold
\begin{equation} \label{eq-def-prod-mfd}
\M =  \St(m,p)\times \R^p \times \St(n,p).
\end{equation}
The tangent space at $(U, \mathbf{x}, V) \in \M$ is
\begin{equation}\label{eq-tangent-space}
	T_{(U, \mathbf{x}, V)}\M = T_U \St(m,p)\times \R^p \times T_V \St(n,p).
\end{equation}
For tangent vectors $\xi=(\xi_U,\xi_x,\xi_V)$ and $\eta=(\eta_U,\eta_x,\eta_V)$ at $(U,\mathbf{x},V)$, we use the product metric
\[
\langle \xi,\eta\rangle_{(U,\mathbf{x},V)}
=\operatorname{tr}(\xi_U^T\eta_U)+\xi_x^T\eta_x+\operatorname{tr}(\xi_V^T\eta_V).
\]
This combines the induced Frobenius metrics on the Stiefel factors with the Euclidean metric on $\R^p$,
as in our implementation and Appendix \ref{sec-rgp}.
On the product manifold $\M$, define a retraction $R: T\M \rightarrow \M$ by
\begin{equation}\label{eq-def-retraction-prod-mfd}
R = R^{\St(m,p)}\times R^{\R^p} \times R^{\St(n,p)},
\end{equation}
where $R^{\St(m,p)}$ and $R^{\St(n,p)}$ are retractions on Stiefel manifolds, and $R^{\R^p}: \R^p \rightarrow \R^p$ is the standard retraction on $\R^p$ defined by
\begin{equation}\label{eq-def-retraction-euclidean}
R_{\mathbf{x}}^{\R^p}(\mathbf{z}) = \mathbf{x} + \mathbf{z}
\end{equation}
for any $\mathbf{x}, \mathbf{z} \in \R^p$. We reformulate Problem \ref{prob-wlra-reg} to the following unconstrained problem on $\M$.

\begin{problem}[The Regularized WLRA Problem on $\M$]\label{prob-wlra-reg-mfd}
Define a function $H: \M \rightarrow \R$ by
\begin{equation}\label{eq-def-wlra-loss-mfd}
	H(U,\mathbf{x},V) = h(U D_p(\mathbf{x}) V^T)
\end{equation}
for $(U,\mathbf{x},V) \in \M$, where $h$ is given in Equation \eqref{eq-def-wlra-loss} and the function $D_p: \R^p \rightarrow \R^{p\times p}$ is defined as
		\begin{equation}\label{eq-def-d-p}
			D_p(\mathbf{x}) = \begin{bmatrix} x_1 & 0 & \cdots & 0 \\ 0 & x_2 & \cdots & 0 \\ \vdots & \vdots & \ddots & \vdots \\ 0 & 0 & \cdots & x_p \end{bmatrix}
		\end{equation}
		for $\mathbf{x} = [x_1, \ldots, x_p]^T \in \R^p$.
	Solve for
\[
	\mathrm{argmin}\{H(U, \mathbf{x}, V) + \lambda \|\mathbf{x}\|^2 ~|~ (U,\mathbf{x},V) \in \M\},
\]
where $\|\mathbf{x}\|$ is the Euclidean norm of $\mathbf{x}$.
Since $\|\mathbf{x}\| = \|U D_p(\mathbf{x}) V^T\|_F$ for $(U, \mathbf{x}, V) \in \M$,
Problem \ref{prob-wlra-reg-mfd} is equivalent to Problem \ref{prob-wlra-reg}.
\end{problem}

The function $H$ given in Equation \eqref{eq-def-wlra-loss-mfd} is the WLRA loss on $\M$.
Riemannian gradient descent provides a standard method for Problem \ref{prob-wlra-reg-mfd} \cite[Algorithm 4.1]{Boumal-Book}.
See \cite{xu-yang-wu-2025-wlra-sgd-manifolds} for an example.
Alternatively, we can impose a retraction-specific convex constraint on the WLRA loss over $\M$, which also alleviates ill-posedness and degeneracy.
For $\mathbf{x}$ in $\R^p$ and $r > 0$, denote the closed ball centered at $\mathbf{x}$ with radius $r$ by
\[
	\overline{B}_r(\mathbf{x}) = \{\mathbf{z} \in \R^p ~|~ \|\mathbf{z} - \mathbf{x} \| \leq r\},
\]
which is a closed convex subset of $\R^p$.

\begin{problem}[WLRA Problem on $\M$ with Retraction-Specific Convex Constraint]\label{prob-wlra-constrained}
Fix a positive number $r$ and let $H:\M \rightarrow \R$ be as in Equation \eqref{eq-def-wlra-loss-mfd}. Solve for
\[
	\mathrm{argmin}\{H(U,\mathbf{x}, V) ~|~ (U,\mathbf{x},V)  \in C^r \},
\]
where $C^r: = \St(m,p) \times \overline{B}_r(\mathbf{0}) \times \St(n,p) \subset \M$ and $\mathbf{0}$ is the zero vector of $\R^p$.
\end{problem}

\begin{proposition}\label{prop-R-convexity-wlra}
For $C^r$ given in Problem \ref{prob-wlra-constrained} and $(U, \mathbf{x}, V) \in C^r$,
define $C^r_{(U, \mathbf{x}, V)}$ to be the subset of $T_{(U, \mathbf{x}, V)}\M$ given by
\begin{equation}\label{eq-def-r-convex-bundle-wlra}
	C^r_{(U, \mathbf{x}, V)}: = T_U \St(m,p) \times \overline{B}_r(-\mathbf{x}) \times T_V \St(n,p).
\end{equation}
For a retraction $R$ on $\M$ given in Equation \eqref{eq-def-retraction-prod-mfd}, $\{C^r_{(U, \mathbf{x}, V)}\}_{(U, \mathbf{x}, V) \in C^r}$ is an $R$-convexity bundle on $C^r$ and $C^r$ is an $R$-convex set on the manifold $\M$.
\end{proposition}

The following lemma gives a connection between Problems \ref{prob-wlra-reg-mfd} and \ref{prob-wlra-constrained}.
\begin{lemma}\label{lemma-relation-bw-reg-constraint}
If $(U^*, \mathbf{x}^*, V^*) \in \M$ is a solution to Problem \ref{prob-wlra-reg-mfd}, then $(U^*, \mathbf{x}^*, V^*) \in C^r$,
where $C^r$ is the retraction-specific convex set given in Problem \ref{prob-wlra-constrained} with
\begin{equation}\label{eq-def-r-lmbda-relation}
	r = \frac{\|\sqrt[\bigodot]{W} \odot A\|_F}{\sqrt{\lambda}},
\end{equation}
where the matrix $\sqrt[\bigodot]{W}$ is defined as the Hadamard (element-wise) root of the matrix $W$, and $\|\sqrt[\bigodot]{W} \odot A\|_F$ is the Frobenius norm
of the Hadamard (element-wise) product of $\sqrt[\bigodot]{W}$ and $A$.
\end{lemma}
Proposition \ref{prop-R-convexity-wlra} and Lemma \ref{lemma-relation-bw-reg-constraint} are proved in Appendix \ref{app-proofs-more}.

\subsection{Numerical Validation}
\label{sec-experiments}
We provide experimental results for Algorithm \ref{algorithm-rgpa}
with the Armijo Stepsize Rules \ref{ur-Armijo-feasible} and \ref{ur-Armijo-projection-arc} to validate our convergence results in Proposition \ref{prop-rgpa-convergence}.
We used the image completion task (modeled as WLRA) as our experimental setting and drew data from the training splits of the MNIST and grayscale CIFAR-10 datasets \cite{lecun-bottou-bengio-haffner-1998-mnist, krizhevsky-2009-cifar10}.
To save computational resources\footnote{To avoid the CPU-GPU synchronization overhead caused by frequent data transfers when implementing the Armijo stepsize rules,
we ran the experiments on a CPU (an Apple Silicon M5 Pro chip with 6 super cores and 12 performance cores).}, we randomly sampled $6000$ images from the MNIST training split and $5000$ images from the CIFAR-10 training split.
The samples were balanced across all image categories in both datasets.
We followed the steps below to simulate the missing data in the image completion task:
\begin{enumerate}
    \item We flattened every sampled image to a column vector and concatenated all the flattened vectors to form a matrix of size $784 \times 6000$ (for MNIST) or $1024 \times 5000$ (for CIFAR-10).
    We used this matrix as the data matrix $A$.
    \item We randomly selected a fraction of the entries in the matrix $A$ and set them to zero. We refer to this fraction as the ``masking rate''.
    We considered an entry $a_{i,j}$ of $A$ ``masked'' if we selected it for this operation\footnote{We used 42 as the random seed when sampling the MNIST and CIFAR-10 datasets and the masked entries.}.
    \item We defined the binary weight matrix $W$ such that $w_{i,j} = 0$ if the entry $a_{i,j}$ of $A$ was masked and $w_{i,j} = 1$ otherwise.
\end{enumerate}

As the main routines in our experiments,
we modeled the image completion task as Problem \ref{prob-wlra-constrained},
and applied Algorithm \ref{algorithm-rgpa}
with the Armijo Stepsize Rules \ref{ur-Armijo-feasible} and \ref{ur-Armijo-projection-arc}.
We refer to the main routines as ``Feasible Direction'' and ``Projection Arc'', respectively.
For the benchmark routine, we modeled the image completion task as Problem \ref{prob-wlra-reg-mfd}
and applied the Riemannian gradient descent algorithm with the Armijo stepsize rule.
We refer to the benchmark routine as ``Regularized Armijo''.
The precedure for applying Algorithm \ref{algorithm-rgpa} to Problem \ref{prob-wlra-constrained} appears in Appendix \ref{sec-rgp}.
The procedure for applying the Riemannian gradient descent to Problem \ref{prob-wlra-reg-mfd} appears in \cite{xu-yang-wu-2025-wlra-sgd-manifolds}.

In our experiments, we chose the following hyperparameters: (1) the low-rank $p \in \{32, 64, 128\}$;
(2) the regularization parameter $\lambda \in \{10^{-2}, 10^{-4}, 10^{-6}\}$ for the benchmark routine (for both main routines, we followed Equation \eqref{eq-def-r-lmbda-relation} to determine the value of $r$ corresponding to each $\lambda$);
(3) two projection-like retractions: the QR-decomposition-based retraction and the polar-decomposition-based retraction\footnote{The mathematical details about these retractions appear in Appendix \ref{sec-rgp}.}.
We refer to each combination of a low-rank $p$, a regularization parameter $\lambda$, and a retraction as a ``configuration''.
Our choices of the hyperparameters above yielded $18$ configurations.
We initialized our experiments with the truncated singular value decomposition of the matrix $A \odot W$ to make the initialization feasible in both main routines.
We used the root mean square error (RMSE) as our primary evaluation metric and set the masking rate to be $0.5$.

\enlargethispage{\baselineskip}
\textbf{Grid Search Stage}: We first ran a grid search on the Armijo-rule parameters $(\beta,\sigma,s)$
at 27 Armijo-rule parameter grid points for each routine within each configuration, with $200$ iterations for each run\footnote{\interlinepenalty=10000 For the ``Projection Arc'' routine,
the Armijo-rule parameters searched in this stage were $(\beta, \sigma, \bar{s})$ given in Rule \ref{ur-Armijo-projection-arc}.
For simplicity of notation, we denote the initial trial step $\bar{s}$ in Rule \ref{ur-Armijo-projection-arc} as $s$ in this section.
Further details of the grid search design appear in Appendix \ref{sec-grid-search-design}.}.
Table \ref{tab:mnist-cifar-armijo-grid} summarizes the grid search results.
For each routine and configuration, ``mean Armijo evaluations'' averages inequality checks per iteration over each successful run's $200$ iterations,
then across successful grid runs; failed runs are excluded and counted separately.
This is not a run total or wall-clock cost; Appendix \ref{sec-armijo-evaluation-metric} defines the counting convention.

\textbf{Routine Comparison Stage}: Next, we compared the three routines over longer runs.
Within each configuration, we ran each routine for $1000$ iterations
with the optimal Armijo-rule parameters selected from the grid search.
Figure \ref{fig:mnist-cifar-armijo-rmse-curves} presents the RMSE convergence curves for all the configurations.

Combining our experimental results from both stages, we observe that our main routines
could outperform the benchmark routine over both short and long horizons
when we use the QR-decomposition-based retraction and the low-rank $p = 128$ on the sampled MNIST dataset.

We would like to emphasize that the main aim of this paper is theoretical: to establish a framework for retraction-based convex optimization on Riemannian manifolds and convergence results within the framework.
We do not aim to provide a comprehensive study of the weighted low-rank approximation.
Hence, our numerical results are not meant to provide a comprehensive comparison with other Riemannian optimization algorithms on the image completion task (for example, those in \cite{boumal-absil-2011-rtrmc, mishra-sepulchre-2014-r3mc}).
Our choice of the benchmark routine is motivated by its theoretical connection with our main routines (Lemma \ref{lemma-relation-bw-reg-constraint}).

To further validate our convergence results, we repeated the two-stage experiments on both sampled datasets with masking rates of $0.25$ and $0.75$,
and report Stage 1 grid search results on the full MNIST and grayscale CIFAR-10 training splits with a masking rate of $0.5$.
These full-training experiments use all $60000$ MNIST images and all $50000$ CIFAR-10 images;
Tables \ref{tab:full-train-armijo-grid} and \ref{tab:full-train-cifar10-armijo-grid} summarize their results.
Appendix \ref{sec-numerical-supp} presents these supplementary results.
This appendix also includes a paired, one-sided $t$-test on the grid search metrics to assess the aggregate differences between the three routines.


\section{Conclusion and Final Remarks}
In this paper, we establish a framework for retraction-based convex optimization on Riemannian manifolds,
which includes a notion of retraction-specific convex sets and retraction-based gradient projection algorithms.
We provide convergence results for the algorithms within this framework.
We applied our algorithms to study the weighted low-rank approximation and numerically validate our convergence results.

\textbf{Limitations.} Wen and Yin \cite{wen-yin-2013} introduced the Cayley transform as an iterative scheme for preserving the orthogonality of iterates on the Stiefel manifold.
Recent studies used the Cayley transform to design optimization algorithms on the Stiefel manifold for training large-scale neural networks \cite{jun-li-sinisa-2020, kume-yamada-2024}.
However, the Cayley transform is not a retraction and therefore falls outside the scope of our framework.
The convergence analysis of stochastic gradient projection methods for constrained non-convex programming problems on Euclidean spaces
is an emerging topic in the recent deep learning literature \cite{bianchi-jakubowicz-2013-multi-agent-psg, ghadimi-lan-zhang-2016-rspg, davis-drusvyatskiy-2019-stochastic-model-based, alacaoglu-lyu-2023-dependent-psg, zheng-lamperski-2026-psgd-goldstein}.
However, the algorithms in this paper are deterministic.
Extending our framework to Cayley transform and stochastic gradient projection algorithms on manifolds could be interesting directions for future research.

\textbf{Code and Data Release.} The code and processed reference results for this project are hosted in the GitHub repository \url{https://github.com/ConglongXu-GWU/manifold-grad-proj-release}.

\begin{table*}[t]
\centering
\small
\setlength{\tabcolsep}{0.6mm}
\begin{tabular}{@{}cc|rrrr|rrrr|rrrr@{}}
\hline
\multicolumn{14}{c}{\textbf{Sampled MNIST}} \\
\hline
\multicolumn{14}{c}{\textbf{QR retraction}} \\
\hline
rank & $\lambda$ & \multicolumn{4}{c|}{\shortstack{Regularized Armijo\\(Benchmark)}} & \multicolumn{4}{c|}{\shortstack{Feasible Direction\\(Main Routine)}} & \multicolumn{4}{c}{\shortstack{Projection Arc\\(Main Routine)}} \\
 & & best & mean & evals & fail. & best & mean & evals & fail. & best & mean & evals & fail. \\
\hline
32 & $10^{-2}$ & 0.2138 $(0.9,0.75,0.5)$ & 0.2153 & 50.3 & 0 & 0.2134 $(0.9,0.75,0.5)$ & 0.2152 & 50.3 & 0 & \textbf{0.2133 $(0.9,0.75,0.3)$} & 0.2152 & 50.3 & 0 \\
32 & $10^{-4}$ & 0.2137 $(0.9,0.75,0.3)$ & 0.2153 & 50.3 & 0 & 0.2134 $(0.9,0.75,0.5)$ & 0.2152 & 50.3 & 0 & \textbf{0.2133 $(0.9,0.75,0.3)$} & 0.2152 & 50.3 & 0 \\
32 & $10^{-6}$ & 0.2140 $(0.9,0.75,0.5)$ & 0.2153 & 50.2 & 0 & 0.2134 $(0.9,0.75,0.5)$ & 0.2152 & 50.3 & 0 & \textbf{0.2133 $(0.9,0.75,0.3)$} & 0.2152 & 50.3 & 0 \\
64 & $10^{-2}$ & \textbf{0.2225 $(0.9,0.75,0.5)$} & 0.2296 & 49.9 & 0 & 0.2236 $(0.9,0.75,0.5)$ & 0.2297 & 49.9 & 0 & 0.2228 $(0.9,0.75,0.5)$ & 0.2297 & 49.9 & 0 \\
64 & $10^{-4}$ & \textbf{0.2228 $(0.9,0.75,0.5)$} & 0.2296 & 49.9 & 0 & 0.2236 $(0.9,0.75,0.5)$ & 0.2297 & 49.9 & 0 & 0.2228 $(0.9,0.75,0.5)$ & 0.2297 & 49.9 & 0 \\
64 & $10^{-6}$ & 0.2232 $(0.9,0.75,0.5)$ & 0.2296 & 49.9 & 0 & 0.2236 $(0.9,0.75,0.5)$ & 0.2297 & 49.9 & 0 & \textbf{0.2228 $(0.9,0.75,0.5)$} & 0.2297 & 49.9 & 0 \\
128 & $10^{-2}$ & \textbf{0.2735 $(0.9,0.75,1)$} & 0.2791 & 49.1 & 0 & 0.2736 $(0.9,0.75,1)$ & 0.2790 & 49.1 & 0 & 0.2738 $(0.9,0.75,1)$ & 0.2790 & 49.1 & 0 \\
128 & $10^{-4}$ & \textbf{0.2734 $(0.9,0.75,1)$} & 0.2790 & 49.1 & 0 & 0.2736 $(0.9,0.75,1)$ & 0.2790 & 49.1 & 0 & 0.2738 $(0.9,0.75,1)$ & 0.2790 & 49.1 & 0 \\
128 & $10^{-6}$ & 0.2737 $(0.9,0.75,0.5)$ & 0.2790 & 49.1 & 0 & \textbf{0.2736 $(0.9,0.75,1)$} & 0.2790 & 49.1 & 0 & 0.2738 $(0.9,0.75,1)$ & 0.2790 & 49.1 & 0 \\
\hline
\multicolumn{14}{c}{\textbf{POLAR retraction}} \\
\hline
rank & $\lambda$ & \multicolumn{4}{c|}{\shortstack{Regularized Armijo\\(Benchmark)}} & \multicolumn{4}{c|}{\shortstack{Feasible Direction\\(Main Routine)}} & \multicolumn{4}{c}{\shortstack{Projection Arc\\(Main Routine)}} \\
 & & best & mean & evals & fail. & best & mean & evals & fail. & best & mean & evals & fail. \\
\hline
32 & $10^{-2}$ & 0.2149 $(0.9,0.75,0.5)$ & 0.2159 & 52.9 & 2 & \textbf{0.2145 $(0.9,0.5,0.3)$} & 0.2160 & 59.5 & 0 & 0.2147 $(0.9,0.5,1)$ & 0.2159 & 62.1 & 0 \\
32 & $10^{-4}$ & 0.2145 $(0.9,0.5,1)$ & 0.2158 & 50.4 & 3 & \textbf{0.2145 $(0.9,0.5,0.3)$} & 0.2160 & 59.5 & 0 & 0.2147 $(0.9,0.5,1)$ & 0.2159 & 62.1 & 0 \\
32 & $10^{-6}$ & 0.2150 $(0.5,0.5,0.3)$ & 0.2160 & 55.6 & 2 & \textbf{0.2145 $(0.9,0.5,0.3)$} & 0.2160 & 59.5 & 0 & 0.2147 $(0.9,0.5,1)$ & 0.2159 & 62.1 & 0 \\
64 & $10^{-2}$ & 0.2280 $(0.75,0.75,0.5)$ & 0.2317 & 57.7 & 0 & 0.2280 $(0.9,0.75,0.5)$ & 0.2320 & 58.9 & 0 & \textbf{0.2274 $(0.9,0.75,0.5)$} & 0.2315 & 58.0 & 0 \\
64 & $10^{-4}$ & 0.2276 $(0.9,0.5,0.5)$ & 0.2315 & 55.8 & 0 & 0.2280 $(0.9,0.75,0.5)$ & 0.2320 & 58.9 & 0 & \textbf{0.2274 $(0.9,0.75,0.5)$} & 0.2315 & 58.0 & 0 \\
64 & $10^{-6}$ & 0.2281 $(0.75,0.5,0.3)$ & 0.2329 & 61.3 & 0 & 0.2280 $(0.9,0.75,0.5)$ & 0.2320 & 58.9 & 0 & \textbf{0.2274 $(0.9,0.75,0.5)$} & 0.2315 & 58.0 & 0 \\
128 & $10^{-2}$ & 0.2772 $(0.75,0.75,0.3)$ & 0.2805 & 57.9 & 0 & \textbf{0.2743 $(0.9,0.75,1)$} & 0.2806 & 56.8 & 0 & 0.2748 $(0.9,0.75,1)$ & 0.2803 & 56.1 & 0 \\
128 & $10^{-4}$ & 0.2774 $(0.9,0.75,1)$ & 0.2806 & 57.8 & 0 & \textbf{0.2743 $(0.9,0.75,1)$} & 0.2806 & 56.8 & 0 & 0.2748 $(0.9,0.75,1)$ & 0.2803 & 56.1 & 0 \\
128 & $10^{-6}$ & 0.2775 $(0.9,0.75,0.5)$ & 0.2806 & 58.1 & 0 & \textbf{0.2743 $(0.9,0.75,1)$} & 0.2806 & 56.8 & 0 & 0.2748 $(0.9,0.75,1)$ & 0.2803 & 56.1 & 0 \\
\hline
\multicolumn{14}{c}{\textbf{Sampled CIFAR-10}} \\
\hline
\multicolumn{14}{c}{\textbf{QR retraction}} \\
\hline
rank & $\lambda$ & \multicolumn{4}{c|}{\shortstack{Regularized Armijo\\(Benchmark)}} & \multicolumn{4}{c|}{\shortstack{Feasible Direction\\(Main Routine)}} & \multicolumn{4}{c}{\shortstack{Projection Arc\\(Main Routine)}} \\
 & & best & mean & evals & fail. & best & mean & evals & fail. & best & mean & evals & fail. \\
\hline
32 & $10^{-2}$ & \textbf{0.2903 $(0.9,0.75,0.5)$} & 0.2937 & 58.3 & 0 & 0.2908 $(0.9,0.75,1)$ & 0.2938 & 58.3 & 0 & 0.2905 $(0.9,0.75,0.5)$ & 0.2938 & 58.4 & 0 \\
32 & $10^{-4}$ & \textbf{0.2895 $(0.9,0.75,1)$} & 0.2936 & 58.3 & 0 & 0.2908 $(0.9,0.75,1)$ & 0.2938 & 58.3 & 0 & 0.2905 $(0.9,0.75,0.5)$ & 0.2938 & 58.4 & 0 \\
32 & $10^{-6}$ & \textbf{0.2898 $(0.75,0.5,0.5)$} & 0.2938 & 58.4 & 0 & 0.2908 $(0.9,0.75,1)$ & 0.2938 & 58.3 & 0 & 0.2905 $(0.9,0.75,0.5)$ & 0.2938 & 58.4 & 0 \\
64 & $10^{-2}$ & 0.3089 $(0.9,0.75,1)$ & 0.3143 & 58.3 & 0 & \textbf{0.3085 $(0.9,0.75,0.5)$} & 0.3143 & 58.2 & 0 & 0.3096 $(0.9,0.75,0.5)$ & 0.3143 & 58.2 & 0 \\
64 & $10^{-4}$ & 0.3093 $(0.9,0.75,0.5)$ & 0.3145 & 58.3 & 0 & \textbf{0.3085 $(0.9,0.75,0.5)$} & 0.3143 & 58.2 & 0 & 0.3096 $(0.9,0.75,0.5)$ & 0.3143 & 58.2 & 0 \\
64 & $10^{-6}$ & 0.3103 $(0.9,0.75,0.5)$ & 0.3144 & 58.3 & 0 & \textbf{0.3085 $(0.9,0.75,0.5)$} & 0.3143 & 58.2 & 0 & 0.3096 $(0.9,0.75,0.5)$ & 0.3143 & 58.2 & 0 \\
128 & $10^{-2}$ & 0.3508 $(0.9,0.75,0.3)$ & 0.3528 & 57.9 & 0 & \textbf{0.3498 $(0.9,0.75,1)$} & 0.3528 & 57.9 & 0 & 0.3499 $(0.9,0.75,1)$ & 0.3528 & 57.9 & 0 \\
128 & $10^{-4}$ & \textbf{0.3496 $(0.9,0.75,1)$} & 0.3528 & 57.9 & 0 & 0.3498 $(0.9,0.75,1)$ & 0.3528 & 57.9 & 0 & 0.3499 $(0.9,0.75,1)$ & 0.3528 & 57.9 & 0 \\
128 & $10^{-6}$ & 0.3511 $(0.9,0.75,1)$ & 0.3529 & 57.9 & 0 & \textbf{0.3498 $(0.9,0.75,1)$} & 0.3528 & 57.9 & 0 & 0.3499 $(0.9,0.75,1)$ & 0.3528 & 57.9 & 0 \\
\hline
\multicolumn{14}{c}{\textbf{POLAR retraction}} \\
\hline
rank & $\lambda$ & \multicolumn{4}{c|}{\shortstack{Regularized Armijo\\(Benchmark)}} & \multicolumn{4}{c|}{\shortstack{Feasible Direction\\(Main Routine)}} & \multicolumn{4}{c}{\shortstack{Projection Arc\\(Main Routine)}} \\
 & & best & mean & evals & fail. & best & mean & evals & fail. & best & mean & evals & fail. \\
\hline
32 & $10^{-2}$ & 0.2937 $(0.75,0.5,0.3)$ & 0.2964 & 56.4 & 6 & \textbf{0.2925 $(0.75,0.5,0.3)$} & 0.2961 & 57.0 & 6 & 0.2942 $(0.75,0.5,0.5)$ & 0.2965 & 63.7 & 6 \\
32 & $10^{-4}$ & 0.2934 $(0.75,0.5,1)$ & 0.2958 & 54.6 & 6 & \textbf{0.2925 $(0.75,0.5,0.3)$} & 0.2961 & 57.0 & 6 & 0.2942 $(0.75,0.5,0.5)$ & 0.2965 & 63.7 & 6 \\
32 & $10^{-6}$ & \textbf{0.2925 $(0.75,0.5,0.3)$} & 0.2960 & 56.0 & 6 & 0.2925 $(0.75,0.5,0.3)$ & 0.2961 & 57.0 & 6 & 0.2942 $(0.75,0.5,0.5)$ & 0.2965 & 63.7 & 6 \\
64 & $10^{-2}$ & 0.3140 $(0.5,0.5,0.5)$ & 0.3166 & 53.7 & 6 & \textbf{0.3130 $(0.75,0.5,1)$} & 0.3162 & 50.5 & 6 & 0.3130 $(0.75,0.5,1)$ & 0.3165 & 58.5 & 6 \\
64 & $10^{-4}$ & 0.3134 $(0.75,0.5,1)$ & 0.3167 & 53.0 & 6 & \textbf{0.3130 $(0.75,0.5,1)$} & 0.3162 & 50.5 & 6 & 0.3130 $(0.75,0.5,1)$ & 0.3165 & 58.5 & 6 \\
64 & $10^{-6}$ & 0.3141 $(0.75,0.5,0.3)$ & 0.3164 & 52.9 & 6 & \textbf{0.3130 $(0.75,0.5,1)$} & 0.3162 & 50.5 & 6 & 0.3130 $(0.75,0.5,1)$ & 0.3165 & 58.5 & 6 \\
128 & $10^{-2}$ & 0.3524 $(0.5,0.25,0.5)$ & 0.3541 & 51.5 & 6 & \textbf{0.3523 $(0.5,0.5,0.5)$} & 0.3540 & 50.3 & 6 & 0.3523 $(0.5,0.5,0.5)$ & 0.3542 & 60.1 & 6 \\
128 & $10^{-4}$ & \textbf{0.3520 $(0.5,0.5,0.5)$} & 0.3540 & 51.5 & 6 & 0.3523 $(0.5,0.5,0.5)$ & 0.3540 & 50.3 & 6 & 0.3523 $(0.5,0.5,0.5)$ & 0.3542 & 60.1 & 6 \\
128 & $10^{-6}$ & \textbf{0.3522 $(0.75,0.5,0.5)$} & 0.3539 & 49.1 & 6 & 0.3523 $(0.5,0.5,0.5)$ & 0.3540 & 50.3 & 6 & 0.3523 $(0.5,0.5,0.5)$ & 0.3542 & 60.1 & 6 \\
\hline
\end{tabular}

\caption{For each retraction-rank-regularization configuration,
the table reports each routine's best and mean final RMSE, mean Armijo inequality checks per iteration (``mean Armijo Evaluations''), and failure count.
Evaluation counts are averaged over each successful run's $200$ iterations, then across successful grid runs.
Failures count grid runs that reached the $200$-backtrack limit.
Failed records are excluded from evaluating the other three metrics.
Parentheses give the Armijo parameters $(\beta,\sigma,s)$ for the best RMSE entry.
Boldface marks the lowest best RMSE among the three routines in each row. The masking rate is set to $0.5$.
}
\label{tab:mnist-cifar-armijo-grid}
\end{table*}

\FloatBarrier

\begin{figure*}[t]
\centering
\includegraphics[width=\textwidth]{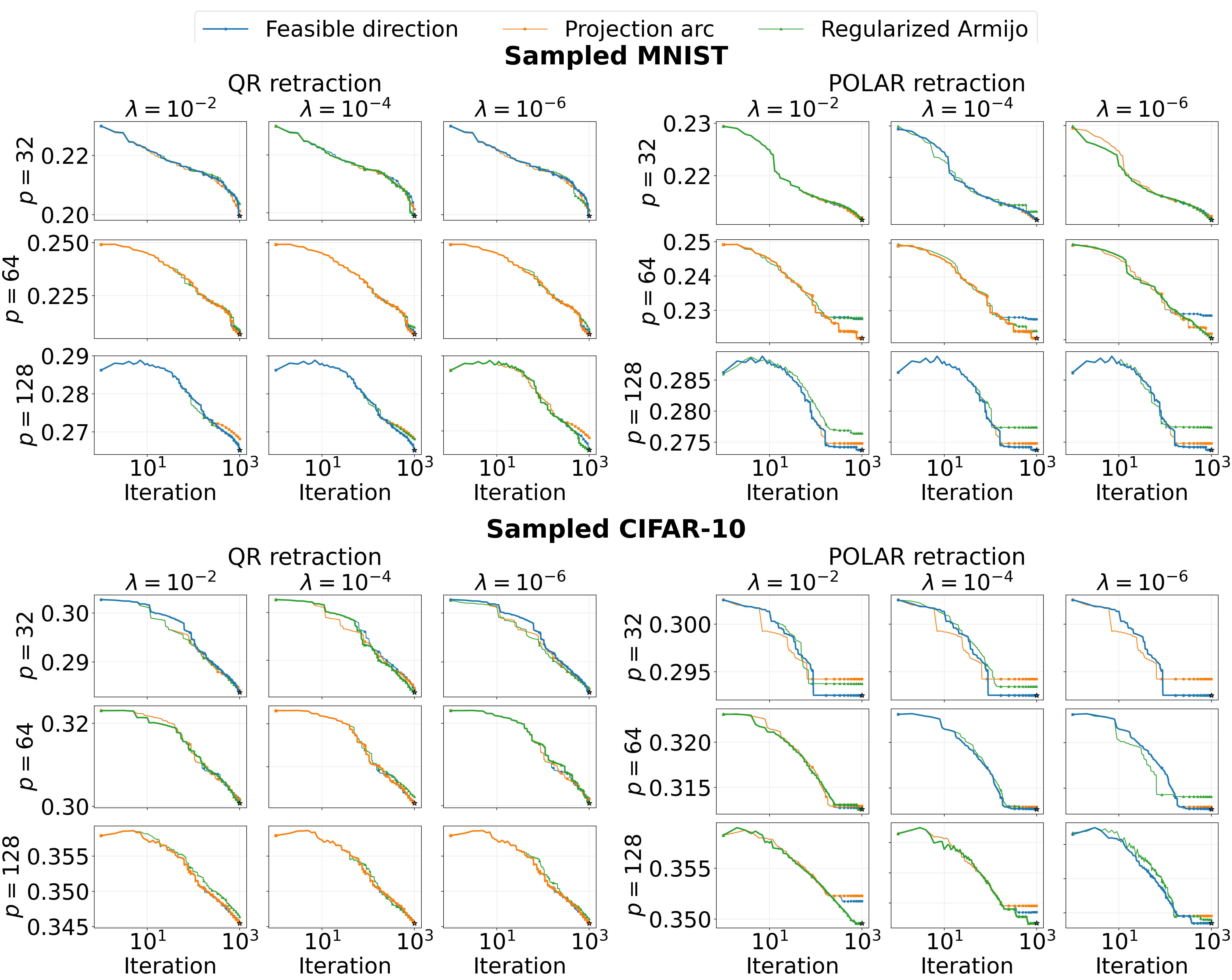}
\caption{RMSE convergence over $1000$ iterations using the optimal Armijo-rule parameters selected from the grid search summarized in Table \ref{tab:mnist-cifar-armijo-grid}.
The left panels use the QR decomposition-based retraction and the right panels use the polar decomposition-based retraction.
The $x$-axis is logarithmic. The thicker curve with a star at its final point marks the lowest final RMSE in that subplot. The masking rate is set to $0.5$.
On the sampled MNIST dataset (the upper panel),
(1) the ``Feasible Direction'' routine appears to consistently outperform the benchmark routine when using the polar-decomposition-based retraction and the low-rank $p = 128$,
(2 )the ``Projection Arc'' routine appears to consistently outperform the benchmark routine when using the QR-decomposition-based retraction and the low-rank $p = 64$.
On the sampled CIFAR-10 dataset (the lower pannel), (1) the ``Feasible Direction'' routine appears to consistently outperform the benchmark routine when using the polar-decomposition-based retraction and the low-rank $p = 32$,
(2) the ``Projection Arc'' routine appears to consistently outperform the Regularized Armijo (benchmark) routine
when using the QR-decomposition-based retraction and the low-rank $p = 128$.}
\label{fig:mnist-cifar-armijo-rmse-curves}
\end{figure*}

\FloatBarrier

\clearpage
\bibliography{references}

\clearpage
\appendix

\section{Proofs of Results in Section \ref{sec-rgpa}}\label{app-proofs-convergence}
Recall that we assume in Section \ref{sec-rgpa} that
\begin{enumerate}
	\item $M$ is a Riemannian manifold;
	\item for every $x \in M$, $\left\langle \ast, \ast\right\rangle_x$ is the Riemannian inner product on $T_x M$, and $\|\ast\|_x$ is the norm on $T_x M$ induced by $\left\langle  \ast, \ast \right\rangle_x$;
	\item $R$ is a retraction on $M$;
	\item $f:M\rightarrow \R$ is a $C^1$ function;
	\item $C$ is an $R$-convex subset of $M$ equipped with an $R$-convexity bundle $\{C_x\}_{x \in C}$.
\end{enumerate}

Note that, $\nabla f$ and $\nabla (f \circ R_x)$ are the gradients of $f$ and $f \circ R_x$, which are the duals of $df$ and $d(f \circ R_x)$ via the Riemannian inner product.
Since $(dR_x)|_{\mathbf{0}_x} = \id_{T_x M}$, we have that, for any $x \in M$ and $\mathbf{v}\in T_x M$
\begin{equation*}
\left\langle \nabla (f\circ R_x) (\mathbf{0}_x), \mathbf{v} \right\rangle_x
= d(f\circ R_x)|_{\mathbf{0}_x} (\mathbf{v})
= (df)|_{x}\circ (dR_x)|_{\mathbf{0}_x} (\mathbf{v})
= df|_{x}(\mathbf{v}) =\left\langle \nabla f(x) ,\mathbf{v}\right\rangle_x.
\end{equation*}
Thus,
\begin{equation}\label{eq-gradient-coincide}
\nabla (f\circ R_x) (\mathbf{0}_x) = \nabla f(x).
\end{equation}

We first present the proof for Lemma \ref{lemma-local-min-stationary}.
\begin{proof}[Proof of Lemma \ref{lemma-local-min-stationary}]
By the local surjectivity of $\{C_x\}_{x \in C}$, there is an open neighborhood $V_{x^\ast}$ of $\mathbf{0}_{x^\ast}$ in $C_{x^\ast}$
such that $R_{x^\ast}|_{V_{x^\ast}}: V_{x^\ast} \rightarrow R_{x^\ast} (V_{x^\ast})$
is a homeomorphism and that $R_{x^\ast}(V_{x^\ast})$ is an open neighborhood of $x^\ast$ in $C$.
Thus, $x^\ast\in C$ is a local minimum of $f|_C$ if and only if $\mathbf{0}_{x^\ast}$
is a local minimum of $(f \circ R_{x^\ast})|_{V_{x^\ast}}$,
that is, a local minimum of $(f \circ R_{x^\ast})|_{C_{x^\ast}}$.
If $\mathbf{0}_{x^\ast}$ is a local minimum of $(f \circ R_{x^\ast})|_{C_{x^\ast}}$,
then $d(f \circ R_{x^\ast})|_{\mathbf{0}_{x^\ast}} (\mathbf{v}) \geq 0$ for all $\mathbf{v} \in C_{x^\ast}$.
But $d(f \circ R_{x^\ast})|_{\mathbf{0}_{x^\ast}} =df|_{x^\ast} \circ dR_{x^\ast}|_{\mathbf{0}_{x^\ast}} =  df|_{x^\ast}$
since $dR_{x^\ast}|_{\mathbf{0}_{x^\ast}} = \id_{T_{x^\ast} M}$. This shows that $x^\ast$ is  a stationary point of $f|_C$.
\end{proof}

In the following subsections, we present proofs for Propositions \ref{prop-rgpa-convergence} and \ref{prop-rgpa-convergence-constant-stepsize}, which closely resemble the proofs of the convergence results of gradient project algorithms in $\R^n$ given in for example \cite[Chapter 3]{Bertsekas-nonlinear-programming}. We will first prove Proposition \ref{prop-rgpa-convergence} for cases where the stepsizes are given by the limited minimization rule and the Armijo rule along the feasible direction. Then, we prove Proposition \ref{prop-rgpa-convergence-constant-stepsize}. After that, we prove Proposition \ref{prop-rgpa-convergence} for the case where the stepsizes are given by the Armijo rule along the projection arc.

\subsection{The limited minimization rule and the Armijo rule along the feasible direction}

\begin{lemma}\label{lemma-stationary-projection}
For a point $x^{\ast} \in C$, the following statements are equivalent:
\begin{enumerate}
	\item $x^\ast$ is a stationary point of $f|_{C}$.
	\item $\Pi_{C_{x^\ast}}(-s\nabla f (x^\ast)) = \mathbf{0}_{x^\ast}$ for some $s>0$.
	\item $\Pi_{C_{x^\ast}}(-s\nabla f (x^\ast)) = \mathbf{0}_{x^\ast}$ for all $s>0$.
\end{enumerate}
\end{lemma}

\begin{proof}
Clearly, (3) $\Rightarrow$ (2). To see that (2) $\Rightarrow$ (1), one just needs to realize that, by Proposition \ref{prop-convex-projection}(2), if $\Pi_{C_{x^\ast}}(-s\nabla f (x^\ast)) = \mathbf{0}_{x^\ast}$ for some $s>0$, then
\[
df_{x^\ast}(\mathbf{v}) = \left\langle \nabla f(x^\ast) , \mathbf{v}\right\rangle_{x^\ast} = -\frac{1}{s}\left\langle -s\nabla f(x^\ast) - \Pi_{C_{x^\ast}}(-s\nabla f (x^\ast)), \mathbf{v} - \Pi_{C_{x^\ast}}(-s\nabla f (x^\ast)) \right\rangle_{x^\ast} \geq 0.
\]
for all $\mathbf{v} \in C_{x^\ast}$. Thus, $x^\ast$ is a stationary point of $f|_{C}$. To see that (1) $\Rightarrow$ (3), note that, if $x^\ast$ is a stationary point of $f|_{C}$, then
\[
\left\langle -s\nabla f(x^\ast) - \mathbf{0}_{x^\ast}, \mathbf{v} - \mathbf{0}_{x^\ast} \right\rangle_{x^\ast} =-s df(\mathbf{v}) \leq 0
\]
for all $s >0$ and all $\mathbf{v} \in C_{x^\ast}$. By Proposition \ref{prop-convex-projection}(2) again, this means that $\Pi_{C_{x^\ast}}(-s\nabla f (x^\ast)) = \mathbf{0}_{x^\ast}$ for all $s>0$.
\end{proof}

\begin{definition}\label{def-gradient-related}
For sequences $\{x_k\}_{k=0}^\infty \subset C$ and $\{\mathbf{v}_k \in C_{x_k}\}_{k=0}^\infty$, $\{\mathbf{v}_k\}$ is said to be gradient related to $\{x_k\}$ if and only if
\begin{equation}\label{eq-def-gradient-related-inequality-1}
\left\langle \nabla f(x_{k}), \mathbf{v}_{k} \right\rangle_{x_{k}} \leq 0
\end{equation}
for every $k\geq 0$ and, for every subsequence $\{x_{k_l}\}_{l=0}^\infty$ converging to a nonstationary point of $f|_C$, the corresponding subsequence $\{\mathbf{v}_{k_l}\}_{l=0}^\infty$ satisfies that $\{\|\mathbf{v}_{k_l}\|_{x_{k_l}}\}_{l=0}^\infty$ is bounded and that
\begin{equation}\label{eq-def-gradient-related-inequality-2}
\limsup_{l \rightarrow \infty} \left\langle \nabla f(x_{k_l}), \mathbf{v}_{k_l} \right\rangle_{x_{k_l}} <0.
\end{equation}
\end{definition}

\begin{lemma}\label{lemma-fix-s-gradient-related}
The sequence $\{\Pi_{C_{x_k}}(-s\nabla f (x_k))\}$ is gradient related to $\{x_k\}$ for any sequence $\{x_k\}_{k=0}^\infty \subset C$ and any $s>0$.
\end{lemma}

\begin{proof}
Recall that $\mathbf{0}_{x} \in C_x$ for $x \in C$. By Proposition \ref{prop-convex-projection}(2), we have that
\[
\left\langle -s\nabla f (x_k) - \Pi_{C_{x_k}}(-s\nabla f (x_k)), \mathbf{0}_{x_k} - \Pi_{C_{x_k}}(-s\nabla f (x_k)) \right\rangle_{x_k} \leq 0.
\]
That is,
\begin{equation}\label{eq-fix-s-gradient-related-bound}
\left\langle \nabla f (x_k) , \Pi_{C_{x_k}}(-s\nabla f (x_k)) \right\rangle_{x_k} \leq -\frac{1}{s}  \|\Pi_{C_{x_k}}(-s\nabla f (x_k))\|_{x_k}^2 \leq 0.
\end{equation}
This shows that $\{\Pi_{C_{x_k}}(-s\nabla f (x_k))\}$ satisfies Inequality \eqref{eq-def-gradient-related-inequality-1}.

Now assume that a subsequence $\{x_{k_l}\}_{l=0}^\infty$ is converging to a nonstationary point $\tilde{x}$ of $f|_C$. Since $f$ is $C^1$, we have that $\lim_{l\rightarrow \infty} \nabla f (x_{k_l}) = \nabla f (\tilde{x})$. By the projective continuity of the collection $\{C_x\}_{x \in C}$, we have that $\lim_{l\rightarrow \infty} \Pi_{C_{x_{k_l}}}(-s\nabla f (x_{k_l})) = \Pi_{C_{\tilde{x}}}(-s\nabla f (\tilde{x}))$. Thus, $\{\|\Pi_{C_{x_{k_l}}}(-s\nabla f (x_{k_l}))\|_{x_{k_l}}\}$ is bounded. By Inequality \eqref{eq-fix-s-gradient-related-bound}, we have
\[
\left\langle \nabla f (x_{k_l}), \Pi_{C_{x_{k_l}}}(-s\nabla f (x_{k_l})) \right\rangle_{x_{k_l}} \leq -\frac{1}{s} \|\Pi_{C_{x_{k_l}}}(-s\nabla f (x_{k_l}))\|_{x_{k_l}}^2.
\]
Letting $l \rightarrow \infty$, we get that
\[
\limsup_{l \rightarrow \infty} \left\langle \nabla f(x_{k_l}), \Pi_{C_{x_{k_l}}}(-s\nabla f (x_{k_l})) \right\rangle_{x_{k_l}} \leq -\frac{1}{s} \|\Pi_{C_{\tilde{x}}}(-s\nabla f (\tilde{x}))\|_{\tilde{x}}^2.
\]
Since $\tilde{x}$ is a nonstationary point of $f|_C$, we know that $\Pi_{C_{\tilde{x}}}(-s\nabla f (\tilde{x})) \neq \mathbf{0}_{\tilde{x}}$ by Lemma \ref{lemma-stationary-projection}. This shows that $\limsup_{l \rightarrow \infty} \left\langle \nabla f(x_{k_l}), \Pi_{C_{x_{k_l}}}(-s\nabla f (x_{k_l})) \right\rangle_{x_{k_l}} < 0$. So $\{\Pi_{C_{x_k}}(-s\nabla f (x_k))\}$ satisfies Inequality \eqref{eq-def-gradient-related-inequality-2}.
\end{proof}

\begin{lemma}\label{lemma-convergence-gradient-related-Armijo}
Fix $\beta, \sigma \in (0,1)$. For sequences $\{x_k\}_{k=0}^\infty \subset C$ and $\{\mathbf{v}_k \in C_{x_k}\}_{k=0}^\infty$, assume that $\{\mathbf{v}_k\}$ is gradient related to $\{x_k\}$ and that
\begin{equation}\label{eq-gradient-related-update}
x_{k+1} = R_{x_k}(\alpha_k \mathbf{v}_k),
\end{equation}
where
$\alpha_k = \beta^{m_k}$, where $m_k$ is the smallest non-negative integer $m$ satisfying
\begin{equation}\label{eq-Armijo-feasible-inequality-gradient-related}
f(x_k) - f(R_{x_k}(\beta^m \mathbf{v}_k)) \geq -\sigma \beta^m \left\langle \nabla f(x_k), \mathbf{v}_k\right\rangle_{x_k}.
\end{equation}
Then every limit point of $\{x_k\}$ is a stationary point of $f|_C$.
\end{lemma}

\begin{proof}
Assume on the contrary that there is a subsequence  $\{x_{k_l}\}_{l=0}^\infty$ of $\{x_k\}$ converging to a nonstationary point $\bar{x}$ of $f|_C$. By Inequalities \eqref{eq-def-gradient-related-inequality-1}, \eqref{eq-Armijo-feasible-inequality-gradient-related} and Equation \eqref{eq-gradient-related-update}, $\{f(x_k)\}$ is a decreasing sequence. So $\lim_{k\rightarrow \infty} f(x_k) =f(\bar{x})$ and, consequently,  $\lim_{k\rightarrow \infty} (f(x_k) - f(x_{k+1})) =0$. Using Inequalities \eqref{eq-def-gradient-related-inequality-1} and \eqref{eq-Armijo-feasible-inequality-gradient-related} again, one gets that
\[
0 \geq \limsup_{l \rightarrow \infty} \alpha_{k_l}\left\langle \nabla f(x_{k_l}), \mathbf{v}_{k_l}\right\rangle_{x_{k_l}}\geq \liminf_{l \rightarrow \infty} \alpha_{k_l}\left\langle \nabla f(x_{k_l}), \mathbf{v}_{k_l}\right\rangle_{x_{k_l}} \geq \frac{1}{\sigma}\lim_{l \rightarrow \infty}(f(x_{k_l +1}) - f(x_{k_l})) =0.
\]
Thus, $\lim_{l \rightarrow \infty} \alpha_{k_l}\left\langle \nabla f(x_{k_l}), \mathbf{v}_{k_l}\right\rangle_{x_{k_l}}=0$. By Inequality \eqref{eq-def-gradient-related-inequality-2}, we have that $\limsup_{l \rightarrow \infty} \left\langle \nabla f(x_{k_l}), \mathbf{v}_{k_l} \right\rangle_{x_{k_l}} <0$. This implies that $\lim_{l \rightarrow \infty} \alpha_{k_l} =0$. Therefore, there is some $m>0$ such that, for $l \geq m$,
\begin{itemize}
	\item $\left\langle \nabla f(x_{k_l}), \mathbf{v}_{k_l} \right\rangle_{x_{k_l}}< \frac{1}{2}\limsup_{l \rightarrow \infty} \left\langle \nabla f(x_{k_l}), \mathbf{v}_{k_l} \right\rangle_{x_{k_l}}<0$,
	\item $m_{k_l} \geq 1$.
\end{itemize}
It follows that
\begin{itemize}
	\item there is an $r >0$ such that $\|\mathbf{v}_{k_l}\|_{x_{k_l}} \geq r$ for $l \geq m$,
	\item $m = m_{k_l} - 1$ does not satisfy Inequality \eqref{eq-Armijo-feasible-inequality-gradient-related} if $l \geq m$.
\end{itemize}
Thus,
\begin{equation}\label{eq-Armijo-feasible-inequality-gradient-related-not}
f(x_{k_l}) - f(R_{x_{k_l}}(\frac{\alpha_{k_l}}{\beta} \mathbf{v}_{k_l})) < -\sigma \frac{\alpha_{k_l}}{\beta}  \left\langle \nabla f(x_{k_l}), \mathbf{v}_{k_l}\right\rangle_{x_{k_l}} \text{ for } l \geq m.
\end{equation}
Set $\mathbf{u}_{k_l} = \frac{r\mathbf{v}_{k_l}}{\|\mathbf{v}_{k_l}\|_{x_{k_l}}} = (1-\frac{r}{\|\mathbf{v}_{k_l}\|_{x_{k_l}}}) \mathbf{0}_{x_{k_l}} + \frac{r}{\|\mathbf{v}_{k_l}\|_{x_{k_l}}} \mathbf{v}_{k_l} \in C_{x_{k_l}}$ and $\bar{\alpha}_{k_l} = \frac{\alpha_{{k_l}}\|\mathbf{v}_{k_l}\|_{x_{k_l}}}{r\beta}>0$ for $l \geq m$. Note that $\{\|\mathbf{v}_{k_l}\|_{x_{k_l}}\}$ is bounded since $\{\mathbf{v}_k\}$ is gradient related to $\{x_k\}$. Hence, $\lim_{l \rightarrow \infty} \bar{\alpha}_{k_l} =0$. Since $\|\mathbf{u}_{k_l}\|_{x_{k_l}} =r \leq \|\mathbf{v}_{k_l}\|_{x_{k_l}}$ for $l\geq m$, we have that the sequence $\{\mathbf{u}_{k_l}\}$ contains a convergent subsequence. That is, there is a subsequence $\{\mathbf{u}_{k_{l_j}}\}_{j=0}^\infty$ of $\{\mathbf{u}_{k_l}\}$ and a $\bar{\mathbf{u}} \in T_{\bar{x}} M$ such that $\lim_{j \rightarrow \infty} \mathbf{u}_{k_{l_j}} = \bar{\mathbf{u}}$. Note that $\|\bar{\mathbf{u}}\|_{\bar{x}} =r$. Moreover, by the projective continuity of $\{C_x\}$, $\Pi_{C_{\bar{x}}}(\bar{\mathbf{u}}) = \lim_{j \rightarrow \infty} \Pi_{C_{x_{k_{l_j}}}}(\mathbf{u}_{k_{l_j}}) = \lim_{j \rightarrow \infty} \mathbf{u}_{k_{l_j}} = \bar{\mathbf{u}}$. So $\bar{\mathbf{u}} \in C_{\bar{x}}$. Inequality \eqref{eq-Armijo-feasible-inequality-gradient-related-not} implies that
\[
\frac{f(x_{k_{l_j}}) - f(R_{x_{k_{l_j}}}(\bar{\alpha}_{k_{l_j}} \mathbf{u}_{k_{l_j}}))}{\bar{\alpha}_{k_{l_j}} } < -\sigma \left\langle \nabla f(x_{k_{l_j}}), \mathbf{u}_{k_{l_j}}\right\rangle_{x_{k_{l_j}}} \text{ for } l_j \geq m.
\]
By the Mean Value Theorem, there is an $\tilde{\alpha}_{k_{l_j}}\in [0,\bar{\alpha}_{k_{l_j}}]$ such that
\[
 -\left\langle \nabla (f\circ R_{x_{k_{l_j}}})(\tilde{\alpha}_{k_{l_j}} \mathbf{u}_{k_{l_j}}), \mathbf{u}_{k_{l_j}}\right\rangle_{x_{k_{l_j}}} < -\sigma \left\langle \nabla f(x_{k_{l_j}}), \mathbf{u}_{k_{l_j}}\right\rangle_{x_{k_{l_j}}} \text{ for } l_j \geq m.
\]
But $\lim_{j \rightarrow \infty} \bar{\alpha}_{k_{l_j}} =0$. So $\lim_{j \rightarrow \infty} \tilde{\alpha}_{k_{l_j}} =0$. Letting $j\rightarrow \infty$ in the above inequality and using Equation \eqref{eq-gradient-coincide}, we get that $ -\left\langle \nabla f(\bar{x}), \bar{\mathbf{u}}\right\rangle_{\bar{x}} \leq -\sigma \left\langle \nabla f(\bar{x}), \bar{\mathbf{u}}\right\rangle_{\bar{x}}$. That is, $(1-\sigma) \left\langle \nabla f(\bar{x}), \bar{\mathbf{u}}\right\rangle_{\bar{x}} \geq 0$. This shows that  $\left\langle \nabla f(\bar{x}), \bar{\mathbf{u}}\right\rangle_{\bar{x}} \geq 0$. On the other hand,
\begin{eqnarray*}
\left\langle \nabla f(\bar{x}), \bar{\mathbf{u}}\right\rangle_{\bar{x}} & = & \lim_{j \rightarrow \infty} \left\langle \nabla f(x_{k_{l_j}}), \mathbf{u}_{k_{l_j}}\right\rangle_{x_{k_{l_j}}} = \lim_{j \rightarrow \infty} \frac{r}{\|\mathbf{v}_{k_{l_j}}\|_{x_{k_{l_j}}}} \left\langle \nabla f(x_{k_{l_j}}), \mathbf{v}_{k_{l_j}}\right\rangle_{x_{k_{l_j}}} \\
& \leq &  r \cdot \frac{\limsup_{j \rightarrow \infty}\left\langle \nabla f(x_{k_{l_j}}), \mathbf{v}_{k_{l_j}}\right\rangle_{x_{k_{l_j}}}}{\limsup_{j \rightarrow \infty}\|\mathbf{v}_{k_{l_j}}\|_{x_{k_{l_j}}}} <0.
\end{eqnarray*}
This is a contradiction.
\end{proof}

\begin{lemma}\label{lemma-convergence-gradient-related-limited-mini}
For sequences $\{x_k\}_{k=0}^\infty \subset C$ and $\{\mathbf{v}_k \in C_{x_k}\}_{k=0}^\infty$, assume that $\{\mathbf{v}_k\}$ is gradient related to $\{x_k\}$ and that
\begin{equation}\label{eq-gradient-related-update-limited-mini}
x_{k+1} = R_{x_k}(\tilde{\alpha}_k \mathbf{v}_k),
\end{equation}
where $\tilde{\alpha}_k \in [0,1]$ satisfies $f(R_{x_k}(\tilde{\alpha}_k \mathbf{v}_k)) = \min_{\alpha \in [0,1]} f(R_{x_k}(\alpha \mathbf{v}_k))$
Then every limit point of $\{x_k\}$ is a stationary point of $f|_C$.
\end{lemma}

\begin{proof}
Assume on the contrary that there is a subsequence  $\{x_{k_l}\}_{l=0}^\infty$ of $\{x_k\}$ converging to a nonstationary point $\bar{x}$ of $f|_C$. Then $\{f(x_{k+1})\}$ is a decreasing sequence converging to $f(\bar{x})$. Now fix $\beta, \sigma \in (0,1)$ and let $\alpha_k = \beta^{m_k}$, where $m_k$ is the smallest non-negative integer $m$ satisfying Inequality \eqref{eq-Armijo-feasible-inequality-gradient-related}. Observe that
\[
f(x_k)-f(x_{k+1}) \geq f(x_k) - f(R_{x_k}(\alpha_k  \mathbf{v}_k)) \geq -\sigma \alpha_k \left\langle \nabla f(x_k), \mathbf{v}_k\right\rangle_{x_k}.
\]
From here on, one simply repeats exactly the same arguments in the proof of Lemma \ref{lemma-convergence-gradient-related-Armijo} for the sequences $\{x_k\}$, $\{\mathbf{v}_k\}$ and $\{\alpha_k\}$. This leads to the same contradiction that there is a vector $\bar{\mathbf{u}} \in C_{\bar{x}}$ such that $\left\langle \nabla f(\bar{x}), \bar{\mathbf{u}}\right\rangle_{\bar{x}} <0 \leq \left\langle \nabla f(\bar{x}), \bar{\mathbf{u}}\right\rangle_{\bar{x}}$.
\end{proof}

\begin{proof}[Proof for Proposition \ref{prop-rgpa-convergence} for Stepsize Rules \ref{ur-limited-mini} and \ref{ur-Armijo-feasible}]
Proposition \ref{prop-rgpa-convergence} for these two stepsize rules follows from Lemmas \ref{lemma-fix-s-gradient-related}, \ref{lemma-convergence-gradient-related-Armijo} and \ref{lemma-convergence-gradient-related-limited-mini}.
\end{proof}

\subsection{The constant stepsize rule}

\begin{lemma}\label{lemma-R-Lipschitz-quadratic-bound}
Assume that $\nabla f$ is $R$-Lipschitz over $C$, that is, $f$ satisfies Inequality \eqref{eq-def-R-Lipschitz}. Then
\begin{equation}\label{eq-R-Lipschitz-quadratic-bound}
f(R_x(\mathbf{v})) \leq f(x) + \left\langle \nabla f(x), \mathbf{v} \right\rangle_x + \frac{L}{2}\|\mathbf{v}\|_x^2
\end{equation}
for every $x \in C$ and $\mathbf{v} \in C_x$.
\end{lemma}

\begin{proof}
Recall that $C_x$ is convex and $\mathbf{0}_x \in C_x$. So $t\mathbf{v} \in C_x$ for $t\in [0,1]$. Set $p(t)=f(R_x(t\mathbf{v}))$ for $t\in [0,1]$. Then $p'(t)=\left\langle  \nabla (f\circ R_x)(t\mathbf{v}), \mathbf{v}\right\rangle_x$. And, by Inequality \eqref{eq-def-R-Lipschitz} and Equation \eqref{eq-gradient-coincide},
\begin{eqnarray*}
f(R_x(\mathbf{v})) - f(x) & = & p(1) -p(0) = \int_0^1 p'(t) dt = \int_0^1 \left\langle  \nabla (f\circ R_x)(t\mathbf{v}), \mathbf{v}\right\rangle_x dt \\
& = & \int_0^1 \left\langle  \nabla f(x)+\nabla (f\circ R_x)(t\mathbf{v})-\nabla f(x), \mathbf{v}\right\rangle_x dt \\
& = &  \left\langle  \nabla f(x), \mathbf{v}\right\rangle_x + \int_0^1 \left\langle  \nabla (f\circ R_x)(t\mathbf{v})-\nabla f(x), \mathbf{v}\right\rangle_x dt \\
& \leq & \left\langle  \nabla f(x), \mathbf{v}\right\rangle_x + \int_0^1 \|\nabla (f\circ R_x)(t\mathbf{v})-\nabla f(x)\|_x \|\mathbf{v}\|_x dt \\
& \leq & \left\langle  \nabla f(x), \mathbf{v}\right\rangle_x + \int_0^1 Lt\|\mathbf{v}\|_x^2 dt  = \left\langle  \nabla f(x), \mathbf{v}\right\rangle_x + \frac{L}{2}\|\mathbf{v}\|_x^2
\end{eqnarray*}
\end{proof}

\begin{proof}[Proof for Proposition \ref{prop-rgpa-convergence-constant-stepsize}]
By Lemma \ref{lemma-R-Lipschitz-quadratic-bound},
\[
f(x_{k+1}) - f(x_k) \leq \left\langle \nabla f(x_k), \Pi_{C_{x_k}}(-s_k\nabla f (x_k)) \right\rangle_{x_k} + \frac{L}{2}\|\Pi_{C_{x_k}}(-s\nabla f (x_k))\|_{x_k}^2.
\]
Comparing this to Inequality \eqref{eq-fix-s-gradient-related-bound}, one gets that
\begin{equation}\label{eq-constant-stepsize-decreasing}
f(x_{k+1}) - f(x_k) \leq  (\frac{L}{2}-\frac{1}{s}) \|\Pi_{C_{x_k}}(-s\nabla f (x_k))\|_{x_k}^2\leq 0,
\end{equation}
where $\frac{L}{2}-\frac{1}{s} <0$ since $s<\frac{2}{L}$. Thus, $\{f(x_k)\}$ is a decreasing sequence. If $\{x_k\}$ has a subsequence $\{x_{k_l}\}$ converging to some $\bar{x} \in C$, then $\lim_{k\rightarrow \infty} f(x_k) = f(\bar{x})$ and  $\lim_{k\rightarrow \infty} (f(x_{k+1}) - f(x_k)) = 0$. Using Inequality \eqref{eq-constant-stepsize-decreasing} again, we get by the Squeeze Theorem that $\lim_{k\rightarrow \infty} \|\Pi_{C_{x_k}}(-s\nabla f (x_k))\|_{x_k} =0$. By the projective continuity of $\{C_x\}$, we have $\|\Pi_{C_{\bar{x}}}(-s\nabla f (\bar{x}))\|_{\bar{x}} = \lim_{l\rightarrow \infty} \|\Pi_{C_{x_{k_l}}}(-s\nabla f (x_{k_l}))\|_{x_{k_l}} =0$. This shows that $\Pi_{C_{\bar{x}}}(-s\nabla f (\bar{x})) =\mathbf{0}_{\bar{x}}$. By Lemma \ref{lemma-stationary-projection}, $\bar{x}$ is a stationary point of $f|_C$.
\end{proof}

\subsection{The Armijo rule along the projection arc}

\begin{lemma}\label{lemma-nonincreasing-auxilary-function}
For every $x \in C$ and $\mathbf{z} \in T_x M$, the function $g_{x,\mathbf{z}}: (0, \infty) \rightarrow \R$ defined by
\begin{equation}\label{eq-g-x-z-func}
g_{x,\mathbf{z}}(s) = \frac{\|\Pi_{C_x}(s\mathbf{z})\|_x}{s}
\end{equation}
for all $s > 0$, is decreasing.
\end{lemma}

\begin{proof}
Apply \cite[Lemma 2.3.1]{Bertsekas-nonlinear-programming} to the tangent space $T_x M$, the convex subset $C_x \subset T_x M$ and the vector $\mathbf{0}_x$.
\end{proof}

In the rest of this section, when applying Lemma \ref{lemma-nonincreasing-auxilary-function}, we always take $\mathbf{z} = -\nabla f(x)$. And the function $g_{x, -\nabla f(x)}: (0, \infty) \rightarrow \R$ defined by
\begin{equation}\label{eq-g-x-grad-f-func}
g_{x, -\nabla f(x)}(s) = \frac{\|\Pi_{C_x}(-s\nabla f(x))\|_x}{s}
\end{equation}
for all $s > 0$, is decreasing.

\begin{lemma}\label{lemma-existence-armijo-step-along-projection-arc}
For every $x \in C$ and $\sigma \in (0, 1)$, there exists a scalar $s_x \in (0, 1]$ such that
\begin{equation}\label{eq-armijo-step-existence}
f(x) - f(R_x(\Pi_{C_x}(-s\nabla f(x))) \geq \sigma\langle \nabla f(x), -\Pi_{C_x}(-s\nabla f(x)) \rangle_x
\end{equation}
for all $s \in [0, s_x]$.
\end{lemma}

\begin{proof}
By Lemma \ref{lemma-stationary-projection}, if $x$ is a stationary point of $f$, for all $s> 0$,
\[
	 \Pi_{C_x}(-s\nabla f(x)) = \mathbf{0}_{x} \text{ and } f(R_x(\Pi_{C_x}(-s\nabla f(x)))) = f(R_x(\mathbf{0}_{x})) = f(x).
\]

In this case, we have
\[
f(x) - f(R_x(\Pi_{C_x}(-s\nabla f(x)))) = \sigma\langle \nabla f(x), -\Pi_{C_x}(-s\nabla f(x)) \rangle_x = 0.
\]
for all $s > 0$.

If $x$ is non-stationary, then $\Pi_{C_x}(-s\nabla f(x)) \neq \mathbf{0}_{x}$ for all $s > 0$ by Lemma \ref{lemma-stationary-projection}. By Proposition \ref{prop-convex-projection}(2), for all $s > 0$,
\[
\langle -s\nabla f(x) - \Pi_{C_x}(-s\nabla f(x)), \mathbf{0}_{x} -\Pi_{C_x}(-s\nabla f(x))  \rangle_x \leq 0.
\]
That is,
\begin{equation}\label{eq-lower-bound-armijo-step-projection-arc}
\langle \nabla f(x), - \Pi_{C_x}(-s\nabla f(x)) \rangle_x \geq \frac{\|\Pi_{C_x}(-s\nabla f(x))\|_x^2}{s}.
\end{equation}
By the Mean Value Theorem and Equation \eqref{eq-gradient-coincide}, for all $s > 0$,
\begin{align*}
&f(x) - f(R_x(\Pi_{C_x}(-s\nabla f(x))) \\
&= (f \circ R_x)(\mathbf{0}_{x}) - (f \circ R_x)(\Pi_{C_x}(-s\nabla f(x))) \\
&= \langle \nabla(f \circ R_x)(\xi_s), -\Pi_{C_x}(-s\nabla f(x)) \rangle_x \\
&= \langle \nabla f(x), -\Pi_{C_x}(-s\nabla f(x)) \rangle_x
+ \langle \nabla(f \circ R_x)(\xi_s) - \nabla(f \circ R_x)(\mathbf{0}_x), -\Pi_{C_x}(-s\nabla f(x)) \rangle_x
\end{align*}
where $\xi_s \in T_x M$ is on the line segment joining $\mathbf{0}_x$ and $\Pi_{C_x}(-s\nabla f(x))$. Since $C_x$ is convex, $\xi_s \in C_x$. Inequality \eqref{eq-armijo-step-existence} is equivalent to
\begin{equation}\label{eq-armijo-step-existence-reformulated}
(1 - \sigma)\langle \nabla f(x), -\Pi_{C_x}(-s\nabla f(x)) \rangle_x \geq \langle \nabla(f \circ R_x)(\mathbf{0}_x) - \nabla(f \circ R_x)(\xi_s), -\Pi_{C_x}(-s\nabla f(x)) \rangle_x.
\end{equation}
From Inequality \eqref{eq-lower-bound-armijo-step-projection-arc} and Lemma \ref{lemma-nonincreasing-auxilary-function}, for all $s \in (0, 1]$,
\[
\langle \nabla f(x), - \Pi_{C_x}(-s\nabla f(x)) \rangle_x \geq \frac{\|\Pi_{C_x}(-s\nabla f(x))\|_x^2}{s}
\geq \|\Pi_{C_x}(-\nabla f(x))\|_x \cdot \|\Pi_{C_x}(-s\nabla f(x))\|_x.
\]
Therefore, Inequality \eqref{eq-armijo-step-existence-reformulated} is true for all $s \in (0, 1]$ satisfying
\begin{equation}\label{eq-armijo-step-existence-reformulated-again}
(1 - \sigma)\|\Pi_{C_x}(-\nabla f(x))\|_x \geq \Big\langle \nabla(f \circ R_x)(\mathbf{0}_x) - \nabla(f \circ R_x)(\xi_s), \frac{-\Pi_{C_x}(-s\nabla f(x))}{\|\Pi_{C_x}(-s\nabla f(x))\|_x} \Big\rangle_x.
\end{equation}
Note that $f: M \rightarrow \R$ and $R_x: T_x M \rightarrow M$ are $C^1$, and
the right-hand side of Equation \eqref{eq-armijo-step-existence-reformulated-again} converges to $0$ as $s \rightarrow 0$.
Thus, there exists an $s_x \in (0, 1]$
such that Inequality \eqref{eq-armijo-step-existence-reformulated-again} is true for all $s \in (0, s_x]$.
Therefore, for such $s_x > 0$, Inequalities \eqref{eq-armijo-step-existence} and \eqref{eq-armijo-step-existence-reformulated} are also true.
\end{proof}

\begin{lemma}\label{lemma-stationarity-projection-arc}
Given a sequence $\{x_k\}_{k=0}^\infty$ generated by Algorithm \ref{algorithm-rgpa} with Stepsize Rule \ref{ur-Armijo-projection-arc}, if $x_k$ is a stationary point of $f|_C$ for some $k$, then $m_k = 0$ and $s_k = \bar{s}$, where $\bar{s} > 0$ is given in Stepsize Rule \ref{ur-Armijo-projection-arc}.
\end{lemma}

\begin{proof}
By Lemma \ref{lemma-stationary-projection}, if $x_k$ is a stationary point of $f|_C$ for some $k$, then $\Pi_{C_{x_k}}(-\bar{s}\nabla f (x_k)) = \mathbf{0}_{x_k}$. And Inequality \eqref{eq-Armijo-projection-arc-inequality} holds with $m_k = 0$. That is, $s_k = \beta^{m_k}\bar{s} = \bar{s}$.
\end{proof}

\begin{proof}[Proof for Proposition \ref{prop-rgpa-convergence} for Stepsize Rule \ref{ur-Armijo-projection-arc}]
Lemma \ref{lemma-existence-armijo-step-along-projection-arc} shows that the $s_k$ given by Stepsize Rule \ref{ur-Armijo-projection-arc} is well defined as a positive number for all $k$.
Let $\bar{x}$ be a limit point of a sequence $\{x_k\}$ generated by Algorithm \ref{algorithm-rgpa} with Stepsize Rule \ref{ur-Armijo-projection-arc}.
There is a subsequence $\{x_{k_l}\}_{l=0}^{\infty}$ of $\{x_k\}$ converging to $\bar{x}$.
By Inequality \eqref{eq-armijo-step-existence} and Stepsize Rule \ref{ur-Armijo-projection-arc}, $\{f(x_k)\}$ is a decreasing sequence.
So $\lim_{k \rightarrow \infty} f(x_k) = f(\bar{x})$ and, consequently, $\lim_{k \rightarrow \infty} (f(x_k) - f(x_{k+1})) = 0$. Next, we consider two cases.

\textbf{Case 1:} $\liminf_{l \rightarrow \infty} s_{k_l} > \hat{s}$ for some $\hat{s} > 0$.
There exist some $L$ such that $s_{k_l} \geq \hat{s}$ for all $l > L$.
Recall that $\bar{s} > 0$ is given in Stepsize Rule \ref{ur-Armijo-projection-arc}.
Since $s_{k_l} \leq \bar{s}$, by Inequality \eqref{eq-lower-bound-armijo-step-projection-arc} and Lemma \ref{lemma-nonincreasing-auxilary-function}, we have:
\begin{align*}
f(x_{k_l}) - f(x_{k_l + 1}) &\geq \sigma\langle \nabla f(x_{k_l}), -\Pi_{C_{x_{k_l}}}(-s_{k_l}\nabla f(x_{k_l})) \rangle_{x_{k_l}} \\
&\geq \sigma\frac{\|\Pi_{C_{x_{k_l}}}(-s_{k_l}\nabla f(x_{k_l}))\|_{x_{k_l}}^2}{s_{k_l}} \\
&\geq \sigma s_{k_l}\frac{\|\Pi_{C_{x_{k_l}}}(-s_{k_l}\nabla f(x_{k_l}))\|_{x_{k_l}}^2}{s_{k_l}^2} \\
&\geq \sigma \hat{s}\frac{\|\Pi_{C_{x_{k_l}}}(-\bar{s}\nabla f(x_{k_l}))\|_{x_{k_l}}^2}{\bar{s}^2}.
\end{align*}
Since $f$ is $C^1$, we have that $\lim_{l \rightarrow \infty} \nabla f(x_{k_l}) = \nabla f(\bar{x})$. By the projective continuity of the $R$-convex bundle $\{C_x\}_{x \in C}$ in Definition \ref{def-R-convex}, we have
\[
\lim_{l \rightarrow \infty}\|\Pi_{C_{x_{k_l}}}(-\bar{s}\nabla f(x_{k_l}))\|_{x_{k_l}} = \|\Pi_{C_{\bar{x}}}(-\bar{s}\nabla f(\bar{x}))\|_{\bar{x}}.
\]
Taking limit as $l \rightarrow \infty$ and we get
\[
0 \geq \sigma \hat{s}\frac{\|\Pi_{C_{\bar{x}}}(-\bar{s}\nabla f(\bar{x}))\|_{\bar{x}}^2}{\bar{s}^2}.
\]
And we get $\Pi_{C_{\bar{x}}}(-\bar{s}\nabla f(\bar{x})) = \mathbf{0}_{\bar{x}}$. By Lemma \ref{lemma-stationary-projection}, $\bar{x}$ is a stationary point of $f|_C$.

\textbf{Case 2:} $\liminf_{l \rightarrow \infty} s_{k_l} = 0$. There exists a subsequence $\{s_{k_{l_j}}\}_{j = 0}^\infty$ of $\{s_{k_l}\}$ such that $\lim_{j \rightarrow \infty} s_{k_{l_j}} = 0$.
Hence, there exists some integer $J$ such that $m_{k_{l_j}} \geq 1$ for all $j \geq J$ where $m_{k_{l_j}}$ is given in Stepsize Rule \ref{ur-Armijo-projection-arc}. Therefore,
\begin{equation}\label{eq-armijo-rule-fail}
f(x_{k_{l_j}}) - f(R_{x_{k_{l_j}}}(\Pi_{C_{x_{k_{l_j}}}}(\beta^{-1}s_{k_{l_j}}\nabla f(x_{k_{l_j}}))))
< \sigma \langle \nabla f(x_{k_{l_j}}), -\Pi_{C_{x_{k_{l_j}}}}(-\beta^{-1}s_{k_{l_j}}\nabla f(x_{k_{l_j}}))\rangle_{x_{k_{l_j}}}
\end{equation}
for all $j > J$. By Lemma \ref{lemma-stationarity-projection-arc}, $x_{k_{l_j}}$ is non-stationary for $j > J$. By Lemma \ref{lemma-stationary-projection}, we have
\[
\|\Pi_{C_{k_{l_j}}}(\beta^{-1}s_{k_{l_j}}\nabla f(x_{k_{l_j}}))\|_{x_{k_{l_j}}} > 0
\]
for all $j > J$. By the Mean Value Theorem, for all $j > J$,
\begin{align*}
&f(x_{k_{l_j}}) - f(R_{x_{k_{l_j}}}(\Pi_{C_{x_{k_{l_j}}}}(-\beta^{-1}s_{k_{l_j}}\nabla f(x_{k_{l_j}}))) \\
= \ &(f \circ R_{x_{k_{l_j}}})(\mathbf{0}_{x_{k_{l_j}}}) - (f \circ R_{x_{k_{l_j}}})(\Pi_{C_{x_{k_{l_j}}}}(-\beta^{-1}s_{k_{l_j}}\nabla f(x_{k_{l_j}}))) \\
= \ &\langle \nabla(f \circ R_{x_{k_{l_j}}})(\xi_j), -\Pi_{C_{x_{k_{l_j}}}}(-\beta^{-1}s_{k_{l_j}}\nabla f(x_{k_{l_j}})) \rangle_{x_{k_{l_j}}} \\
= \ &\langle \nabla f(x_{k_{l_j}}), -\Pi_{C_{x_{k_{l_j}}}}(-s_{k_{l_j}}\nabla f(x_{k_{l_j}})) \rangle_{x_{k_{l_j}}} \\
+ \ &\langle \nabla(f \circ R_{x_{k_{l_j}}})(\xi_j) - \nabla(f \circ R_{x_{k_{l_j}}})(\mathbf{0}_{x_{k_{l_j}}}), -\Pi_{C_{x_{k_{l_j}}}}(-\beta^{-1}s_{k_{l_j}}\nabla f(x_{k_{l_j}})) \rangle_{x_{k_{l_j}}}
\end{align*}
where $\xi_j$ is on the line segment joining $\mathbf{0}_{x_{k_{l_j}}}$ and $\Pi_{C_{x_{k_{l_j}}}}(-\beta^{-1}s_{k_{l_j}}\nabla f(x_{k_{l_j}}))$. Therefore, Inequality \eqref{eq-armijo-rule-fail} is equivalent to
\begin{align*}
& (1-\sigma)\langle \nabla f(x_{k_{l_j}}),  -\Pi_{C_{x_{k_{l_j}}}}(-\beta^{-1}s_{k_{l_j}}\nabla f(x_{k_{l_j}})) \rangle_{x_{k_{l_j}}} \\
& < \langle  \nabla(f \circ R_{x_{k_{l_j}}})(\xi_j) - \nabla(f \circ R_{x_{k_{l_j}}})(\mathbf{0}_{x_{k_{l_j}}}), \Pi_{C_{x_{k_{l_j}}}}(-\beta^{-1}s_{k_{l_j}}\nabla f(x_{k_{l_j}})) \rangle_{x_{k_{l_j}}}
\end{align*}
for all $j > J$. By Inequality \eqref{eq-lower-bound-armijo-step-projection-arc} and Lemma \ref{lemma-nonincreasing-auxilary-function}, we get
\begin{align*}
\langle \nabla f(x_{k_{l_j}}), -\Pi_{C_{x_{k_{l_j}}}}(-\beta^{-1}s_{k_{l_j}}\nabla f(x_{k_{l_j}})) \rangle_{x_{k_{l_j}}} & \geq
\frac{\|\Pi_{C_{x_{k_{l_j}}}}(-\beta^{-1}s_{k_{l_j}}\nabla f(x_{k_{l_j}}))\|_{x_{k_{l_j}}}^2}{\beta^{-1}s_{k_{l_j}}} \\
& \geq \frac{1}{\bar{s}}\|\Pi_{C_{x_{k_{l_j}}}}(-\bar{s}\nabla f(x_{k_{l_j}}))\|_{x_{k_{l_j}}} \|\Pi_{C_{x_{k_{l_j}}}}(-\beta^{-1}s_{k_{l_j}}\nabla f(x_{k_{l_j}}))\|_{x_{k_{l_j}}}\\
\end{align*}
for all $j > J$, where $\bar{s}$ is given in Stepsize Rule \ref{ur-Armijo-projection-arc}. Combining the two inequalities above, we get
\begin{align*}
&\frac{1 - \sigma}{\bar{s}}\|\Pi_{C_{x_{k_{l_j}}}}(-\bar{s}\nabla f(x_{k_{l_j}}))\|_{x_{k_{l_j}}} \|\Pi_{C_{x_{k_{l_j}}}}(-\beta^{-1}s_{k_{l_j}}\nabla f(x_{k_{l_j}}))\|_{x_{k_{l_j}}} \\
& < \langle \nabla(f \circ R_{x_{k_{l_j}}})(\xi_j) - \nabla(f \circ R_{x_{k_{l_j}}})(\mathbf{0}_{x_{k_{l_j}}}), \Pi_{C_{x_{k_{l_j}}}}(-\beta^{-1}s_{k_{l_j}}\nabla f(x_{k_{l_j}})) \rangle_{x_{k_{l_j}}} \\
& \leq \|\nabla(f \circ R_{x_{k_{l_j}}})(\xi_j) - \nabla(f \circ R_{x_{k_{l_j}}})(\mathbf{0}_{x_{k_{l_j}}})\|_{x_{k_{l_j}}} \|\Pi_{C_{x_{k_{l_j}}}}(-\beta^{-1}s_{k_{l_j}}\nabla f(x_{k_{l_j}})) \|_{x_{k_{l_j}}}
\end{align*}
for all $j > J$. That is,
\begin{equation}\label{eq-case-2-upper-bound}
\frac{1 - \sigma}{\bar{s}}\|\Pi_{C_{x_{k_{l_j}}}}(-\bar{s}\nabla f(x_{k_{l_j}}))\|_{x_{k_{l_j}}} < \|\nabla(f \circ R_{x_{k_{l_j}}})(\xi_j) - \nabla(f \circ R_{x_{k_{l_j}}})(\mathbf{0}_{x_{k_{l_j}}})\|_{x_{k_{l_j}}}
\end{equation}
for all $j > J$. Since $\lim_{j \rightarrow \infty} x_{k_{l_j}} = \bar{x}$ and $\lim_{j \rightarrow \infty} s_{k_{l_j}} = 0$, we have that $\xi_j \rightarrow \mathbf{0}_{\bar{x}}$ as $j \rightarrow \infty$.
By the projective continuity of the $R$-convexity bundle $\{C_x\}_{x \in C}$ in Definition \ref{def-R-convex},
\[
\Pi_{C_{\bar{x}}}(-\bar{s}\nabla f(\bar{x})) = \lim_{j \rightarrow \infty}\Pi_{C_{x_{k_{l_j}}}}(-\beta^{-1}s_{k_{l_j}}\nabla f(x_{k_{l_j}})) = \mathbf{0}_{\bar{x}}.
\]
By Lemma \ref{lemma-stationary-projection}, $\bar{x}$ is a stationary point of $f|_C$.
\end{proof}


\section{Supplemental Results for Section \ref{sec-wlra-constraint}}\label{sec-rgp}
Recall that we assume in Section \ref{sec-wlra-constraint} that
\begin{enumerate}
	\item $m$, $n$ and $p$ are fixed positive integers with the low-rank $p$ satisfying $p \leq \min\{m,n\}$;
	\item $\R^{m\times n}$ is the space of $m\times n$ matrices;
	\item $\St(n,p)$ is the $n\times p$ Stiefel manifold;
	\item $A=[a_{i,j}]$ in $\R^{m\times n}$ is a given matrix of constants, that is, the matrix of data to be approximated in WLRA;
	\item $W=[w_{i,j}]$ in $\R^{m\times n}$ is a given matrix of weights in WLRA satisfying $w_{i,j} \geq 0$ for $(i,j) \in \{1,2,\dots,m\}\times \{1,2,\dots,n\}$ and $\sum_{i=1}^m\sum_{j=1}^n w_{i,j} > 0$;
	\item the regularization parameter $\lambda > 0$ is a fixed positive real value.
\end{enumerate}

In subsections \ref{sec-orthogonal-proj-prod-mfd} - \ref{sec-retractions-prod-mfd}, we recall the geometry on the product manifold $\M$ given by Equation \eqref{eq-def-prod-mfd},
which also appeared in \cite{xu-yang-wu-2025-wlra-sgd-manifolds}.

\subsection{Orthogonal Projection onto $M$}\label{sec-orthogonal-proj-prod-mfd}
\begin{lemma}\cite[Equation (3.37)]{AMS}\label{lemma-gradient-submfd}
Let $\overline{M}$ be a Riemannian manifold and $f: \overline{M} \rightarrow \R$ be a differentiable function. Assume that $M$ is a Riemannian submanifold of $\overline{M}$. For any $x$ in $M$, denote by $\pi_x:T_x\overline{M} \rightarrow T_x M$ the orthogonal projection. Then $\nabla (f|_M)(x) = \pi_x(\nabla f(x))$ for any $x$ in $M$.
\end{lemma}

\begin{lemma}\label{lemma-orthogonal-proj}
	\begin{enumerate}
		\item \cite[Example 3.6.2]{AMS} Denote the orthogonal projection onto the Stiefel manifold $\St(n,p)$ at $X \in \St(n,p)$ by $\Pi_X: \R^{n\times p} \rightarrow T_X \St(n,p)$. Then, for any $\xi$ in $\R^{n\times p}$,
		\begin{equation}\label{eq-def-orthogonal-proj-stiefel}
			\Pi_X(\xi)= (I_n -XX^T)\xi +\frac{1}{2}X(X^T\xi-\xi^TX).
		\end{equation}
		\item Denote the orthogonal projection onto the product manifold $\M$ at $(U, \mathbf{x}, V) \in \M$ by $\Pi_{(U, \mathbf{x}, V)}: \R^{m\times p} \times \R^p \times \R^{n\times p} \rightarrow T_{(U, \mathbf{x}, V)} \M$,
		where $T_{(U, \mathbf{x}, V)} \M$ is the tangent space of $\M$ at $(U, \mathbf{x}, V)$ given in Equation \eqref{eq-tangent-space}. Then, for any $(\xi, \mathbf{y}, \eta)$ in $\R^{m\times p} \times \R^p \times \R^{n\times p}$,
			\begin{equation}\label{eq-def-orthogonal-proj-prod-mfd}
				\Pi_{(U, \mathbf{x}, V)}(\xi, \mathbf{y}, \eta)= (\Pi_U(\xi), \mathbf{y}, \Pi_V(\eta)),
			\end{equation}
		where $\Pi_U$ and $\Pi_V$ are the orthogonal projections on Stiefel manifolds given in Equation \eqref{eq-def-orthogonal-proj-stiefel}.
	\end{enumerate}
\end{lemma}

\subsection{Gradient of the WLRA Loss on $\M$}\label{sec-grad-wlra}
With $h: \R^{m\times n} \rightarrow \R$ given in Equation \eqref{eq-def-wlra-loss}, define a function $\tilde{H}: \R^{m\times p}\times \R^p\times \R^{n\times p} \rightarrow \R$ by
	\[
		\tilde{H}(U, \mathbf{x}, V) = h(UD_p(\mathbf{x})V^T)
	\]
for $(U, \mathbf{x}, V) \in \R^{m\times p}\times \R^p\times \R^{n\times p}$, where $D_p: \R^p \rightarrow \R^{p\times p}$ is given in Equation \eqref{eq-def-d-p}.
Note that
\begin{equation}\label{eq-submfd}
	H = \tilde{H}\vert_{\M}.
\end{equation}
Define $(G_U, G_{\mathbf{x}}, G_V) \in \R^{m\times p}\times \R^p \times \R^{n\times p}$ by
	\begin{align}
		G_U =& \left[\frac{\partial \tilde{H}}{\partial u_{i,l}}\right]_{m\times p}
			= \left[\sum_{j=1}^n -2w_{i,j}(a_{i,j} - p_{i,j})x_l v_{j,l}\right]_{m\times p}, \label{eq-def-grad-u}\\
		G_V =& \left[\frac{\partial \tilde{H}}{\partial v_{j,l}}\right]_{n\times p}
			= \left[\sum_{i=1}^m -2w_{i,j}(a_{i,j} - p_{i,j})x_l u_{i,l}\right]_{n\times p}, \label{eq-def-grad-v} \\
		G_{\mathbf{x}} =& \left[\frac{\partial \tilde{H}}{\partial x_l}\right]_p
					   = \left[\sum_{i=1}^m\sum_{j=1}^n -2w_{i,j}(a_{i,j} - p_{i,j})u_{i, l}v_{j, l}\right]_p. \label{eq-def-grad-x}
	\end{align}
Note that, the gradient of $\tilde{H}$ at $(U, \mathbf{x}, V) \in \R^{m\times p}\times \R^p\times \R^{n\times p}$ is given by
\begin{equation}\label{eq-def-euclidean-grad}
	\nabla \tilde{H}(U, \mathbf{x}, V) = (G_U, G_{\mathbf{x}}, G_V).
\end{equation}

\begin{lemma}\label{lemma-rie-grad-wlra-loss-mfd}
At $(U, \mathbf{x}, V) \in \M$ the gradient of the function $H: \M \rightarrow \R$, given in Equation \eqref{eq-def-wlra-loss-mfd}, is given by
\begin{equation}\label{eq-def-rie-grad}
	\nabla H(U, \mathbf{x}, V) =  (\Pi_U(G_U), G_{\mathbf{x}}, \Pi_V(G_V))
\end{equation}
where $G_U$, $G_V$ and $G_{\mathbf{x}}$ are given by Equations \eqref{eq-def-grad-u}, \eqref{eq-def-grad-v} and \eqref{eq-def-grad-x} respectively, and $\Pi_U$ and $\Pi_V$ are given by Equation \eqref{eq-def-orthogonal-proj-stiefel}.
\end{lemma}

\begin{proof}
Since the product manifold $\M$ is a submanifold of $\R^{m\times p}\times \R^p\times \R^{n\times p}$, by Lemma \ref{lemma-gradient-submfd} and Equation \eqref{eq-submfd}, we have
\begin{align*}
	\nabla H(U, \mathbf{x}, V) &= \nabla (\tilde{H}\vert_{\M})(U, \mathbf{x}, V)
	= \Pi_{(U, \mathbf{x}, V)}(\nabla \tilde{H}(U, \mathbf{x}, V)) \\
	&= \Pi_{(U, \mathbf{x}, V)}(G_U, G_{\mathbf{x}}, G_v) = (\Pi_U(G_U), G_{\mathbf{x}}, \Pi_V(G_V))
\end{align*}
for $(U, \mathbf{x}, V) \in \M$, where $\nabla \tilde{H}(U, \mathbf{x}, V)$ is given in Equation \eqref{eq-def-euclidean-grad} and $\Pi_{(U, \mathbf{x}, V)}$ is given in Equation \eqref{eq-def-orthogonal-proj-prod-mfd}.
\end{proof}

\subsection{Projection-like Retractions on $\M$}\label{sec-retractions-prod-mfd}
We recall two projection-like retractions on the Stiefel manifolds and extend them to the product manifold $\M$.

\subsubsection{The QR-Decomposition-Based Retraction on the Manifold $\M$}
\begin{definition}\label{def-qf}
	Assume that $p\leq n$. For an $n\times p$ real matrix $C$ with linearly independent columns, define $\qf(C)=Q$ in the $QR$ decomposition $C=QR$, where
	\begin{itemize}
		\item $Q\in \St(n,p)$,
		\item $R$ is a $p \times p$ upper triangular matrix with positive entries along its diagonal.
	\end{itemize}
	Computationally, $\qf(C)$ can be obtained by applying the Gram-Schmidt Process on the columns of $C$, scaling the resulting orthogonal set of vectors into an orthonormal set and then using this orthonormal set of vectors as the columns of $\qf(C)$.
\end{definition}

\begin{lemma}\cite[Example 4.1.3]{AMS}\label{lemma-stiefel-qr}
	For any $X\in \St(n,p)$ and $\xi \in T_X \St(n,p)$, $X+\xi \in \R^{n\times p}$ has linearly independent columns. Define $R^{\St(n,p)}: T\St(n,p)\rightarrow \St(n,p)$ by
	\begin{equation}\label{eq-def-qr-retraction-stiefel}
		R^{\St(n,p)}_X(\xi) = \qf(X+\xi).
	\end{equation}
	$R^{\St(n,p)}$ is a retraction on $\St(n,p)$.
\end{lemma}

\begin{definition}[QR-Decomposition-Based Retraction on $\M$]\label{def-retraction-qr}
	For $(U,\mathbf{x},V) \in \M$ and $(\xi,\hat{\mathbf{x}},\eta) \in T_{(U, \mathbf{x}, V)}\M$, define
	\begin{equation}\label{eq-def-qr-retraction}
		R_{(U,\mathbf{x},V)}(\xi,\hat{\mathbf{x}},\eta) = (\qf(U+\xi),\mathbf{x} +\hat{\mathbf{x}}, \qf(V +\eta)).
	\end{equation}
	Then, by Lemma \ref{lemma-stiefel-qr}, $R: T\M \rightarrow \M$ is a retraction on $\M$. We call it the QR-decomposition-based retraction on $\M$.
\end{definition}

\subsubsection{The Polar-Decomposition-Based Retraction on the Manifold $\M$}
\begin{lemma}\cite[Example 4.1.3]{AMS}\label{lemma-stiefel-polar}
	For any $X\in \St(n,p)$ and $\xi \in T_X \St(n,p)$, define $R^{\St(n,p)}: T\St(n,p)\rightarrow \St(n,p)$ by
	\begin{equation}\label{eq-def-polar-retraction-stiefel}
		R^{\St(n,p)}_X(\xi) = \pf(X,\xi):= (X + \xi)(I_p + \xi^T\xi)^{-1/2}.
	\end{equation}
	$R^{\St(n,p)}$ is a retraction on $\St(n,p)$.
\end{lemma}
When $p$ is small, the numerical cost of evaluating Equation \ref{eq-def-polar-retraction} is reasonable because it only involves the eigenvalue decomposition of the small $p\times p$ matrix $I_p + \xi^T\xi$.

\begin{definition}[Polar-Decomposition-Based Retraction on $\M$]\label{def-retraction-polar}
	For $(U,\mathbf{x}, V) \in \M$ and $(\xi,\hat{\mathbf{x}},\eta) \in T_{(U, \mathbf{x}, V)}\M$, define
	\begin{equation}\label{eq-def-polar-retraction}
		R_{(U,\mathbf{x},V)}(\xi,\hat{\mathbf{x}},\eta) = (\pf(U,\xi),\mathbf{x}+\hat{\mathbf{x}}, \pf(V,\eta)).
	\end{equation}
	Then, by Lemma \ref{lemma-stiefel-polar}, $R: T\M \rightarrow \M$ is a retraction on $\M$. We call it the polar-decomposition-based retraction on $\M$.
\end{definition}

\subsection{Retraction-Based Gradient Projection}\label{subsec-rgp}
Recall that, with a retraction $R$ given by Equation \eqref{eq-def-qr-retraction} or \eqref{eq-def-polar-retraction},
we can fix the $R$-convex set $C^r$ given in Problem \ref{prob-wlra-constrained} on the product manifold $\M$ with the $R$-convexity bundle $\{C^r_{(U, \mathbf{x}, V)}\}_{(U, \mathbf{x}, V) \in C^r}$ given in Equation \eqref{eq-def-r-convex-bundle-wlra}.
Also, recall that the projection onto the closed ball $\overline{B}_r(\mathbf{x})$ is $\Pi_{\overline{B}_r(\mathbf{x})}: \R^p \rightarrow \overline{B}_r(\mathbf{x})$ given by
\begin{equation}\label{eq-projection-closed-ball}
\Pi_{\overline{B}_r(\mathbf{x})}(\mathbf{y}) = \mathbf{x} + \min\{1, \frac{r}{\|\mathbf{y} - \mathbf{x}\|}\}(\mathbf{y} - \mathbf{x}).
\end{equation}
Note that,
\begin{equation}\label{eq-projection-opposite-ball}
\Pi_{\overline{B}_r(-\mathbf{x})}(\mathbf{y}) = -\mathbf{x} + \Pi_{\overline{B}_r(\mathbf{0})}(\mathbf{x} + \mathbf{y})
\end{equation}
where $\Pi_{\overline{B}_r(\mathbf{0})}$ is the projection onto the closed ball $\overline{B}_r(\mathbf{0})$.

Denote by $\id_{T_U \St(m,p)}$ and $\id_{T_V \St(n,p)}$ the identity maps on $T_U \St(m,p)$ and $T_V \St(n,p)$, respectively.
The projection map $\Pi_{C^r_{(U, \mathbf{x}, V)}}: T_U \St(m,p) \times \R^p \times T_V \St(n,p) \rightarrow C^r_{(U, \mathbf{x}, V)}$ is defined by
\begin{equation}\label{eq-proj-prod-mfd}
\Pi_{C^r_{(U, \mathbf{x}, V)}} = \id_{T_U \St(m,p)} \times \Pi_{\overline{B}_r(-\mathbf{x})} \times \id_{T_V \St(n,p)}
\end{equation}
where $\Pi_{\overline{B}_r(-\mathbf{x})}$ is given in Equation \eqref{eq-projection-opposite-ball}.

Now, we are ready to introduce our retraction-based gradient projection algorithms for Problem \ref{prob-wlra-constrained}.

\FloatBarrier

\begin{algorithm}[!htbp]
\caption{Retraction-Based Gradient Projection for Problem \ref{prob-wlra-constrained} with Projection-Like Retractions}
\label{algorithm-rgpa-wlra-retractions}
\begin{algorithmic}[1]
\Require the product manifold $\M$ given in Equation \eqref{eq-def-prod-mfd},
the $R$-convex subset $C^r$ of $\M$ given in Problem \ref{prob-wlra-constrained} with the $R$-convex bundle $\{C_{(U,\mathbf{x},V)}^r\}_{(U,\mathbf{x},V)\in C^r}$ given in Equation \eqref{eq-def-r-convex-bundle-wlra},
$(U_0,\mathbf{x}_0,V_0)\in C^r$,
the objective $H:\M\to\R$ given in Equation \eqref{eq-def-wlra-loss-mfd} with the gradient $\nabla H$ given by Equation \eqref{eq-def-rie-grad},
a retraction $R$ given by Equation \eqref{eq-def-qr-retraction} or \eqref{eq-def-polar-retraction},
a stepsize rule given by Rule \ref{ur-Armijo-feasible} or \ref{ur-Armijo-projection-arc} for the sequences $\{\alpha_k\}$ and $\{s_k\}$
\For{$k=0,1,2,\ldots$}
    \State Let $\Pi_{C_{(U_k,\mathbf{x}_k,V_k)}^r}:T_{(U_k,\mathbf{x}_k,V_k)}\M \rightarrow C_{(U_k,\mathbf{x}_k,V_k)}^r$ be the projection onto $C_{(U_k,\mathbf{x}_k,V_k)}^r$ given in Equation \eqref{eq-proj-prod-mfd}
    \State $d_k \gets \Pi_{C_{(U_k,\mathbf{x}_k,V_k)}^r}\bigl(-s_k\nabla H(U_k,\mathbf{x}_k,V_k)\bigr)$
    \State $(U_{k+1},\mathbf{x}_{k+1},V_{k+1})\gets R_{(U_k,\mathbf{x}_k,V_k)}(\alpha_k d_k)$
\EndFor
\end{algorithmic}
\end{algorithm}

\FloatBarrier

Note that, in Algorithm \ref{algorithm-rgpa-wlra-retractions}, $(U_k,\mathbf{x}_k,V_k) \in C^r$ and $d_k \in  C^r_{(U_k,\mathbf{x}_k,V_k)}$ for $k = 0, 1, 2, \ldots$.

\subsection{Proofs of Proposition \ref{prop-R-convexity-wlra} and Lemma \ref{lemma-relation-bw-reg-constraint}}\label{app-proofs-more}

For $\mathbf{x} \in \R^p$ and $r > 0$, denote the open ball centered at $\mathbf{x}$ with radius $r$ by
\[
B_r(\mathbf{x}) = \{\mathbf{z} \in \R^p ~|~ \|\mathbf{z} - \mathbf{x} \| < r\}.
\]

\begin{proof}[Proof of Proposition \ref{prop-R-convexity-wlra}]
Since $\overline{B}_r(-\mathbf{x})$ is closed and convex in $\R^p$, $C^r_{(U, \mathbf{x}, V)}$ is closed and convex in $T_{(U, \mathbf{x}, V)}\M$ given in Equation \eqref{eq-tangent-space} for all $(U, \mathbf{x}, V) \in C^r$.
Therefore, $\{C^r_{(U, \mathbf{x}, V)}\}_{(U, \mathbf{x}, V) \in C^r}$ is fiber-wise convex.

For all $(U, \mathbf{x}, V) \in C^r$, since $R_{\mathbf{x}}^{\R^p}(\overline{B}_r(-\mathbf{x})) = \overline{B}_r(\mathbf{0})$ where $R^{\R^p}$ is the standard retraction on $\R^p$ given in Equation \eqref{eq-def-retraction-euclidean}, we have that $R_{(U, \mathbf{x}, V)}(C^r_{(U, \mathbf{x}, V)}) \subset C^r$.
Since $\mathbf{x} \in \overline{B}_r(\mathbf{0})$, we have that $\|\mathbf{x}\| = \|\mathbf{0} - (-\mathbf{x})\| \leq r$ and $\mathbf{0} \in \overline{B}_r(-\mathbf{x})$.
Denote the zero vector in $T_U \St(m,p)$ by $\mathbf{0}_U$ and the zero vector in $T_V \St(n,p)$ by $\mathbf{0}_V$. We conclude that $(\mathbf{0}_U, \mathbf{0}, \mathbf{0}_V) \in C^r_{(U, \mathbf{x}, V)}$.

By the Inverse Function Theorem, there exists a open neigborhood $W_U$ of $\mathbf{0}_U$ on $T_U \St(m,p)$ such that $R_U^{\St(m,p)}|_{W_U}: W_U \rightarrow R_U^{St(m,p)}(W_U)$ is a homeomorphism,
where $R^{St(m,p)}$ is a retraction on Stiefel manifolds.
Similarly, there exists a open neigborhood $W_V$ of $\mathbf{0}_V$ on $T_V \St(n,p)$  such that $R_V^{\St(n,p)}|_{W_V}: W_V \rightarrow R_V^{St(n,p)}(W_V)$ is a homeomorphism,
where $R^{St(n,p)}$ is a retraction on Stiefel manifolds.
Define $W_{(U, \mathbf{x}, V)} = W_U \times \overline{B}_r(-\mathbf{x}) \times W_V$,
which is an open neighborhood of $(\mathbf{0}_U, \mathbf{0}, \mathbf{0}_V)$ in $C^r_{(U, \mathbf{x}, V)}$.
$R_{(U, \mathbf{x}, V)}|_{W_{(U, \mathbf{x}, V)}} = R_U^{\St(m,p)}|_{W_U} \times R_{\mathbf{x}}^{\R^p}|_{\overline{B}_r(-\mathbf{x})} \times R_V^{\St(n,p)}|_{W_V}$ is a homeomorphism.
Therefore, $\{C^r_{(U, \mathbf{x}, V)}\}_{(U, \mathbf{x}, V) \in C^r}$ is locally surjective.

Given a sequence $\{\mathbf{x}_k\}_{k=1}^\infty \subset \overline{B}_r(\mathbf{0})$ and a sequence $\{\mathbf{y}_k\}_{k=1}^\infty \subset \R^p$
that satisfy $\lim_{k \rightarrow \infty} \mathbf{x}_k \rightarrow \tilde{\mathbf{x}} \in \overline{B}_r(\mathbf{0})$ and $\lim_{k \rightarrow \infty} \mathbf{y}_k \rightarrow \tilde{\mathbf{y}} \in \R^p$,
by Proposition \ref{prop-convex-projection}(3) and Equation \eqref{eq-projection-opposite-ball}, we have
\begin{align*}
	\lim_{k \rightarrow \infty} \Pi_{\overline{B}_r(-\mathbf{x}_k)}(\mathbf{y}_k) & = \lim_{k \rightarrow \infty} \left(-\mathbf{x}_k + \Pi_{\overline{B}_r(\mathbf{0})}(\mathbf{x}_k + \mathbf{y}_k)\right) \\
	&= \lim_{k \rightarrow \infty} -\mathbf{x}_k + \lim_{k \rightarrow \infty} \Pi_{\overline{B}_r(\mathbf{0})}(\mathbf{x}_k + \mathbf{y}_k) \\
	&= -\tilde{\mathbf{x}} + \Pi_{\overline{B}_r(\mathbf{0})}(\tilde{\mathbf{x}} + \tilde{\mathbf{y}}) \\
	&= \Pi_{\overline{B}_r(-\tilde{\mathbf{x}})}(\tilde{\mathbf{y}}).
\end{align*}

Therefore, given a sequence $\{(U_k, \mathbf{x}_k, V_k)\}_{k=1}^\infty \in C^r$ and a sequence $\{(Y_k, \mathbf{y}_k, Z_k) \in T_{U_k} \St(m,p) \times \R^p \times T_{V_k} \St(n,p)\}_{k=1}^\infty$ satisfying
\[
	\lim_{k \rightarrow \infty} (U_k, \mathbf{x}_k, V_k) \rightarrow (\tilde{U}, \tilde{\mathbf{x}}, \tilde{V}) \in C^r \text{ and } \lim_{k \rightarrow \infty} (Y_k, \mathbf{y}_k, Z_k) \rightarrow (\tilde{Y}, \tilde{\mathbf{y}}, \tilde{Z}) \in T_{\tilde{U}} \St(m,p) \times \R^p \times T_{\tilde{V}} \St(n,p),
\]
by Equation \eqref{eq-proj-prod-mfd}, we have
\begin{align*}
\lim_{k \rightarrow \infty} \Pi_{C^r_{(U_k, \mathbf{x}_k, V_k)}}(Y_k, \mathbf{y}_k, Z_k) & = (\lim_{k \rightarrow \infty} \id_{T_{U_k} \St(m,p)}(Y_k), \lim_{k \rightarrow \infty}\Pi_{\overline{B}_r(-\mathbf{x}_k)}(\mathbf{y}_k), \lim_{k \rightarrow \infty} \id_{T_{V_k} \St(n,p)}(Z_k)) \\
&= (\lim_{k \rightarrow \infty} Y_k, \lim_{k \rightarrow \infty}\Pi_{\overline{B}_r(-\mathbf{x}_k)}(\mathbf{y}_k), \lim_{k \rightarrow \infty} Z_k) \\
&= (\tilde{Y}, \Pi_{\overline{B}_r(-\tilde{\mathbf{x}})}(\tilde{\mathbf{y}}), \tilde{Z})
= \Pi_{C^r_{(\tilde{U}, \tilde{\mathbf{x}}, \tilde{V})}}(\tilde{Y}, \tilde{\mathbf{y}}, \tilde{Z})
\in C^r_{(\tilde{U}, \tilde{\mathbf{x}}, \tilde{V})}.
\end{align*}

Thus, $\{C^r_{(U, \mathbf{x}, V)}\}_{(U, \mathbf{x}, V) \in C^r}$ is projectively continuous.
\end{proof}

\begin{proof}[Proof of Lemma \ref{lemma-relation-bw-reg-constraint}]
If $(U^*, \mathbf{x}^*, V^*) \in \M$ is a solution to Problem \ref{prob-wlra-reg-mfd}, then we have
\[
	h(U^* D_p(\mathbf{x}^*) V^{*{^T}}) + \lambda \|\mathbf{x}^*\|^2 \leq h(U^* D_p(\mathbf{0}) V^{*{^T}}) = \|\sqrt[\bigodot]{W} \odot A\|^2_F.
\]
Therefore, we conclude that
\[
	\|\mathbf{x}^*\| \leq \frac{\|\sqrt[\bigodot]{W} \odot A\|_F}{\sqrt{\lambda}}.
\]
That is, $(U^*, \mathbf{x}^*, V^*) \in C^r$ with $r = \frac{\|\sqrt[\bigodot]{W} \odot A\|_F}{\sqrt{\lambda}}$.
\end{proof}

\clearpage

\section{Supplementary Numerical Results and Discussion}
\label{sec-numerical-supp}

In this appendix, we present more numerical validation for the convergence results in Proposition \ref{prop-rgpa-convergence}.
Following Section \ref{sec-experiments}, we refer to the main routines from applying Algorithm \ref{algorithm-rgpa} to Problem \ref{prob-wlra-constrained}
with the Armijo Stepsize Rules \ref{ur-Armijo-feasible} and \ref{ur-Armijo-projection-arc}  as ``Feasible Direction'' and ``Projection Arc'', respectively.
We refer to the benchmark routine from applying the Riemannian gradient descent algorithm to Problem \ref{prob-wlra-reg-mfd} with the Armijo stepsize rule as ``Regularized Armijo''.

\subsection{Our Design of The Grid Search Experiments}
\label{sec-grid-search-design}

Table \ref{tab:armijo-grid-search-design} summarizes the common parameter design for the grid search experiments reported in Section \ref{sec-experiments} and this appendix.

\begin{table}[!htbp]
\centering
\small
\begin{tabular}{lll}
\hline
Category & Parameter & Values or setting \\
\hline
Armijo rule & Contraction factor $\beta$ & $\{0.5,0.75,0.9\}$ \\
            & Sufficient-decrease factor $\sigma$ & $\{0.25,0.5,0.75\}$ \\
            & Initial trial step $s$ & $\{0.3,0.5,1.0\}$ \\
\hline
Configuration & Retraction & QR, polar \\
              & Low-rank $p$ & $\{32,64,128\}$ \\
              & Regularization parameter $\lambda$ & $\{10^{-2},10^{-4},10^{-6}\}$ \\
\hline
Fixed setting & Iterations per run & $200$ \\
              & Maximum backtracking steps & $200$ \\
              & Armijo tolerance & $10^{-12}$ \\
\hline
\end{tabular}
\caption{Parameters used in the Armijo grid searches.
Each combination of the Armijo rule parameters $\beta$, $\sigma$ and $s$ defines one Armijo grid point, giving $27$ points for each routine within a retraction-rank-regularization configuration.
There are 18 retraction-rank-regularization configurations per dataset.
With the three routines, each dataset therefore has $27 \times 18 \times 3 = 1458$ runs in each grid search.
The Armijo tolerance is a numerical slack in the Armijo stepsize rules computations.}
\label{tab:armijo-grid-search-design}
\end{table}

\FloatBarrier

\subsection{More Two-Stage Experiments on the Sampled Datasets}
\label{sec-two-stage-sampled}
We repeated the two-stage experiments on the same sampled MNIST and CIFAR-10 datasets with masking rates $0.25$ and $0.75$.

\subsubsection{Masking Rate $= 0.75$}
\label{sec-sampled-mnist-cifar-mask-rate-0.75}
We present the numerical results from the experiments with the masking rate set to $0.75$.
Table \ref{tab:mnist-cifar-mask075-armijo-grid} summarizes short-run grid-search results.
Figure \ref{fig:mnist-cifar-mask075-armijo-rmse-curves} presents the long-run RMSE convergence curves for all the configurations.
Combining the results from both stages, we observe that our main routines
could outperform the benchmark routine over both short and long horizons
when we use the QR-decomposition-based retraction and the low-rank $p = 64$ on the sampled MNIST dataset,
or when we use the polar-decomposition-based retraction and the low-rank $p = 64$ or $p = 128$ on the sampled CIFAR-10 dataset.

\begin{table}[!htbp]
\centering
\small
\renewcommand{\arraystretch}{0.95}
\setlength{\tabcolsep}{0.35mm}
\begin{tabular}{@{}cc|rrrr|rrrr|rrrr@{}}
\hline
\multicolumn{14}{c}{\textbf{Sampled MNIST}} \\
\hline
\multicolumn{14}{c}{\textbf{QR retraction}} \\
\hline
rank & $\lambda$ & \multicolumn{4}{c|}{\shortstack{Regularized Armijo\\(Benchmark)}} & \multicolumn{4}{c|}{\shortstack{Feasible Direction\\(Main Routine)}} & \multicolumn{4}{c}{\shortstack{Projection Arc\\(Main Routine)}} \\
 & & best & mean & evals & fail. & best & mean & evals & fail. & best & mean & evals & fail. \\
\hline
32 & $10^{-2}$ & \textbf{0.2734 $(0.9,0.75,0.5)$} & 0.2786 & 43.6 & 0 & 0.2735 $(0.9,0.75,0.3)$ & 0.2786 & 43.5 & 0 & 0.2735 $(0.9,0.75,0.5)$ & 0.2785 & 43.6 & 0 \\
32 & $10^{-4}$ & 0.2745 $(0.9,0.75,0.3)$ & 0.2786 & 43.6 & 0 & \textbf{0.2735 $(0.9,0.75,0.3)$} & 0.2786 & 43.5 & 0 & 0.2735 $(0.9,0.75,0.5)$ & 0.2785 & 43.6 & 0 \\
32 & $10^{-6}$ & \textbf{0.2729 $(0.9,0.75,0.3)$} & 0.2786 & 43.6 & 0 & 0.2735 $(0.9,0.75,0.3)$ & 0.2786 & 43.5 & 0 & 0.2735 $(0.9,0.75,0.5)$ & 0.2785 & 43.6 & 0 \\
64 & $10^{-2}$ & 0.2862 $(0.9,0.75,1)$ & 0.2904 & 42.9 & 0 & \textbf{0.2831 $(0.9,0.75,0.5)$} & 0.2902 & 42.9 & 0 & 0.2872 $(0.9,0.75,0.3)$ & 0.2904 & 42.9 & 0 \\
64 & $10^{-4}$ & 0.2862 $(0.9,0.75,1)$ & 0.2904 & 42.9 & 0 & \textbf{0.2831 $(0.9,0.75,0.5)$} & 0.2902 & 42.9 & 0 & 0.2872 $(0.9,0.75,0.3)$ & 0.2904 & 42.9 & 0 \\
64 & $10^{-6}$ & 0.2864 $(0.9,0.75,0.3)$ & 0.2904 & 42.9 & 0 & \textbf{0.2831 $(0.9,0.75,0.5)$} & 0.2902 & 42.9 & 0 & 0.2872 $(0.9,0.75,0.3)$ & 0.2904 & 42.9 & 0 \\
128 & $10^{-2}$ & \textbf{0.3080 $(0.9,0.5,1)$} & 0.3102 & 41.9 & 0 & 0.3084 $(0.9,0.5,1)$ & 0.3102 & 41.9 & 0 & 0.3082 $(0.9,0.75,0.3)$ & 0.3102 & 42.0 & 0 \\
128 & $10^{-4}$ & 0.3085 $(0.9,0.75,0.3)$ & 0.3103 & 42.0 & 0 & 0.3084 $(0.9,0.5,1)$ & 0.3102 & 41.9 & 0 & \textbf{0.3082 $(0.9,0.75,0.3)$} & 0.3102 & 42.0 & 0 \\
128 & $10^{-6}$ & 0.3084 $(0.9,0.5,1)$ & 0.3102 & 41.9 & 0 & 0.3084 $(0.9,0.5,1)$ & 0.3102 & 41.9 & 0 & \textbf{0.3082 $(0.9,0.75,0.3)$} & 0.3102 & 42.0 & 0 \\
\hline
\multicolumn{14}{c}{\textbf{POLAR retraction}} \\
\hline
rank & $\lambda$ & \multicolumn{4}{c|}{\shortstack{Regularized Armijo\\(Benchmark)}} & \multicolumn{4}{c|}{\shortstack{Feasible Direction\\(Main Routine)}} & \multicolumn{4}{c}{\shortstack{Projection Arc\\(Main Routine)}} \\
 & & best & mean & evals & fail. & best & mean & evals & fail. & best & mean & evals & fail. \\
\hline
32 & $10^{-2}$ & 0.2821 $(0.5,0.5,0.5)$ & 0.2836 & 51.0 & 9 & 0.2802 $(0.75,0.5,1)$ & 0.2827 & 51.0 & 8 & \textbf{0.2770 $(0.9,0.75,0.3)$} & 0.2826 & 69.6 & 8 \\
32 & $10^{-4}$ & 0.2803 $(0.75,0.25,0.5)$ & 0.2828 & 47.2 & 9 & 0.2802 $(0.75,0.5,1)$ & 0.2827 & 51.0 & 8 & \textbf{0.2770 $(0.9,0.75,0.3)$} & 0.2826 & 69.6 & 8 \\
32 & $10^{-6}$ & 0.2801 $(0.75,0.5,1)$ & 0.2828 & 47.0 & 9 & 0.2802 $(0.75,0.5,1)$ & 0.2827 & 51.0 & 8 & \textbf{0.2770 $(0.9,0.75,0.3)$} & 0.2826 & 69.6 & 8 \\
64 & $10^{-2}$ & 0.2886 $(0.9,0.5,1)$ & 0.2926 & 61.4 & 2 & \textbf{0.2852 $(0.9,0.75,0.3)$} & 0.2923 & 57.2 & 3 & 0.2887 $(0.9,0.5,1)$ & 0.2927 & 64.9 & 2 \\
64 & $10^{-4}$ & 0.2904 $(0.9,0.75,0.5)$ & 0.2928 & 60.6 & 3 & \textbf{0.2852 $(0.9,0.75,0.3)$} & 0.2923 & 57.2 & 3 & 0.2887 $(0.9,0.5,1)$ & 0.2927 & 64.9 & 2 \\
64 & $10^{-6}$ & 0.2907 $(0.75,0.25,1)$ & 0.2929 & 64.0 & 1 & \textbf{0.2852 $(0.9,0.75,0.3)$} & 0.2923 & 57.2 & 3 & 0.2887 $(0.9,0.5,1)$ & 0.2927 & 64.9 & 2 \\
128 & $10^{-2}$ & \textbf{0.3070 $(0.9,0.5,1)$} & 0.3102 & 44.0 & 0 & 0.3080 $(0.9,0.75,0.5)$ & 0.3103 & 43.5 & 0 & 0.3078 $(0.9,0.5,1)$ & 0.3103 & 44.8 & 0 \\
128 & $10^{-4}$ & \textbf{0.3069 $(0.9,0.5,1)$} & 0.3102 & 44.8 & 0 & 0.3080 $(0.9,0.75,0.5)$ & 0.3103 & 43.5 & 0 & 0.3078 $(0.9,0.5,1)$ & 0.3103 & 44.8 & 0 \\
128 & $10^{-6}$ & 0.3080 $(0.9,0.75,1)$ & 0.3102 & 41.5 & 1 & 0.3080 $(0.9,0.75,0.5)$ & 0.3103 & 43.5 & 0 & \textbf{0.3078 $(0.9,0.5,1)$} & 0.3103 & 44.8 & 0 \\
\hline
\multicolumn{14}{c}{\textbf{Sampled CIFAR-10}} \\
\hline
\multicolumn{14}{c}{\textbf{QR retraction}} \\
\hline
rank & $\lambda$ & \multicolumn{4}{c|}{\shortstack{Regularized Armijo\\(Benchmark)}} & \multicolumn{4}{c|}{\shortstack{Feasible Direction\\(Main Routine)}} & \multicolumn{4}{c}{\shortstack{Projection Arc\\(Main Routine)}} \\
 & & best & mean & evals & fail. & best & mean & evals & fail. & best & mean & evals & fail. \\
\hline
32 & $10^{-2}$ & 0.4143 $(0.9,0.75,1)$ & 0.4183 & 51.7 & 0 & 0.4164 $(0.9,0.75,0.5)$ & 0.4185 & 51.6 & 0 & \textbf{0.4142 $(0.9,0.75,1)$} & 0.4183 & 51.6 & 0 \\
32 & $10^{-4}$ & 0.4147 $(0.9,0.75,1)$ & 0.4183 & 51.7 & 0 & 0.4164 $(0.9,0.75,0.5)$ & 0.4185 & 51.6 & 0 & \textbf{0.4142 $(0.9,0.75,1)$} & 0.4183 & 51.6 & 0 \\
32 & $10^{-6}$ & 0.4158 $(0.9,0.75,1)$ & 0.4184 & 51.7 & 0 & 0.4164 $(0.9,0.75,0.5)$ & 0.4185 & 51.6 & 0 & \textbf{0.4142 $(0.9,0.75,1)$} & 0.4183 & 51.6 & 0 \\
64 & $10^{-2}$ & 0.4269 $(0.9,0.75,0.3)$ & 0.4291 & 51.3 & 0 & \textbf{0.4259 $(0.9,0.75,0.3)$} & 0.4290 & 51.4 & 0 & 0.4260 $(0.9,0.75,0.5)$ & 0.4289 & 51.4 & 0 \\
64 & $10^{-4}$ & 0.4266 $(0.9,0.75,0.3)$ & 0.4291 & 51.4 & 0 & \textbf{0.4259 $(0.9,0.75,0.3)$} & 0.4290 & 51.4 & 0 & 0.4260 $(0.9,0.75,0.5)$ & 0.4289 & 51.4 & 0 \\
64 & $10^{-6}$ & \textbf{0.4228 $(0.9,0.75,0.3)$} & 0.4289 & 51.3 & 0 & 0.4259 $(0.9,0.75,0.3)$ & 0.4290 & 51.4 & 0 & 0.4260 $(0.9,0.75,0.5)$ & 0.4289 & 51.4 & 0 \\
128 & $10^{-2}$ & 0.4462 $(0.9,0.75,0.3)$ & 0.4478 & 50.8 & 0 & \textbf{0.4461 $(0.9,0.75,0.3)$} & 0.4478 & 50.9 & 0 & 0.4463 $(0.9,0.75,0.5)$ & 0.4478 & 50.8 & 0 \\
128 & $10^{-4}$ & \textbf{0.4459 $(0.9,0.75,0.3)$} & 0.4478 & 50.8 & 0 & 0.4461 $(0.9,0.75,0.3)$ & 0.4478 & 50.9 & 0 & 0.4463 $(0.9,0.75,0.5)$ & 0.4478 & 50.8 & 0 \\
128 & $10^{-6}$ & \textbf{0.4457 $(0.9,0.75,0.3)$} & 0.4478 & 50.8 & 0 & 0.4461 $(0.9,0.75,0.3)$ & 0.4478 & 50.9 & 0 & 0.4463 $(0.9,0.75,0.5)$ & 0.4478 & 50.8 & 0 \\
\hline
\multicolumn{14}{c}{\textbf{POLAR retraction}} \\
\hline
rank & $\lambda$ & \multicolumn{4}{c|}{\shortstack{Regularized Armijo\\(Benchmark)}} & \multicolumn{4}{c|}{\shortstack{Feasible Direction\\(Main Routine)}} & \multicolumn{4}{c}{\shortstack{Projection Arc\\(Main Routine)}} \\
 & & best & mean & evals & fail. & best & mean & evals & fail. & best & mean & evals & fail. \\
\hline
32 & $10^{-2}$ & \textbf{0.4191 $(0.5,0.25,0.3)$} & 0.4207 & 76.4 & 9 & 0.4194 $(0.5,0.25,0.3)$ & 0.4207 & 79.5 & 9 & 0.4194 $(0.5,0.25,0.3)$ & 0.4208 & 79.5 & 9 \\
32 & $10^{-4}$ & \textbf{0.4192 $(0.75,0.25,0.5)$} & 0.4207 & 78.8 & 9 & 0.4194 $(0.5,0.25,0.3)$ & 0.4207 & 79.5 & 9 & 0.4194 $(0.5,0.25,0.3)$ & 0.4208 & 79.5 & 9 \\
32 & $10^{-6}$ & \textbf{0.4191 $(0.5,0.25,0.3)$} & 0.4207 & 79.7 & 9 & 0.4194 $(0.5,0.25,0.3)$ & 0.4207 & 79.5 & 9 & 0.4194 $(0.5,0.25,0.3)$ & 0.4208 & 79.5 & 9 \\
64 & $10^{-2}$ & 0.4296 $(0.75,0.25,0.5)$ & 0.4312 & 73.8 & 9 & 0.4296 $(0.75,0.25,0.5)$ & 0.4313 & 71.4 & 9 & \textbf{0.4296 $(0.75,0.25,0.5)$} & 0.4314 & 80.6 & 9 \\
64 & $10^{-4}$ & 0.4297 $(0.75,0.25,0.5)$ & 0.4313 & 72.2 & 9 & 0.4296 $(0.75,0.25,0.5)$ & 0.4313 & 71.4 & 9 & \textbf{0.4296 $(0.75,0.25,0.5)$} & 0.4314 & 80.6 & 9 \\
64 & $10^{-6}$ & 0.4296 $(0.75,0.25,0.5)$ & 0.4313 & 72.3 & 9 & 0.4296 $(0.75,0.25,0.5)$ & 0.4313 & 71.4 & 9 & \textbf{0.4296 $(0.75,0.25,0.5)$} & 0.4314 & 80.6 & 9 \\
128 & $10^{-2}$ & 0.4468 $(0.75,0.25,1)$ & 0.4495 & 61.4 & 9 & \textbf{0.4467 $(0.75,0.25,1)$} & 0.4494 & 60.0 & 9 & 0.4467 $(0.75,0.25,1)$ & 0.4494 & 77.4 & 9 \\
128 & $10^{-4}$ & 0.4475 $(0.75,0.25,1)$ & 0.4497 & 56.8 & 9 & \textbf{0.4467 $(0.75,0.25,1)$} & 0.4494 & 60.0 & 9 & 0.4467 $(0.75,0.25,1)$ & 0.4494 & 77.4 & 9 \\
128 & $10^{-6}$ & 0.4475 $(0.75,0.25,1)$ & 0.4495 & 61.0 & 9 & \textbf{0.4467 $(0.75,0.25,1)$} & 0.4494 & 60.0 & 9 & 0.4467 $(0.75,0.25,1)$ & 0.4494 & 77.4 & 9 \\
\hline
\end{tabular}

\caption{For each retraction-rank-regularization configuration,
the table reports each routine's best and mean final RMSE, mean Armijo inequality checks per iteration (``mean Armijo Evaluations''), and failure count.
Evaluation counts are averaged over each successful run's $200$ iterations, then across successful grid runs.
Failures count grid runs without a successful final metric.
Failed records are excluded from evaluating the other three metrics.
Parentheses give the Armijo parameters $(\beta,\sigma,s)$ for the best RMSE entry.
Boldface marks the lowest best RMSE among the three routines in each row.
The upper table reports results for the sampled MNIST dataset, and the lower table reports results for the sampled CIFAR-10 dataset.
The masking rate is set to 0.75.}
\label{tab:mnist-cifar-mask075-armijo-grid}
\end{table}

\begin{figure}[!htbp]
\centering
\includegraphics[width=\textwidth]{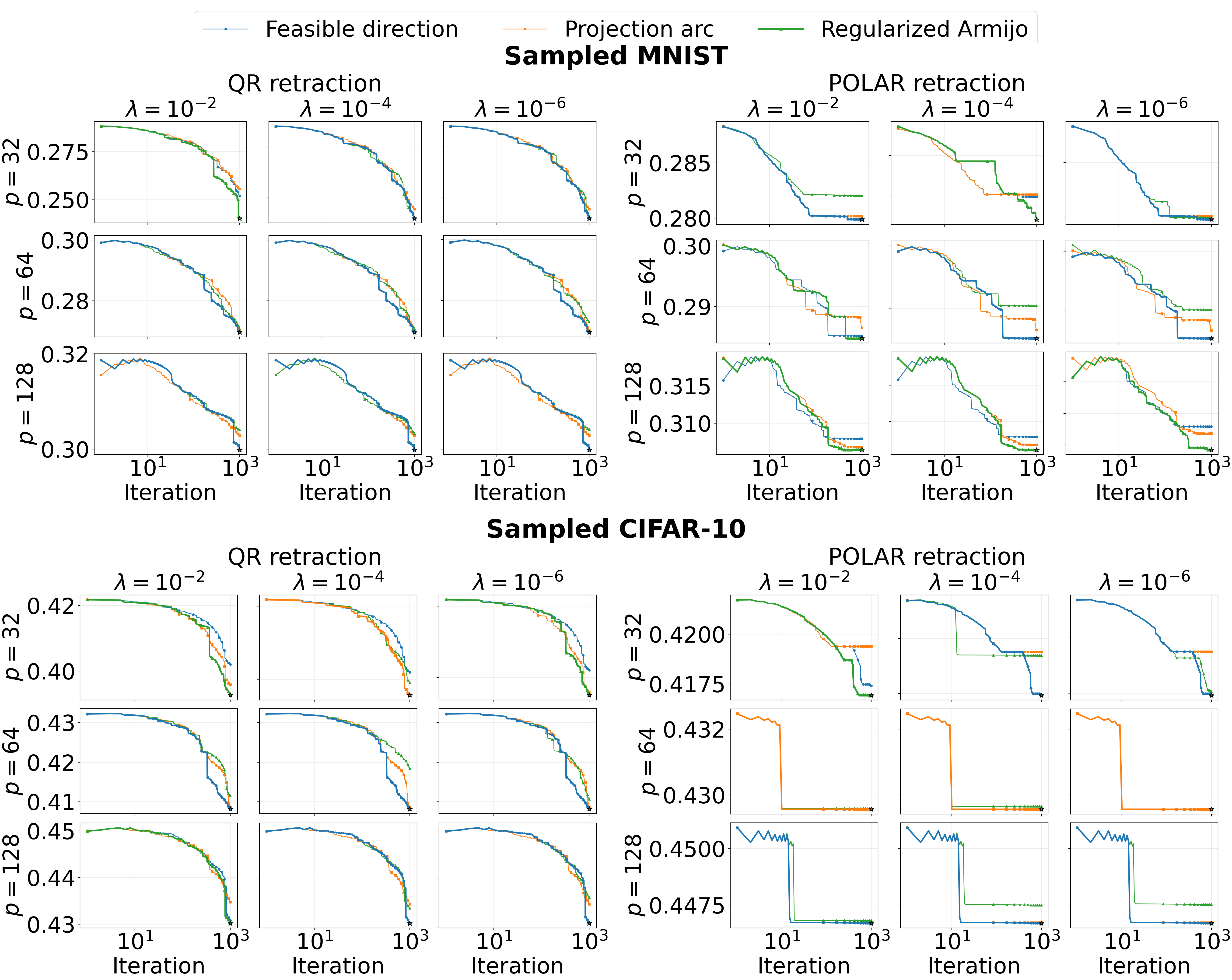}
\caption{RMSE convergence over $1000$ iterations using the optimal Armijo-rule parameters selected from the grid search summarized in Table \ref{tab:mnist-cifar-mask075-armijo-grid}.
    The left panels use the QR-decomposition-based retraction and the right panels use the polar-decomposition-based retraction.
    The $x$-axis is logarithmic. The thicker curve with a star at its final point marks the lowest final RMSE in that subplot.
    The masking rate is set to 0.75. On the sampled MNIST dataset (the upper panel),
    the ``Feasible Direction'' routine appears to consistently outperform the Regularized Armijo (benchmark) routine when using the QR-decomposition-based retraction and the low-rank $p = 64$ or $p = 128$.
    On the sampled CIFAR-10 dataset (the lower panel), (1) the ``Feasible Direction'' routine appears to consistently outperform the Regularized Armijo (benchmark) routine
    when using the QR-decomposition-based retraction and the low-rank $p = 64$,
    or when using the polar-decomposition-based retraction and the low-rank $p = 128$;
    (2) the ``Projection Arc'' routine appears to consistently outperform the Regularized Armijo (benchmark) routine
    when using the polar-decomposition-based retraction and the low-rank $p = 64$.}
\label{fig:mnist-cifar-mask075-armijo-rmse-curves}
\end{figure}

\FloatBarrier

\subsubsection{Masking Rate $= 0.25$}
\label{sec-sampled-mnist-cifar-mask-rate-0.25}
We present the numerical results from the experiments with the masking rate set to $0.25$.
Table \ref{tab:mnist-cifar-mask025-armijo-grid} summarizes the short-run grid-search results.
Figure \ref{fig:mnist-cifar-mask025-armijo-rmse-curves} presents the long-run RMSE convergence curves for all the configurations.
Combining the results from both stages, we observe that our main routines
could outperform the benchmark routine over both short and long horizons
when we use the polar-decomposition-based retraction and the low-rank $p = 128$ on the sampled MNIST dataset,
although the best-performing main routine differs between the two stages.

\begin{table}[!htbp]
\centering
\small
\renewcommand{\arraystretch}{0.95}
\setlength{\tabcolsep}{0.35mm}
\begin{tabular}{@{}cc|rrrr|rrrr|rrrr@{}}
\hline
\multicolumn{14}{c}{\textbf{Sampled MNIST}} \\
\hline
\multicolumn{14}{c}{\textbf{QR retraction}} \\
\hline
rank & $\lambda$ & \multicolumn{4}{c|}{\shortstack{Regularized Armijo\\(Benchmark)}} & \multicolumn{4}{c|}{\shortstack{Feasible Direction\\(Main Routine)}} & \multicolumn{4}{c}{\shortstack{Projection Arc\\(Main Routine)}} \\
 & & best & mean & evals & fail. & best & mean & evals & fail. & best & mean & evals & fail. \\
\hline
32 & $10^{-2}$ & \textbf{0.1649 $(0.9,0.75,1)$} & 0.1654 & 54.2 & 0 & 0.1651 $(0.9,0.75,0.3)$ & 0.1654 & 54.3 & 0 & 0.1650 $(0.9,0.75,0.3)$ & 0.1654 & 54.3 & 0 \\
32 & $10^{-4}$ & 0.1651 $(0.9,0.75,0.5)$ & 0.1654 & 54.2 & 0 & 0.1651 $(0.9,0.75,0.3)$ & 0.1654 & 54.3 & 0 & \textbf{0.1650 $(0.9,0.75,0.3)$} & 0.1654 & 54.3 & 0 \\
32 & $10^{-6}$ & 0.1651 $(0.9,0.75,0.3)$ & 0.1654 & 54.2 & 0 & 0.1651 $(0.9,0.75,0.3)$ & 0.1654 & 54.3 & 0 & \textbf{0.1650 $(0.9,0.75,0.3)$} & 0.1654 & 54.3 & 0 \\
64 & $10^{-2}$ & 0.1493 $(0.9,0.75,1)$ & 0.1525 & 54.2 & 0 & \textbf{0.1492 $(0.9,0.75,1)$} & 0.1524 & 54.1 & 0 & 0.1493 $(0.9,0.75,1)$ & 0.1524 & 54.1 & 0 \\
64 & $10^{-4}$ & 0.1492 $(0.9,0.75,1)$ & 0.1525 & 54.2 & 0 & \textbf{0.1492 $(0.9,0.75,1)$} & 0.1524 & 54.1 & 0 & 0.1493 $(0.9,0.75,1)$ & 0.1524 & 54.1 & 0 \\
64 & $10^{-6}$ & 0.1493 $(0.9,0.75,1)$ & 0.1524 & 54.1 & 0 & \textbf{0.1492 $(0.9,0.75,1)$} & 0.1524 & 54.1 & 0 & 0.1493 $(0.9,0.75,1)$ & 0.1524 & 54.1 & 0 \\
128 & $10^{-2}$ & 0.2231 $(0.9,0.75,0.3)$ & 0.2293 & 53.7 & 0 & 0.2250 $(0.9,0.75,0.5)$ & 0.2294 & 53.7 & 0 & \textbf{0.2225 $(0.9,0.75,0.3)$} & 0.2292 & 53.7 & 0 \\
128 & $10^{-4}$ & 0.2233 $(0.9,0.75,0.5)$ & 0.2292 & 53.7 & 0 & 0.2250 $(0.9,0.75,0.5)$ & 0.2294 & 53.7 & 0 & \textbf{0.2225 $(0.9,0.75,0.3)$} & 0.2292 & 53.7 & 0 \\
128 & $10^{-6}$ & 0.2240 $(0.9,0.75,0.3)$ & 0.2293 & 53.7 & 0 & 0.2250 $(0.9,0.75,0.5)$ & 0.2294 & 53.7 & 0 & \textbf{0.2225 $(0.9,0.75,0.3)$} & 0.2292 & 53.7 & 0 \\
\hline
\multicolumn{14}{c}{\textbf{POLAR retraction}} \\
\hline
rank & $\lambda$ & \multicolumn{4}{c|}{\shortstack{Regularized Armijo\\(Benchmark)}} & \multicolumn{4}{c|}{\shortstack{Feasible Direction\\(Main Routine)}} & \multicolumn{4}{c}{\shortstack{Projection Arc\\(Main Routine)}} \\
 & & best & mean & evals & fail. & best & mean & evals & fail. & best & mean & evals & fail. \\
\hline
32 & $10^{-2}$ & \textbf{0.1653 $(0.5,0.75,0.3)$} & 0.1657 & 61.0 & 2 & 0.1653 $(0.75,0.75,0.3)$ & 0.1657 & 60.4 & 2 & 0.1654 $(0.75,0.5,0.3)$ & 0.1657 & 59.0 & 3 \\
32 & $10^{-4}$ & 0.1654 $(0.9,0.75,1)$ & 0.1658 & 61.8 & 2 & \textbf{0.1653 $(0.75,0.75,0.3)$} & 0.1657 & 60.4 & 2 & 0.1654 $(0.75,0.5,0.3)$ & 0.1657 & 59.0 & 3 \\
32 & $10^{-6}$ & 0.1655 $(0.5,0.75,0.5)$ & 0.1657 & 60.7 & 2 & \textbf{0.1653 $(0.75,0.75,0.3)$} & 0.1657 & 60.4 & 2 & 0.1654 $(0.75,0.5,0.3)$ & 0.1657 & 59.0 & 3 \\
64 & $10^{-2}$ & 0.1516 $(0.5,0.5,0.5)$ & 0.1544 & 64.0 & 0 & 0.1514 $(0.5,0.5,0.5)$ & 0.1544 & 63.7 & 0 & \textbf{0.1513 $(0.9,0.75,0.3)$} & 0.1542 & 60.2 & 1 \\
64 & $10^{-4}$ & \textbf{0.1507 $(0.9,0.75,0.3)$} & 0.1542 & 63.2 & 0 & 0.1514 $(0.5,0.5,0.5)$ & 0.1544 & 63.7 & 0 & 0.1513 $(0.9,0.75,0.3)$ & 0.1542 & 60.2 & 1 \\
64 & $10^{-6}$ & 0.1516 $(0.5,0.5,0.5)$ & 0.1543 & 64.5 & 0 & 0.1514 $(0.5,0.5,0.5)$ & 0.1544 & 63.7 & 0 & \textbf{0.1513 $(0.9,0.75,0.3)$} & 0.1542 & 60.2 & 1 \\
128 & $10^{-2}$ & 0.2282 $(0.75,0.5,0.3)$ & 0.2305 & 59.7 & 0 & \textbf{0.2258 $(0.9,0.75,1)$} & 0.2301 & 58.2 & 0 & 0.2262 $(0.9,0.75,1)$ & 0.2302 & 58.8 & 0 \\
128 & $10^{-4}$ & 0.2270 $(0.9,0.75,0.3)$ & 0.2303 & 59.9 & 0 & \textbf{0.2258 $(0.9,0.75,1)$} & 0.2301 & 58.2 & 0 & 0.2262 $(0.9,0.75,1)$ & 0.2302 & 58.8 & 0 \\
128 & $10^{-6}$ & 0.2269 $(0.9,0.75,0.3)$ & 0.2302 & 57.9 & 0 & \textbf{0.2258 $(0.9,0.75,1)$} & 0.2301 & 58.2 & 0 & 0.2262 $(0.9,0.75,1)$ & 0.2302 & 58.8 & 0 \\
\hline
\multicolumn{14}{c}{\textbf{Sampled CIFAR-10}} \\
\hline
\multicolumn{14}{c}{\textbf{QR retraction}} \\
\hline
rank & $\lambda$ & \multicolumn{4}{c|}{\shortstack{Regularized Armijo\\(Benchmark)}} & \multicolumn{4}{c|}{\shortstack{Feasible Direction\\(Main Routine)}} & \multicolumn{4}{c}{\shortstack{Projection Arc\\(Main Routine)}} \\
 & & best & mean & evals & fail. & best & mean & evals & fail. & best & mean & evals & fail. \\
\hline
32 & $10^{-2}$ & 0.1704 $(0.75,0.5,0.3)$ & 0.1740 & 62.3 & 0 & 0.1704 $(0.75,0.5,0.3)$ & 0.1743 & 57.3 & 2 & \textbf{0.1702 $(0.75,0.5,0.3)$} & 0.1741 & 62.3 & 0 \\
32 & $10^{-4}$ & 0.1704 $(0.75,0.5,0.3)$ & 0.1742 & 60.1 & 1 & 0.1704 $(0.75,0.5,0.3)$ & 0.1743 & 57.3 & 2 & \textbf{0.1702 $(0.75,0.5,0.3)$} & 0.1741 & 62.3 & 0 \\
32 & $10^{-6}$ & 0.1705 $(0.75,0.5,0.3)$ & 0.1742 & 60.0 & 1 & 0.1704 $(0.75,0.5,0.3)$ & 0.1743 & 57.3 & 2 & \textbf{0.1702 $(0.75,0.5,0.3)$} & 0.1741 & 62.3 & 0 \\
64 & $10^{-2}$ & 0.1759 $(0.9,0.75,0.3)$ & 0.1886 & 62.3 & 0 & \textbf{0.1756 $(0.9,0.75,0.3)$} & 0.1885 & 62.2 & 0 & 0.1763 $(0.9,0.75,0.3)$ & 0.1886 & 62.3 & 0 \\
64 & $10^{-4}$ & 0.1761 $(0.9,0.75,0.3)$ & 0.1885 & 62.3 & 0 & \textbf{0.1756 $(0.9,0.75,0.3)$} & 0.1885 & 62.2 & 0 & 0.1763 $(0.9,0.75,0.3)$ & 0.1886 & 62.3 & 0 \\
64 & $10^{-6}$ & \textbf{0.1740 $(0.9,0.75,0.3)$} & 0.1884 & 62.2 & 0 & 0.1756 $(0.9,0.75,0.3)$ & 0.1885 & 62.2 & 0 & 0.1763 $(0.9,0.75,0.3)$ & 0.1886 & 62.3 & 0 \\
128 & $10^{-2}$ & 0.2293 $(0.9,0.75,0.3)$ & 0.2402 & 62.1 & 0 & \textbf{0.2286 $(0.9,0.75,0.3)$} & 0.2401 & 62.1 & 0 & 0.2287 $(0.9,0.75,0.3)$ & 0.2400 & 62.1 & 0 \\
128 & $10^{-4}$ & 0.2295 $(0.9,0.75,0.3)$ & 0.2402 & 62.1 & 0 & \textbf{0.2286 $(0.9,0.75,0.3)$} & 0.2401 & 62.1 & 0 & 0.2287 $(0.9,0.75,0.3)$ & 0.2400 & 62.1 & 0 \\
128 & $10^{-6}$ & \textbf{0.2280 $(0.9,0.75,0.3)$} & 0.2401 & 62.1 & 0 & 0.2286 $(0.9,0.75,0.3)$ & 0.2401 & 62.1 & 0 & 0.2287 $(0.9,0.75,0.3)$ & 0.2400 & 62.1 & 0 \\
\hline
\multicolumn{14}{c}{\textbf{POLAR retraction}} \\
\hline
rank & $\lambda$ & \multicolumn{4}{c|}{\shortstack{Regularized Armijo\\(Benchmark)}} & \multicolumn{4}{c|}{\shortstack{Feasible Direction\\(Main Routine)}} & \multicolumn{4}{c}{\shortstack{Projection Arc\\(Main Routine)}} \\
 & & best & mean & evals & fail. & best & mean & evals & fail. & best & mean & evals & fail. \\
\hline
32 & $10^{-2}$ & \textbf{0.1744 $(0.5,0.5,0.3)$} & 0.1766 & 51.8 & 6 & 0.1744 $(0.5,0.5,0.3)$ & 0.1766 & 51.4 & 6 & 0.1746 $(0.5,0.5,0.3)$ & 0.1767 & 57.2 & 6 \\
32 & $10^{-4}$ & 0.1748 $(0.5,0.5,0.3)$ & 0.1769 & 52.6 & 6 & \textbf{0.1744 $(0.5,0.5,0.3)$} & 0.1766 & 51.4 & 6 & 0.1746 $(0.5,0.5,0.3)$ & 0.1767 & 57.2 & 6 \\
32 & $10^{-6}$ & \textbf{0.1744 $(0.5,0.5,0.3)$} & 0.1766 & 51.4 & 6 & 0.1744 $(0.5,0.5,0.3)$ & 0.1766 & 51.4 & 6 & 0.1746 $(0.5,0.5,0.3)$ & 0.1767 & 57.2 & 6 \\
64 & $10^{-2}$ & 0.1908 $(0.5,0.5,0.3)$ & 0.1947 & 53.8 & 6 & 0.1880 $(0.75,0.5,1)$ & 0.1938 & 51.5 & 6 & \textbf{0.1848 $(0.75,0.75,0.3)$} & 0.1932 & 58.2 & 6 \\
64 & $10^{-4}$ & 0.1897 $(0.75,0.5,1)$ & 0.1945 & 54.6 & 6 & 0.1880 $(0.75,0.5,1)$ & 0.1938 & 51.5 & 6 & \textbf{0.1848 $(0.75,0.75,0.3)$} & 0.1932 & 58.2 & 6 \\
64 & $10^{-6}$ & 0.1903 $(0.75,0.75,1)$ & 0.1940 & 52.5 & 6 & 0.1880 $(0.75,0.5,1)$ & 0.1938 & 51.5 & 6 & \textbf{0.1848 $(0.75,0.75,0.3)$} & 0.1932 & 58.2 & 6 \\
128 & $10^{-2}$ & \textbf{0.2368 $(0.75,0.75,0.3)$} & 0.2431 & 50.4 & 6 & 0.2400 $(0.75,0.5,0.5)$ & 0.2438 & 53.9 & 6 & 0.2392 $(0.75,0.5,0.5)$ & 0.2437 & 58.7 & 6 \\
128 & $10^{-4}$ & 0.2407 $(0.75,0.75,1)$ & 0.2438 & 52.4 & 6 & 0.2400 $(0.75,0.5,0.5)$ & 0.2438 & 53.9 & 6 & \textbf{0.2392 $(0.75,0.5,0.5)$} & 0.2437 & 58.7 & 6 \\
128 & $10^{-6}$ & 0.2417 $(0.75,0.75,1)$ & 0.2442 & 54.1 & 6 & 0.2400 $(0.75,0.5,0.5)$ & 0.2438 & 53.9 & 6 & \textbf{0.2392 $(0.75,0.5,0.5)$} & 0.2437 & 58.7 & 6 \\
\hline
\end{tabular}

\caption{For each retraction-rank-regularization configuration,
the table reports each routine's best and mean final RMSE, mean Armijo inequality checks per iteration (``mean Armijo Evaluations''), and failure count.
Evaluation counts are averaged over each successful run's $200$ iterations, then across successful grid runs.
Failures count grid runs without a successful final metric.
Failed records are excluded from evaluating the other three metrics.
Parentheses give the Armijo parameters $(\beta,\sigma,s)$ for the best RMSE entry.
Boldface marks the lowest best RMSE among the three routines in each row.
The upper table reports results for the sampled MNIST dataset, and the lower table reports results for the sampled CIFAR-10 dataset.
The masking rate is set to 0.25.}
\label{tab:mnist-cifar-mask025-armijo-grid}
\end{table}

\begin{figure}[!htbp]
\centering
\includegraphics[width=\textwidth]{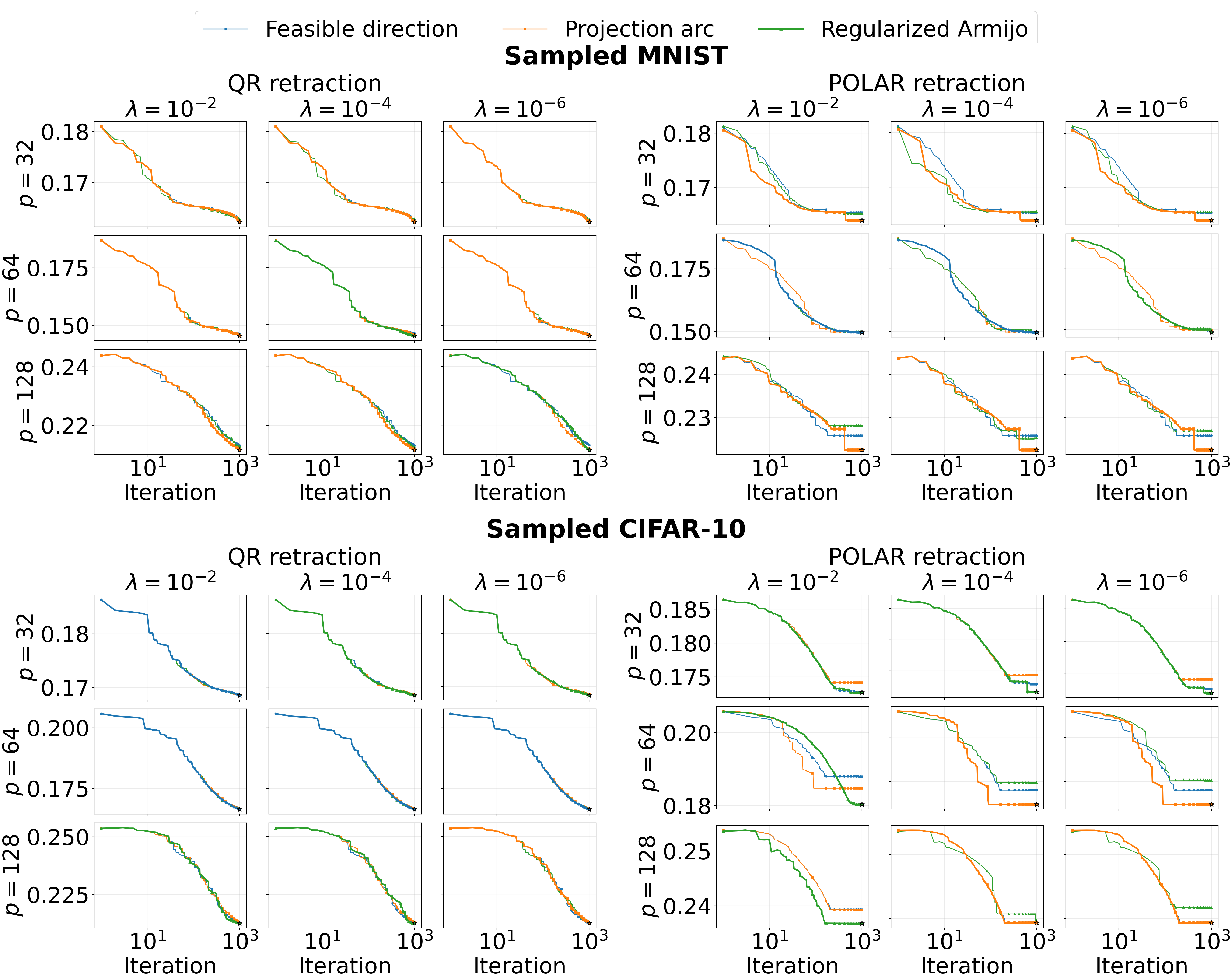}
\caption{RMSE convergence over $1000$ iterations using the optimal Armijo-rule parameters selected from the grid search summarized in Table \ref{tab:mnist-cifar-mask025-armijo-grid}.
    The left panels use the QR-decomposition-based retraction and the right panels use the polar-decomposition-based retraction.
    The $x$-axis is logarithmic. The thicker curve with a star at its final point marks the lowest final RMSE in that subplot.
    The masking rate is set to 0.25. On the sampled MNIST dataset (the upper panel),
    the ``Projection Arc'' routine appears to consistently outperform the Regularized Armijo (benchmark) routine
    when using the QR-decomposition-based retraction and the low-rank $p = 32$, or
    using the polar-decomposition-based retraction and the low-rank $p = 32$ or $p = 128$.
    On the sampled CIFAR-10 dataset (the lower panel),
    the ``Feasible Direction'' routine appears to consistently outperform the Regularized Armijo (benchmark) routine
    when we using the QR-decomposition-based retraction and the low-rank $p = 64$.}
\label{fig:mnist-cifar-mask025-armijo-rmse-curves}
\end{figure}

\FloatBarrier

\subsection{Grid Search Experiments on the Full Training Splits}
\label{sec-full-train-mnist-grid-search}
Tables \ref{tab:full-train-armijo-grid} and \ref{tab:full-train-cifar10-armijo-grid} summarize the Stage 1 grid searches
for the full MNIST and grayscale CIFAR-10 training splits, respectively.
Both use masking rate $0.5$ and the common 200-iteration grid design in Table \ref{tab:armijo-grid-search-design}.

\begin{table}[!htbp]
\centering
\small
\setlength{\tabcolsep}{1.2mm}
\begin{tabular}{@{}cc|rrr|rrr|rrr@{}}
\hline
\multicolumn{11}{c}{\textbf{Full-train MNIST}} \\
\hline
\multicolumn{11}{c}{\textbf{QR retraction}} \\
\hline
rank & $\lambda$ & \multicolumn{3}{c|}{\shortstack{Regularized Armijo\\(Benchmark)}} & \multicolumn{3}{c|}{\shortstack{Feasible Direction\\(Main Routine)}} & \multicolumn{3}{c}{\shortstack{Projection Arc\\(Main Routine)}} \\
 & & best & mean & fail. & best & mean & fail. & best & mean & fail. \\
\hline
32 & $10^{-2}$ & 0.2140 $(0.9,0.75,0.3)$ & 0.2143 & 0 & \textbf{0.2139 $(0.9,0.75,0.5)$} & 0.2143 & 0 & 0.2140 $(0.9,0.75,0.3)$ & 0.2143 & 0 \\
32 & $10^{-4}$ & 0.2139 $(0.9,0.75,0.5)$ & 0.2143 & 0 & \textbf{0.2139 $(0.9,0.75,0.5)$} & 0.2143 & 0 & 0.2140 $(0.9,0.75,0.3)$ & 0.2143 & 0 \\
32 & $10^{-6}$ & 0.2139 $(0.9,0.75,0.5)$ & 0.2143 & 0 & \textbf{0.2139 $(0.9,0.75,0.5)$} & 0.2143 & 0 & 0.2140 $(0.9,0.75,0.3)$ & 0.2143 & 0 \\
64 & $10^{-2}$ & 0.2230 $(0.9,0.75,1)$ & 0.2277 & 0 & \textbf{0.2226 $(0.9,0.75,0.5)$} & 0.2275 & 0 & 0.2232 $(0.9,0.75,0.5)$ & 0.2276 & 0 \\
64 & $10^{-4}$ & \textbf{0.2222 $(0.9,0.75,1)$} & 0.2276 & 0 & 0.2226 $(0.9,0.75,0.5)$ & 0.2275 & 0 & 0.2232 $(0.9,0.75,0.5)$ & 0.2276 & 0 \\
64 & $10^{-6}$ & \textbf{0.2224 $(0.9,0.75,1)$} & 0.2274 & 0 & 0.2226 $(0.9,0.75,0.5)$ & 0.2275 & 0 & 0.2232 $(0.9,0.75,0.5)$ & 0.2276 & 0 \\
128 & $10^{-2}$ & 0.2729 $(0.9,0.75,0.5)$ & 0.2781 & 0 & \textbf{0.2725 $(0.9,0.75,0.3)$} & 0.2779 & 0 & 0.2729 $(0.9,0.75,0.3)$ & 0.2779 & 0 \\
128 & $10^{-4}$ & \textbf{0.2725 $(0.9,0.75,0.5)$} & 0.2779 & 0 & 0.2725 $(0.9,0.75,0.3)$ & 0.2779 & 0 & 0.2729 $(0.9,0.75,0.3)$ & 0.2779 & 0 \\
128 & $10^{-6}$ & \textbf{0.2718 $(0.9,0.75,0.5)$} & 0.2779 & 0 & 0.2725 $(0.9,0.75,0.3)$ & 0.2779 & 0 & 0.2729 $(0.9,0.75,0.3)$ & 0.2779 & 0 \\
\hline
\multicolumn{11}{c}{\textbf{POLAR retraction}} \\
\hline
rank & $\lambda$ & \multicolumn{3}{c|}{\shortstack{Regularized Armijo\\(Benchmark)}} & \multicolumn{3}{c|}{\shortstack{Feasible Direction\\(Main Routine)}} & \multicolumn{3}{c}{\shortstack{Projection Arc\\(Main Routine)}} \\
 & & best & mean & fail. & best & mean & fail. & best & mean & fail. \\
\hline
32 & $10^{-2}$ & \textbf{0.2140 $(0.9,0.5,1)$} & 0.2148 & 3 & 0.2142 $(0.75,0.5,0.5)$ & 0.2148 & 3 & 0.2141 $(0.75,0.5,0.5)$ & 0.2148 & 4 \\
32 & $10^{-4}$ & 0.2142 $(0.75,0.5,0.5)$ & 0.2148 & 2 & 0.2142 $(0.75,0.5,0.5)$ & 0.2148 & 3 & \textbf{0.2141 $(0.75,0.5,0.5)$} & 0.2148 & 4 \\
32 & $10^{-6}$ & 0.2142 $(0.75,0.5,0.5)$ & 0.2148 & 3 & 0.2142 $(0.75,0.5,0.5)$ & 0.2148 & 3 & \textbf{0.2141 $(0.75,0.5,0.5)$} & 0.2148 & 4 \\
64 & $10^{-2}$ & \textbf{0.2247 $(0.75,0.5,0.5)$} & 0.2294 & 1 & 0.2255 $(0.5,0.5,1)$ & 0.2296 & 0 & 0.2255 $(0.5,0.5,0.5)$ & 0.2298 & 1 \\
64 & $10^{-4}$ & 0.2263 $(0.5,0.5,0.5)$ & 0.2304 & 1 & \textbf{0.2255 $(0.5,0.5,1)$} & 0.2296 & 0 & 0.2255 $(0.5,0.5,0.5)$ & 0.2298 & 1 \\
64 & $10^{-6}$ & \textbf{0.2255 $(0.5,0.5,0.5)$} & 0.2298 & 0 & 0.2255 $(0.5,0.5,1)$ & 0.2296 & 0 & 0.2255 $(0.5,0.5,0.5)$ & 0.2298 & 1 \\
128 & $10^{-2}$ & \textbf{0.2757 $(0.9,0.75,0.5)$} & 0.2805 & 1 & 0.2770 $(0.5,0.75,0.3)$ & 0.2809 & 1 & 0.2759 $(0.9,0.75,0.5)$ & 0.2805 & 1 \\
128 & $10^{-4}$ & 0.2774 $(0.5,0.25,0.3)$ & 0.2814 & 0 & 0.2770 $(0.5,0.75,0.3)$ & 0.2809 & 1 & \textbf{0.2759 $(0.9,0.75,0.5)$} & 0.2805 & 1 \\
128 & $10^{-6}$ & 0.2761 $(0.75,0.5,1)$ & 0.2807 & 1 & 0.2770 $(0.5,0.75,0.3)$ & 0.2809 & 1 & \textbf{0.2759 $(0.9,0.75,0.5)$} & 0.2805 & 1 \\
\hline
\end{tabular}

\caption{For each retraction-rank-regularization configuration,
the table reports each routine's best final RMSE (``best RMSE''), mean final RMSE (``mean RMSE''),
and number of failures (``failures'').
The number of failures is the count of grid points whose runs failed by reaching the maximum number of backtracking steps
(set to $200$) before producing a successful final metric.
Failed records are excluded from evaluating the other two metrics.
Parentheses give the Armijo parameters $(\beta,\sigma,s)$ for the best RMSE entry.
Boldface marks the lowest best RMSE among the three routines in each row.
The masking rate is set to 0.5.}
\label{tab:full-train-armijo-grid}
\end{table}

\clearpage

\begin{table}[!htbp]
\centering
\small
\setlength{\tabcolsep}{1.2mm}
\begin{tabular}{@{}cc|rrr|rrr|rrr@{}}
\hline
\multicolumn{11}{c}{\textbf{Full-train CIFAR-10}} \\
\hline
\multicolumn{11}{c}{\textbf{QR retraction}} \\
\hline
rank & $\lambda$ & \multicolumn{3}{c|}{\shortstack{Regularized Armijo\\(Benchmark)}} & \multicolumn{3}{c|}{\shortstack{Feasible Direction\\(Main Routine)}} & \multicolumn{3}{c}{\shortstack{Projection Arc\\(Main Routine)}} \\
 & & best & mean & fail. & best & mean & fail. & best & mean & fail. \\
\hline
32 & $10^{-2}$ & 0.2876 $(0.75,0.5,0.5)$ & 0.2891 & 4 & \textbf{0.2873 $(0.9,0.5,1)$} & 0.2889 & 1 & 0.2874 $(0.9,0.5,1)$ & 0.2890 & 2 \\
32 & $10^{-4}$ & 0.2874 $(0.9,0.5,1)$ & 0.2891 & 4 & \textbf{0.2873 $(0.9,0.5,1)$} & 0.2889 & 1 & 0.2874 $(0.9,0.5,1)$ & 0.2890 & 2 \\
32 & $10^{-6}$ & \textbf{0.2873 $(0.9,0.5,1)$} & 0.2890 & 2 & \textbf{0.2873 $(0.9,0.5,1)$} & 0.2889 & 1 & 0.2874 $(0.9,0.5,1)$ & 0.2890 & 2 \\
\hline
64 & $10^{-2}$ & 0.2968 $(0.75,0.5,0.5)$ & 0.2997 & 1 & 0.2967 $(0.75,0.5,0.5)$ & 0.2995 & 0 & \textbf{0.2966 $(0.9,0.75,1)$} & 0.2995 & 0 \\
64 & $10^{-4}$ & 0.2966 $(0.75,0.5,0.5)$ & 0.2995 & 0 & 0.2967 $(0.75,0.5,0.5)$ & 0.2995 & 0 & \textbf{0.2966 $(0.9,0.75,1)$} & 0.2995 & 0 \\
64 & $10^{-6}$ & \textbf{0.2966 $(0.75,0.5,0.5)$} & 0.2995 & 0 & 0.2967 $(0.75,0.5,0.5)$ & 0.2995 & 0 & 0.2966 $(0.9,0.75,1)$ & 0.2995 & 0 \\
\hline
128 & $10^{-2}$ & 0.3289 $(0.9,0.75,0.5)$ & 0.3305 & 0 & 0.3288 $(0.9,0.75,0.5)$ & 0.3305 & 0 & \textbf{0.3287 $(0.9,0.75,0.3)$} & 0.3305 & 0 \\
128 & $10^{-4}$ & \textbf{0.3285 $(0.9,0.75,0.5)$} & 0.3305 & 0 & 0.3288 $(0.9,0.75,0.5)$ & 0.3305 & 0 & 0.3287 $(0.9,0.75,0.3)$ & 0.3305 & 0 \\
128 & $10^{-6}$ & 0.3289 $(0.9,0.75,0.3)$ & 0.3305 & 0 & 0.3288 $(0.9,0.75,0.5)$ & 0.3305 & 0 & \textbf{0.3287 $(0.9,0.75,0.3)$} & 0.3305 & 0 \\
\hline
\multicolumn{11}{c}{\textbf{POLAR retraction}} \\
\hline
rank & $\lambda$ & \multicolumn{3}{c|}{\shortstack{Regularized Armijo\\(Benchmark)}} & \multicolumn{3}{c|}{\shortstack{Feasible Direction\\(Main Routine)}} & \multicolumn{3}{c}{\shortstack{Projection Arc\\(Main Routine)}} \\
 & & best & mean & fail. & best & mean & fail. & best & mean & fail. \\
\hline
32 & $10^{-2}$ & \textbf{0.2888 $(0.75,0.5,0.3)$} & 0.2911 & 6 & 0.2891 $(0.75,0.5,0.3)$ & 0.2909 & 6 & 0.2896 $(0.75,0.25,0.5)$ & 0.2914 & 6 \\
32 & $10^{-4}$ & 0.2896 $(0.75,0.5,0.3)$ & 0.2910 & 6 & \textbf{0.2891 $(0.75,0.5,0.3)$} & 0.2909 & 6 & 0.2896 $(0.75,0.25,0.5)$ & 0.2914 & 6 \\
32 & $10^{-6}$ & \textbf{0.2891 $(0.75,0.5,0.3)$} & 0.2909 & 6 & \textbf{0.2891 $(0.75,0.5,0.3)$} & 0.2909 & 6 & 0.2896 $(0.75,0.25,0.5)$ & 0.2914 & 6 \\
\hline
64 & $10^{-2}$ & 0.2997 $(0.75,0.5,1)$ & 0.3022 & 6 & \textbf{0.2995 $(0.75,0.5,1)$} & 0.3022 & 6 & 0.3006 $(0.75,0.5,1)$ & 0.3025 & 6 \\
64 & $10^{-4}$ & 0.2997 $(0.75,0.5,1)$ & 0.3021 & 6 & \textbf{0.2995 $(0.75,0.5,1)$} & 0.3022 & 6 & 0.3006 $(0.75,0.5,1)$ & 0.3025 & 6 \\
64 & $10^{-6}$ & 0.3001 $(0.75,0.5,1)$ & 0.3023 & 6 & \textbf{0.2995 $(0.75,0.5,1)$} & 0.3022 & 6 & 0.3006 $(0.75,0.5,1)$ & 0.3025 & 6 \\
\hline
128 & $10^{-2}$ & 0.3310 $(0.75,0.5,1)$ & 0.3324 & 6 & \textbf{0.3306 $(0.75,0.5,0.3)$} & 0.3324 & 6 & 0.3308 $(0.75,0.5,0.3)$ & 0.3325 & 6 \\
128 & $10^{-4}$ & \textbf{0.3292 $(0.75,0.5,0.3)$} & 0.3324 & 6 & 0.3306 $(0.75,0.5,0.3)$ & 0.3324 & 6 & 0.3308 $(0.75,0.5,0.3)$ & 0.3325 & 6 \\
128 & $10^{-6}$ & \textbf{0.3302 $(0.75,0.5,0.3)$} & 0.3323 & 6 & 0.3306 $(0.75,0.5,0.3)$ & 0.3324 & 6 & 0.3308 $(0.75,0.5,0.3)$ & 0.3325 & 6 \\
\hline
\end{tabular}

\caption{Stage 1 grid search on all $50000$ grayscale CIFAR-10 training images,
forming a $1024\times50000$ matrix with masking rate $0.5$.
For each retraction-rank-regularization configuration, the table reports each routine's best final RMSE,
mean final RMSE over successful grid points, and number of failures.
Each routine has $27$ grid points per configuration and runs for $200$ iterations.
Failures reached the maximum of $200$ backtracking steps and are excluded from both RMSE metrics.
Parentheses give the selected Armijo parameters $(\beta,\sigma,s)$.
Boldface marks the lowest best RMSE among the three routines using unrounded values;
selection ties are broken by final unregularized WLRA loss, with exact remaining ties marked jointly.}
\label{tab:full-train-cifar10-armijo-grid}
\end{table}

On full-training CIFAR-10, the best-performing routine depends on the configuration.
With the QR retraction, Projection Arc is the unique best routine in four of the nine configurations,
Feasible Direction in two, and Regularized Armijo in two; Feasible Direction and Regularized Armijo tie in the remaining configuration.
With the polar retraction, Feasible Direction is the unique best routine in five configurations
and Regularized Armijo in three, with the same two routines tied in the remaining configuration.
Both ties occur at $p=32$ and $\lambda=10^{-6}$ and persist after the unregularized-loss tie-breaker.
Of the $1458$ grid runs, $1276$ succeed and $182$ fail at the backtracking limit.
The QR failure counts are $11$, $3$, and $6$ for Regularized Armijo, Feasible Direction, and Projection Arc, respectively;
each polar routine has $54$ failures.
Thus, the RMSE comparisons are conditional on successful runs and should be read alongside the failure counts.
These Stage 1 results describe performance over $200$ iterations, not long-run superiority.

\clearpage

\subsection{Paired, One-Sided $t$-Tests Across the Experiment Results}

\textbf{Grid Search Experiments.}
To assess the aggregate differences across all the grid search experiments in this paper,
Table \ref{tab:sampled-mask050-paired-ttests} reports paired, one-sided $t$-tests
over the 144 matched retraction-rank-regularization configurations in
Tables \ref{tab:mnist-cifar-armijo-grid}, \ref{tab:mnist-cifar-mask075-armijo-grid},
\ref{tab:mnist-cifar-mask025-armijo-grid}, \ref{tab:full-train-armijo-grid}, and \ref{tab:full-train-cifar10-armijo-grid} on ``best RMSE'' and ``mean RMSE'', the two evaluation metrics common across all the tables.
Tables \ref{tab:mnist-cifar-armijo-grid}, \ref{tab:mnist-cifar-mask075-armijo-grid},
and \ref{tab:mnist-cifar-mask025-armijo-grid} each contribute 36 configurations.
Tables \ref{tab:full-train-armijo-grid} and \ref{tab:full-train-cifar10-armijo-grid} each contribute 18 configurations.
For each main routine and evaluation metric, we define the paired difference as the metric value for the main routine
minus the metric value for benchmark routine.
We test the alternative hypothesis that the mean paired difference is negative.

\begin{table}[!htbp]
\centering
\small
\begin{tabular}{llrrrc}
\hline
Comparison & Metric & mean difference & $t_{143}$ & one-sided $p$ & Outperforms? \\
\hline
\shortstack{Feasible Direction (Main Routine) vs\\Regularized Armijo (Benchmark)} & best RMSE & $-2.617\times 10^{-4}$ & -2.5414 & $6.05\times 10^{-3}$ & \textbf{Yes} \\
 & mean RMSE & $-4.821\times 10^{-5}$ & -2.4839 & $7.07\times 10^{-3}$ & \textbf{Yes} \\
\hline
\shortstack{Projection Arc (Main Routine) vs\\Regularized Armijo (Benchmark)} & best RMSE & $-2.472\times 10^{-4}$ & -2.3232 & 0.011 & \textbf{Yes} \\
 & mean RMSE & $-4.902\times 10^{-5}$ & -2.0029 & 0.024 & \textbf{Yes} \\
\hline
\end{tabular}

\caption{Paired, one-sided $t$-tests comparing Feasible Direction (main routine) and Projection Arc (main routine) with Regularized Armijo (benchmark)
across the 144 matched configurations in Tables \ref{tab:mnist-cifar-armijo-grid}, \ref{tab:mnist-cifar-mask075-armijo-grid}, \ref{tab:mnist-cifar-mask025-armijo-grid}, \ref{tab:full-train-armijo-grid}, and \ref{tab:full-train-cifar10-armijo-grid}.
The mean difference is the named main routine minus Regularized Armijo (benchmark),
so a negative value favors the main routine because lower RMSE values are preferred.
The alternative hypothesis is that the mean paired difference is negative.
``Outperforms'' is reported as Yes when the unadjusted one-sided $p$-value is below $0.05$;
no multiple-testing correction is applied.}
\label{tab:sampled-mask050-paired-ttests}
\end{table}

At the unadjusted $0.05$ significance level,
the tests support lower best RMSE and lower mean RMSE for both main routines than for the benchmark routine.
For the ``Feasible Direction'' routine, the mean differences are $-2.617\times 10^{-4}$ for best RMSE
($t_{143}=-2.5414$, $p=0.0061$)
and $-4.821\times 10^{-5}$ for mean RMSE
($t_{143}=-2.4839$, $p=0.0071$).
For the ``Projection Arc'' routine, the mean differences are $-2.472\times 10^{-4}$ for best RMSE
($t_{143}=-2.3232$, $p=0.0108$)
and $-4.902\times 10^{-5}$ for mean RMSE
($t_{143}=-2.0029$, $p=0.0235$).
These tests provide aggregate statistical evidence across the pooled configurations.
They do not by themselves establish the practical importance of the observed RMSE differences.

\textbf{Routine Comparison Experiments.}
Table \ref{tab:all-long-run-final-rmse-paired-ttests} applies the same paired-test design
to the final RMSE values from the sampled-dataset 1000-iteration experiments in
Figures \ref{fig:mnist-cifar-armijo-rmse-curves},
\ref{fig:mnist-cifar-mask075-armijo-rmse-curves}, and
\ref{fig:mnist-cifar-mask025-armijo-rmse-curves}.
Each figure contributes 36 matched configurations,
giving 108 pairs for each main-routine comparison.

\begin{table}[!htbp]
\centering
\small
\begin{tabular}{llrrrc}
\hline
Comparison & Metric & mean difference & $t_{107}$ & one-sided $p$ & Outperforms? \\
\hline
\shortstack{Feasible Direction (Main Routine) vs\\Regularized Armijo (Benchmark)} & final RMSE & $-2.463\times 10^{-4}$ & -0.9209 & 0.180 & No \\
\hline
\shortstack{Projection Arc (Main Routine) vs\\Regularized Armijo (Benchmark)} & final RMSE & $-2.206\times 10^{-4}$ & -0.8556 & 0.197 & No \\
\hline
\end{tabular}

\caption{Paired, one-sided $t$-tests comparing Feasible Direction (main routine) and Projection Arc (main routine) with Regularized Armijo (benchmark)
across the 108 matched configurations in
Figures \ref{fig:mnist-cifar-armijo-rmse-curves},
\ref{fig:mnist-cifar-mask075-armijo-rmse-curves}, and
\ref{fig:mnist-cifar-mask025-armijo-rmse-curves}.
The mean difference is the final RMSE for the named main routine minus that for Regularized Armijo (benchmark),
so a negative value favors the main routine because lower RMSE values are preferred.
The alternative hypothesis is that the mean paired difference is negative.
``Outperforms'' is reported as Yes when the unadjusted one-sided $p$-value is below $0.05$;
no multiple-testing correction is applied.}
\label{tab:all-long-run-final-rmse-paired-ttests}
\end{table}

At the unadjusted $0.05$ significance level,
the tests do not support lower final RMSE for either main routine than for the benchmark routine.
For the ``Feasible Direction'' routine, the mean difference is $-2.463\times 10^{-4}$
($t_{107}=-0.9209$, $p=0.1796$).
For the ``Projection Arc'' routine, the mean difference is $-2.206\times 10^{-4}$
($t_{107}=-0.8556$, $p=0.1971$).
Both mean differences are negative, but both $p$-values exceed the significance threshold.
Thus, the pooled long-run analysis does not provide statistically significant evidence
that either main routine has lower final RMSE than the benchmark routine.
This aggregate result does not rule out the configuration-specific advantages shown in
Figures \ref{fig:mnist-cifar-armijo-rmse-curves},
\ref{fig:mnist-cifar-mask075-armijo-rmse-curves}, and
\ref{fig:mnist-cifar-mask025-armijo-rmse-curves}.

\subsection{Our Choice of RMSE as the Primary Evaluation Metric}

The root mean square error (RMSE) is widely used to evaluate accuracy in matrix-completion tasks
\cite{mazumder-hastie-tibshirani-2010-spectral-regularization,
zhou-tao-2013-greedy-bilateral-sketch,
hsieh-olsen-2014-active-subspace-selection,
ganti-balzano-willett-2015-monotonic-single-index,
hastie-mazumder-lee-zadeh-2015-fast-als,
bhaskar-2016-probabilistic-quantized-matrix-completion,
wang-zhang-gu-2017-variance-reduction-matrix-recovery,
liu-mao-wong-2020-median-matrix-completion,
elmahdy-et-al-2020-hierarchical-graph-matrix-completion,
wang-et-al-2021-model-free-weighting,
salakhutdinov-mnih-2007-probabilistic-matrix-factorization,
salakhutdinov-mnih-2008-bayesian-pmf}.
We choose RMSE as our primary evaluation metric for two additional reasons.

First, Proposition \ref{prop-rgpa-convergence} states that every limit point of the sequence generated by
Algorithm \ref{algorithm-rgpa} is stationary when using Armijo Stepsize Rules
\ref{ur-Armijo-feasible} or \ref{ur-Armijo-projection-arc}.
It does not provide an asymptotic convergence-rate guarantee.
Accordingly, our experiments compare completion accuracy through RMSE rather than infer empirical convergence rates from objective values.

Moreover, Problems \ref{prob-wlra-reg-mfd} and \ref{prob-wlra-constrained} have different objective functions.
Their objective values are therefore not directly comparable across routines.
In particular, the benchmark objective (the objective function of Problem \ref{prob-wlra-reg-mfd}) includes the regularization term $\lambda\|\mathbf{x}\|^2$, whereas the main-routine objective (the objective function of Problem \ref{prob-wlra-constrained}) does not.
Thus, the benchmark objective can differ from the main-routine objective at initialization, and its value varies with the regularization parameter $\lambda$.
RMSE applies the same definition to all three routines and allows a direct comparison on a common scale.
\FloatBarrier


\section{Reproducibility Details}
\label{sec-reproducibility-details}

\subsection{Datasets and Preprocessing}
The experiments use the training splits of MNIST and CIFAR-10.
For the sampled experiments, the data artifacts contain 6000 balanced MNIST images and 5000 balanced CIFAR-10 images.
The full-training Stage 1 experiments use all 60000 MNIST training images and all 50000 CIFAR-10 training images.
Each image is flattened and stored as one column of a matrix.
The resulting sampled matrices have sizes $784\times 6000$ for MNIST and $1024\times 5000$ for grayscale CIFAR-10.
The full-training matrices have sizes $784\times60000$ for MNIST and $1024\times50000$ for grayscale CIFAR-10.
Each serialized data artifact is a dictionary containing the complete matrix \texttt{M\_full}, the masked matrix \texttt{M\_masked}, and the binary weight matrix \texttt{W}.

We use masking rates $0.25$, $0.5$, and $0.75$ for the sampled datasets and masking rate $0.5$ for both full-training datasets.
The weight is zero at a masked entry and one otherwise.
The masked matrix is obtained by multiplying the complete matrix by the weight matrix.
We use random seed $42$ for dataset sampling and masking.
All routines for a fixed dataset, retraction, rank, and regularization value receive clones of the same truncated-singular-value-decomposition initialization.

\subsection{Grid Searches and Long Runs}
The grid searches use ranks $p\in\{32,64,128\}$ and regularization values
$\lambda\in\{10^{-2},10^{-4},10^{-6}\}$ with the QR and polar retractions.
For each routine and retraction-$p$-$\lambda$ configuration, the Armijo grid is
\[
(\beta,\sigma,s)\in
\{0.5,0.75,0.9\}\times
\{0.25,0.5,0.75\}\times
\{0.3,0.5,1.0\}.
\]
Thus, each routine has 27 Armijo grid points per configuration.
Every grid-search run uses 200 iterations, maximum backtracking count 200, Armijo tolerance $10^{-12}$, CPU execution, and \texttt{float32} arithmetic.
For the main routines, the constraint radius is computed as
$\|A\odot W\|_F/\sqrt{\lambda}$.

For each routine and configuration, successful grid points are ranked by final RMSE and then by final unregularized WLRA loss.
For the sampled-dataset routine comparisons reported here, the selected Armijo parameters seed the 1000-iteration convergence runs.
The full-training tables report only Stage 1 grid-search results.
If a selected long run fails, the runner records the failed attempt and tries the next successful source-grid candidate.
The raw experiment records retain failed runs.
The table generator excludes failed records from best and mean RMSE and reports their number separately.

\subsection{Metrics and Statistical Tests}
\label{sec-armijo-evaluation-metric}
RMSE is evaluated only on masked entries.
For a reconstructed matrix $\widehat A$, it is
\[
\operatorname{RMSE}
=
\left(
\frac{\sum_{i,j}(1-w_{i,j})(a_{i,j}-\widehat a_{i,j})^2}
{\sum_{i,j}(1-w_{i,j})}
\right)^{1/2}.
\]
The grid-search tables report the lowest final RMSE, the mean final RMSE over successful grid points, and the number of failed grid points.
Tables based on the sampled datasets also report the mean number of Armijo inequality evaluations.
For a fixed routine and configuration, let $\mathcal S$ be the set of successful grid runs and
let $E_{g,k}$ count the Armijo inequality checks performed at iteration $k$ of run $g$.
The reported mean is
\[
\overline E
=\frac{1}{|\mathcal S|}\sum_{g\in\mathcal S}
\left(\frac{1}{200}\sum_{k=0}^{199}E_{g,k}\right).
\]
Thus, the statistic averages checks per iteration within each successful run and then across successful grid runs.
The count includes rejected checks and the accepted check when performed; a zero-step exit adds no inequality check.
Failed runs are excluded from this average and their number is reported separately.
This statistic is neither the total number of checks in a $200$-iteration run nor a wall-clock cost.

The grid-search paired tests pool the 144 retraction-$p$-$\lambda$ configurations reported in Tables
\ref{tab:mnist-cifar-armijo-grid}, \ref{tab:mnist-cifar-mask075-armijo-grid},
\ref{tab:mnist-cifar-mask025-armijo-grid}, \ref{tab:full-train-armijo-grid}, and \ref{tab:full-train-cifar10-armijo-grid}.
The long-run paired tests pool the final RMSE from the 108 matched sampled-dataset configurations reported in Figures
\ref{fig:mnist-cifar-armijo-rmse-curves},
\ref{fig:mnist-cifar-mask075-armijo-rmse-curves}, and
\ref{fig:mnist-cifar-mask025-armijo-rmse-curves}.
For each main routine, the paired difference is the main-routine metric minus the Regularized Armijo metric.
The tests use the one-sided alternative that the mean difference is negative and the unadjusted significance threshold $0.05$.
The implementation uses \texttt{scipy.stats.ttest\_rel} with \texttt{alternative="less"}.

\subsection{Hardware and Software}
The reported experiments were executed on an Apple M5 Pro CPU with 18 cores (6 super and 12 performance cores) and 48 GB of memory.
The recorded software environment uses Python 3.10.9, PyTorch 2.12.0, NumPy 1.23.5, SciPy 1.10.1, and Matplotlib 3.7.2 on macOS.
The exhaustive shards set
\texttt{OMP\_NUM\_THREADS=2},
\texttt{MKL\_NUM\_THREADS=2},
\texttt{VECLIB\_MAXIMUM\_THREADS=2}, and
\texttt{NUMEXPR\_NUM\_THREADS=2}.

\subsection{Commands and Artifact Provenance}
The QR and polar grid searches are launched from the repository root with
\path{experiments/runners/armijo/grid_search_qr.py} and
\path{experiments/runners/armijo/grid_search_polar.py}, respectively.
The command-line arguments specify the dataset artifact, routine, parameter grid, ranks, regularization values, iteration count, automatic radius, RMSE mode, device, dtype, and Armijo tolerance.
The selected 1000-iteration runs use
\path{experiments/runners/armijo/run_selected_history_grid.py}.
The table summaries use
\path{experiments/summaries/armijo/summarize_armijo_convergence_methods.py},
the convergence figures use
\path{experiments/plotting/armijo/curves/plot_convergence_grid.py},
and Tables \ref{tab:sampled-mask050-paired-ttests} and
\ref{tab:all-long-run-final-rmse-paired-ttests} use
\path{experiments/summaries/armijo/paired_t_tests.py}.

The full-training CIFAR-10 table data use \path{experiments/summaries/armijo/summarize_full_train_stage1.py}; the expanded grid-search tests use the paired-test program with all eight dataset prefixes specified explicitly.
The public release includes the derived table data and paired-test results.
The grid-search loader selects the six Stage 1 shards explicitly to exclude the separately recorded 1000-iteration histories.

The public code release contains the implementation, tests, experiment runners, summary and plotting code, processed table data, and reference figures.
It excludes the private research repository's Git history, local logs, private manuscript evidence, and large serialized dataset tensors.
The public MNIST and CIFAR-10 training data must therefore be obtained separately and converted to the three-tensor format described above before rerunning the experiments.
The included processed tables and figures are reference outputs rather than substitutes for an independent rerun.
To reproduce them, users must generate the dataset artifacts, rerun the grid-search and selected-history experiments, and then apply the included summary, paired-test, and plotting programs to the newly generated raw records.

\end{document}